\documentclass{amsart}

\usepackage{amssymb}
\usepackage{xypic}
\usepackage{url}
\usepackage[dvipsnames]{color}
\usepackage{epsfig}
\usepackage{etoolbox}
\usepackage{mathtools}
\usepackage{hyperref}

\newtheorem{theorem}[equation]{Theorem}
\newtheorem{lemma}[equation]{Lemma}
\newtheorem{proposition}[equation]{Proposition}
\newtheorem{prop}[equation]{Proposition} 
\newtheorem{corollary}[equation]{Corollary}
\theoremstyle{definition}
\newtheorem{algorithm}[equation]{Algorithm}
\newtheorem{assumption}[equation]{Assumption}
\newtheorem{question}[equation]{Question}
\newtheorem{definition}[equation]{Definition}
\newtheorem{example}[equation]{Example}
\newtheorem{method}[equation]{Method}
\newtheorem{notation}[equation]{Notation}
\newtheorem{remark}[equation]{Remark}
\theoremstyle{plain}
\numberwithin{equation}{section} 
\preto{\table}{\stepcounter{equation}}

\preto{\figure}{\stepcounter{equation}}

\begin{document}

\let\myS\S 

\def\({\left(}
\def\){\right)}

\def\iso{\simeq}
\def\rightiso{\buildrel{\simeq}\over{\rightarrow}}

\def\ol{\overline}

\def\a{\alpha}
\def\b{\beta}
\def\cc{\gamma}
\def\d{\delta}
\def\e{\varepsilon}
\def\s{\sigma}
\def\t{\tau}
\def\o{\omega}
\def\G{\Gamma}
\def\l{\lambda}

\def\A{\mathbb{A}}
\def\F{\mathbb{F}}
\def\H{\mathbb{H}}
\def\M{\mathbb{M}}
\def\P{\mathbb{P}}
\def\Q{\mathbb{Q}}
\def\R{\mathbb{R}}
\def\Z{\mathbb{Z}}

\def\Cp{{\mathbb{C}_p}}
\def\Qp{\Q_p}
\def\Zp{\Z_p}

\def\Abar{\overline{A}}
\def\Adag{A^\dagger}
\def\Ahat{\widehat A}
\def\phi{\varphi}
\def\phit{\phi_t^\s}
\def\phibar{\overline{\phi}} 
\def\myphicr #1 {\phi_{\cry, #1}^*}
\def\phihat{\phi_\s\hbox{\hskip-6pt\textasciicircum}} 
\def\phihatt{\phi_\t\hbox{\hskip-6pt\textasciicircum}} 
\def\phihatA{\phi_{\s,\hat A}} 
\def\phidag{\phi_\s^\dagger} 
\def\phidagA{\phi_{\s,A^\dagger}} 
\def\phidagAt{\phi_{\t,A^\dagger}} 

\def\D{\mathcal{D}}
\def\ee{\mathcal{E}}
\def\I{\mathcal{I}}
\def\O{\mathcal{O}}
\def\OCq{\mathcal{O}_{C,q}}
\def\OCgen{\mathcal{O}_{C,\eta}}

\def\bK{\bf K}
\def\bL{\bf L}

\def\tK{\widetilde K}
\def\tL{\widetilde L}
\def\tR{\widetilde R}
\def\ta{\tilde \alpha}
\def\tb{\tilde \beta}
\def\tM{\widetilde M}

\def\an{\textup{an}}
\def\coker{\textup{coker}}
\def\cry{\textup{cr}}
\def\dd{\textup{d}}
\def\id{\textup{id}}
\def\im{\textup{im}}
\def\pr{\textup{pr}}
\def\Frac{\textup{Frac}}
\def\Id{\textup{Id}}
\def\rig{\textup{rig}}
\def\Jac{\textup{Jac}}
\def\MW{\textup{MW}}
\def\ord{\textup{ord}}
\def\reduction{\textup{red}}
\def\rk{\textup{rk}}
\def\Spec{\textup{Spec}}
\def\Tr{\textup{Tr}} 
\def\tr{\textup{tr}} 

\def\maxp{\operatornamewithlimits{\textup{max}^+}}
\def\sump #1 #2 {\operatornamewithlimits{}{{\sum}}^{{\smash{\prime}}#2}_{#1}}

\def\la{\langle}
\def\ra{\rangle}

\def\tint{{\textstyle\int}} 

\def\Rx{R\la \x \ra}
\def\Rtildex{\tR\la \x \ra}
\def\Rxdag{R\la \x \ra ^\dagger}
\def\Rtab #1 #2 {R_{#1,#2}((t))}
\def\RtabN #1 #2 #3 {R_{#1, #2; #3}((t))}
\def\Rtstar{R_*((t))}

\def\Rt{R((t))}
\def\Rthat{\widehat{R((t))}}
\def\Stab #1 #2 #3 {S_{#1,#2}^{#3}((t))}

\def\RN #1 #2 {#1_{#2}}

\def\VV#1{V(#1)}
\def\W#1{W(#1)}

\def\ex{\xi}

\def\Aut{\textup{Aut}}
\def\dr{{\textup{dR}}}
\def\Gal{\textup{Gal}}
\def\hdr{H_\dr}
\def\hdre{H_{\dr,a}}
\def\hdree{H_{\dr,\textup{alg}}}
\def\hrig{H_\rig}
\def\hcr{H_\cry}
\def\hmw{H_{\MW}}
\def\res{\operatorname{Res}}

\def\pair#1{\langle #1 \rangle}
\def\myproj{\pi}
\def\mynewmap{\lambda}

\def\pa{p'}

\def\fpb{{\overline{\F}_p}}

\def\realpos{\R_{>0}}
\def\x{\mathbf{x}}
\def\S{\mathbf{S}}
\def\sol{\mathbf{s}} 
\def\soll{\tilde\sol}
\def\ts{\tilde{s}}   

\def\tf{\tilde f}
\def\th{\tilde h}
\def\tu{\tilde u}
\def\tx{\tilde x}
\def\ty{\tilde y}
\def\tN{\tilde N}


\def\Ia #1 {I_{#1}}
\def\tom{\widetilde \omega}
\def\ot{\widetilde{O}}

\title{Frobenius lifts and point counting for smooth curves}

\author{Amnon Besser}
\address{Department of Mathematics\\Ben-Gurion University of the Negev\\P.O.B. 653\\Be'er-Sheva 84105\\Israel}

\author{Rob de Jeu}
\address{Faculteit der B\`etawetenschappen\\Afdeling Wiskunde\\Vrije Universiteit Amsterdam\\De Boelelaan~1111\\1081 HV Amsterdam\\The Netherlands}

\author{Pengju Guan}
\address{Faculteit der B\`etawetenschappen\\Afdeling Wiskunde\\Vrije Universiteit Amsterdam\\De Boelelaan~1111\\1081 HV Amsterdam\\The Netherlands}

\author{Muxi Li}
\address{Department of Mathematics\\School for Mathematics Science\\Tiangong University\\Binshuixi Road 399\\300190 Tianjin\\China}

\thanks{\emph{Funding acknowledgement.}
This publication is part of the project `Computing syntomic regulators for curves and surfaces' with file number 613.009.139 of the research programme `Mathematics
Clusters, PhD Positions', which is financed by the Dutch Research Council (NWO).
Rob de Jeu would like to thank the Isaac Newton Institute for Mathematical Sciences, Cambridge, for support and hospitality during the programme
`$K$-theory, algebraic cycles and motivic homotopy theory'
(KAH and KAH2)
where work on this paper was undertaken. This work was supported by EPSRC grant no EP/R0146051.
This author is also partially supported by the COGENT project which has
received funding from the European Union’s Horizon Europe Programme
under the Marie Sklodowska-Curie actions HORIZON-MSCA-2023-DN-01 call
(Grant agreement ID: 101169527), and from UK Research and Innovation.
}

\begin{abstract}
We describe an algorithm for computing the zeta function of a
proper, smooth curve over a finite field~$ k $ of characteristic~$ p $,
when the curve is given together with some auxiliary data, including a lift~$ C $ to
the valuation ring in a finite extension of~$ \Q_p $.
The algorithm is denominator-free if the ramification is at most~$ p $.
Our method computes
the matrix of the action of a semilinear Frobenius on the first
de Rham cohomology group of the curve by means of Poincar\'e duality,
using cup products that can be computed from
local expansions of a globally defined lift of Frobenius.
Its complexity is softly cubic in the field degree
for (general) smooth planar curves, for which we work out our
general estimates in more detail.

We make explicit how to compute a suitable basis of the first de
Rham cohomology group of~$ C $, based on 1-forms with `locally integrable polar parts',
in both the general case and when the curve is smooth planar.

We show the crystalline Frobenius preserves the first de Rham
cohomology group of~$ C $ if the ramification is at most~$ p $,
improving upon known results.

In an appendix we prove a well-known formula for the cup product,
and a formula by Serre, 
on the first de Rham cohomology for a curve in characteristic
zero, for which no reference seems to exist.
\end{abstract}



\maketitle

\tableofcontents

\section{Introduction} \label{sec:intro}

Let $ p $ be a prime number and let  $ k $ be  a finite field of
characteristic~$p$. An important problem of algorithmic 
number theory is to count the number of points of a smooth (and
usually proper) variety $Y$ defined over $k$. By point
counting we mean, 
more precisely, the computation of the matrix of a suitable
Frobenius map, acting on some \'etale or crystalline cohomology
group of $Y$.
It is well-known that obtaining such a matrix to a sufficiently
high precision allows an exact determination of its characteristic
polynomial as its coefficients satisfy the Weil bounds (see
Section~\ref{algorithm-estimates}).

The modern theory of point counting begins with the paper of
Schoof~\cite{Schoof85} for counting points on an elliptic curve $E$ by effectively
computing the action of Frobenius on the first \'etale cohomology group
of $E$. This direction of using \'etale cohomology was pursued by various
other authors, and still provides the best method when the field is prime
($q=p$) or close to being prime.

Other point counting methods, beginning with the work of Satoh~\cite{Sat00},
use crystalline cohomology. To describe these methods, let us fix some
more notation.

\begin{notation} \label{basicnot}
  Let $K$ be a finite extension of the field $\Qp$ of $p$-adic
  numbers, with ramification index $ e $, valuation ring $ R$,
  uniformiser $\pi$, and 
  residue field $R/\pi R$ isomorphic to $k$. We normalise the
  valuation $ v $ on $ K $ by $v(p)=1$. We fix an automorphism~$\s$
  of $K$, and denote by $\ol{\s} $ the induced map on $ k $,
  which is given by~$x\mapsto x^{\pa}$ with $\pa$ a power of $p$.
\end{notation}

If~$ K/\Qp $ is Galois then~$ \s $ in~$ \Aut(K/\Qp) $ with~$ \ol{\s} $
the~$ \pa $-th power map on~$ k $ always exists. If~$ e = 1 $
then~$ K/\Qp $ is Galois and~$\s$ is unique.
For the point counting method below, we take~$ \pa' $ equal to the order of a subfield
of~$ k $ (and most naturally~$ p = \pa $), but a large part of what we want to describe
works without this assumption. Our generality, including allowing
for ramification, will be useful in future work
regarding syntomic regulators and Coleman integration.

In counting methods based on crystalline cohomology, one computes an
effective representation for the crystalline cohomology group
$\hcr^i(Y/W(k)) \otimes K$. This has a $\s$-semilinear endomorphism
$\myphicr {\pa} \otimes \s $, and the sought-after $ k $-linear Frobenius is then obtained as a
power of this (see Method~\ref{zeta-method} below).

When $Y$ can be lifted to characteristic $0$, and if~$e\le p-1$
(which always holds if  $e=1$), then crystalline cohomology
is the de Rham cohomology of the lift by \cite[Theorem~2.8]{Ber-Ogu83}.  If the Frobenius
endomorphism can be lifted as well, then the endomorphism $\myphicr {\pa} \otimes \s $ corresponds
to the action of the lift on this de Rham cohomology. This is
the case in, for example, Satoh's algorithm~\cite{Sat00}.

Finding a lift of Frobenius for a proper variety is rarely possible. An
alternative is to only lift Frobenius on an affine open part. As
crystalline cohomology is infinite dimensional in this case, one has
to use a more refined cohomology theory, the Monsky-Washnitzer
cohomology~\cite{Mon-Was68}, which is a special case of Berthelot's
rigid cohomology~\cite{Ber97}. This cohomology theory associates to
the affine variety $\Spec(\Abar)$ de Rham cohomology groups of
$\Adag_K = \Adag \otimes_R K$, where $\Adag$ is a ``weakly complete''~$R$-algebra whose
reduction modulo $\pi$ is $\Abar$. The action of $\myphicr {\pa} \otimes \s $
corresponds to the action on these Monsky-Washnitzer cohomology
groups of a $\s$-semilinear endomorphism of $\Adag$
that reduces to the $\pa$-th power map on~$ \Abar $.

The use of Monsky-Washnitzer cohomology
in point counting algorithms was pioneered in the seminal paper
of Kedlaya~\cite{Ked01} on counting points on hyperelliptic
curves. Kedlaya's ideas can be extended to more general
curves~\cite{Ger03,Den-Ver06,CDV06, tuit16} (see also the overview~\cite{Cha08}).

Kedlaya type algorithms generally consist of two main components.
\begin{enumerate}
\item An explicit lift of Frobenius to an endomorphism of $\Adag$,
 usually given in a straightforward manner.

\item A reduction algorithm that uses a basis for
  $\hdr^i(\Adag_K)$ and shows how to explicitly write any $i$-form as
  a linear combination of basis elements plus an exact differential.
\end{enumerate}
Extending each of these steps from hyperelliptic curves to more
general curves in an efficient way proved to be a non-trivial task.

Note that by Corollaire~7.4 of \cite[Expos\'e~III]{SGA1}, any proper, smooth curve
over~$ k $ lifts to a smooth, proper curve over the Witt vectors~$ W(k) $
of~$ k $, necessarily
with geometrically irreducible generic fibre if the curve over $ k $ is geometrically
irreducible.
We shall therefore describe a point counting algorithm for a curve in the following rather general situation.
This also fixes notation for in the rest of the paper.

We let $f:C \to \Spec(R)$ be a proper, smooth curve 
over our valuation ring~$ R $ with geometrically irreducible fibres.
\begin{notation} \label{globalnotation}
(1)
We let~$ g $ be the genus of the generic fibre $C_K$, which equals
that of the special fibre $C_k$.

(2)
We assume given a Zariski open affine $ X = \Spec(A) $ in $ C $ that
contains the generic point of $ C_k $, 
with~$ A = R[x_1,\dots,x_n]/(f_2,\dots,f_n) $
for given~$ f_2,\dots,f_n $.
Note that the reduction~$ \ol{X} = \Spec(\Abar) $ with
$ \Abar = A \otimes_R k = k[x_1,\dots,x_n]/(\ol{f_2},\dots,\ol{f_n})$
is a smooth complete intersection.
(See Assumption~\ref{Aassumption} for the terminology and
Remark~\ref{smoothlcirem} for the existence of such an~$ X $.)

(3)
We assume given $\o_1,\ldots, \o_{2g}$ in~$ \Omega_{A/R}^1 $
that form an~$ R $-basis for the cohomology group $\hdr^1(C_R/R)$.%
\footnote{%
We refer to Remark~\ref{ComputewithKbasis} if one wants to, instead, use such
elements in~$ \Omega_{A_K/K}^1 $ that form a~$ K $-basis of~$ \hdr^1(C_K/K) $;
or to Section~\ref{de-rham}
(in particular Proposition~\ref{newOCDprop}, as well as
Algorithms~\ref{KtoRbasis} and~\ref{KtoRbasisOmega})
if one wants to effectively
compute the data for~$ A $ from that of~$ A_K $.}%
\end{notation}

To perform point counting on these curves, we introduce three new
techniques.

The first is a general, explicit procedure, inspired by Section~2 of~\cite{Arabia01},
for lifting the~$ \pa $-th power Frobenius for smooth, complete intersections.
This method introduces variables corresponding to the defining equations,
and uses these to find a correction to the naive approximate lift given by raising
to the power $ \pa $ and applying~$ \s $ to the coefficients. Using
this correction gives a~$ \s $-semilinear map~$ \phidag $ on~$ R\la x_1,\dots,x_n \ra^\dagger $
that induces a map~$\phidagA$ on~$ \Adag $, which itself reduces to the desired~$ \pa $-th
power map on~$ \Abar $.
We make this procedure explicit and provide 
estimates on the coefficients in~$ \phidag $.

The second technique is only for curves, and avoids a 
(generally computationally expensive)
reduction algorithm by
replacing it with computations of cup products by means of residues.
One observes that in order to know the matrix of $\myphicr {\pa} \otimes \s $
above,
it suffices to compute all cup products $ [\o_i] \cup [\o_j] $
and~$ [\o_i] \cup [ (\myphicr {\pa} \otimes \s) ( \o_j)] $.
Our techniques reduce the computation of these cup products to a
computation of the residues of the forms~$ (\int \o_i) \o_j $ and~$ (\int \o_i) \phidagA (\o_j) $
on certain annuli, called ends, which are 
``at the boundary'' of the rigid space associated with~$\Adag$.

Finally, using that we only need to calculate explicitly with $ \phidagA $ at
the ends, we compute its action directly at the ends using expansions
in (one) local parameter, as opposed to first computing it globally
as a power series in several variables and then substituting local
expansions at the ends.

Overall, the resulting algorithm for point counting is asymptotically
softly cubic in the field degree. This is the same complexity as
Kedlaya's algorithm~\cite{Ked01}, which is restricted to the case of
hyperelliptic curves, and the algorithm of Castryck, Denef and
Vercauteren for non-degenerate curves~\cite{CDV06}.
The dependence on
the genus is somewhat worse for general curves but reduces in specific
situations. Finally, the dependence on $p$ is essentially linear. We
have not attempted an improvement in this direction in the style of~\cite{Har07}. Our overall complexity estimate in Proposition~\ref{complprop} is $\ot(p g^5 l^3)$, compared with $\ot(p g^6 l^3) $ obtained in~\cite[p.962]{tuit16}

The paper is organised as follows.
In Section~\ref{sec:cup-res} we explain how to obtain the matrix
of $ \myphicr {\pa} \otimes \s $
from cup products on rigid analytic spaces, and in Section~\ref{algorithm}
we describe this method algorithmically, given suitable input.
In Section~\ref{global-frobenius} we discuss how to obtain the
required lift~$ \phidagA $ on~$ \Adag $. Although we shall need it
only in the case of curves, we present the result and the estimates
on the coefficients involved for more general $ R $-algebras $ A $ 
that reduce to smooth complete intersections over~$ k $.
Section~\ref{global-examples} makes the resulting maps and estimates
more explicit in the situation of an affine plane curve or a localisation
of such a curve.  It also briefly discusses how to recover Kedlaya's
approach to hyperelliptic curves in \cite{Ked01} from our work.
Sections~\ref{expansions} and~\ref{local-frobenius} discuss how to obtain expansions
of the action of our global lift~$ \phidag $ of Section~\ref{global-frobenius}
at the ends, thus avoiding the explicit computation (in several variables) of $ \phidag $
or the induced~$ \phidagA $ on~$ \Adag $.
Section~\ref{local-examples} returns to 
some of the examples discussed in Section~\ref{global-examples},
considering them from the points of view of leaving out only
one point, or obtaining a simpler lift by localising more.
We obtain estimates for computing cup products using expansions
in Section~\ref{cup-product-estimates}.
In Section~\ref{eisp} we combine this with a condition on (expanded)
1-forms to have `integrable polar part' to bound from below,
at each end, the valuation that can occur when computing the
local contribution to a cup product, after applying our lift
of Frobenius to one of the 1-forms. As new, purely theoretical result, we find that
the lift of the crystalline Frobenius maps the first de Rham cohomology group
to itself also when the ramification is equal to the residue
characteristic (see Corollary~\ref{stable-cor}).
In Section~\ref{algorithm-estimates} we turn our earlier estimates
into a finite precision version, with finite expansions, of the algorithm in Section~\ref{algorithm} that
computes the zeta function of $ C_k $.
Sections~\ref{de-rham} discusses how to compute a basis of~$ \hdr^1(C/R) $
in the form that is required for our algorithm. This is based
on using forms of the second kind with an additional integrability
condition on the polar parts. We make this
more explicit if~$ C $ is smooth planar in Section~\ref{smooth-planar}.
These two sections are mostly independent of the rest of the paper,
and may be of independent interest.
Section~\ref{complexity} discusses the complexity of the algorithm
for (general) smooth, planar curves, for which we work out our
general estimates in more detail in Example~\ref{gen-planar-ex}.
Finally, Appendix~\ref{algebraic-cup-products} describes
the first de Rham cohomology group of a curve over a field of
characteristic zero, in terms of differential forms of the second kind, as well as a formula, due to Serre,
for the cup product pairing,  in terms of residues and integrals.
This material is well-known, but we could not find a suitable reference.

We would like to thank
Ahmed Abbes,
Fabrizio Andreatta,
Bruno Chiarellotto,
Kiran Kedlaya,
Deepam Patel,
and
Jan Tuitman
for interesting, useful correspondence and discussions, the
late Bas Edixhoven for bringing~\cite{Ka-Lu} to our attention,
and the late Pierre Berthelot for doing this with~\cite{Ber-Ogu83}.

\bigskip

Throughout the paper, we use the following notation.
\begin{notation} \label{global-notation2}
 We let $ R[[x_1,\dots,x_n]] $ denote
 the formal power series ring in variables~$ x_1,\dots,x_n $ with coefficients in~$ R $,
 $ R\la x_1,\dots,x_n \ra $ the subring where the coefficients
 tend to 0 in $ R $, and $ R\la x_1,\dots,x_n \ra^\dagger $ the
 subring of $ R\la x_1,\dots,x_n \ra $ consisting of overconvergent
 power series.  We shall often use multi-index notation, writing $ \x $
 for $ x_1,\dots,x_n $, and $ \x^I $ for $ x_1^{i_1} \dots x_n^{i_n} $
 if $ I = (i_1,\dots,i_n) $. With $ |I| = i_1 + \dots + i_n $,
the elements of~$ \Rxdag $ are those $ \sum_I a_I \x^I $ in
 $ \Rx $ for which $ \cc $ and $ \d $ exist with $ \cc > 0 $ and
 $ v(a_I) \ge \cc |I| + \d $ for all $ I $.
Equivalently, there exists a $ D $ in $ (\realpos)^n $ and $ \d $
in $ \Q $ such that $ v(a_I) \ge D \cdot I + \d $ for all $ I $,
where $ D \cdot I  $ is the inner product.
If~$ A $ is as in Notation~\ref{globalnotation}(2), then we let~$ \Ahat $
(respectively~$ \Adag) $ denote the quotient of~$ R\la x_1, \dots, x_n \ra $
(respectively~$ R\la x_1, \dots, x_n \ra ^\dagger $)
modulo the ideal~$ (f_2, \dots, f_n) $.
\end{notation}

\section{Computing the matrix of Frobenius using cup products and residues} \label{sec:cup-res}

In this section we describe the
strategy for computing the matrix of the Frobenius operator $\myphicr {\pa} \otimes \s $ on
$\hcr^1(C_k / {W(k)})\otimes K $. The cohomology itself is relatively easy to compute, as
it is isomorphic to de Rham cohomology of $C_K$. However, the action of Frobenius is
not easy to capture because usually Frobenius will not lift to $C$. Instead we use Berthelot's rigid cohomology~\cite{Ber97}. This is a functorial cohomology theory $Y \mapsto \hrig^1(Y / K)  $ on the category of $k $-varieties, taking values in finite dimensional $K $-vector spaces with an action of a semilinear operator (Frobenius). This cohomology forms the bridge between various cohomology theories required for our algorithm to work. In particular, we have
(see \cite[(2.3.1)]{Ber97} and \cite[Corollaire~5.7]{Ber97}) that
\begin{equation} \label{hcr-hrig-hdr}
  \hcr^1(C_k / {W(k)})\otimes K \cong 
  \hrig^1(C_k / K) \cong
  \hdr^1(C_K / K)\;.
\end{equation} 
The leftmost isomorphism is compatible with the action of Frobenius by~\cite[Corollary~5.3.4]{cais}. Both isomorphisms exist for proper, smooth varieties in any dimension.
On the other hand, for a smooth affine variety $U $, and in particular  for the open affine $U = C_k \setminus Z$, with~$Z$ 
non-empty of codimension~1, we have an isomorphism~\cite[Proposition~1.10]{Ber97},
\begin{equation}\label{rigtoMW}
  \hrig^1(U / K) \cong\hmw^1(U /K)\;.
\end{equation}
The group on the right is Monsky-Washnitzer cohomology~\cite{Mon-Was68,Put86}, which is the
cohomology group  of a complex of differential forms of a certain
``overconvergent algebra'' $ R $-algebra $\Adag$, which, in degree $1$ and for $U$ of
dimension $1$, translates into a functorial isomorphism 
\begin{equation}\label{eq:coker}
  \hmw^1(U / K)\cong \coker \left(\Adag \otimes_R K \xrightarrow{\dd} \Omega^1(\Adag \otimes_R K / K) \right)
\,.
\end{equation}
The isomorphism~\eqref{rigtoMW} is again compatible with the Frobenius operator by~\cite[Proposition~8.3.12]{lestum}.

We now explain in more detail the Frobenius action on all three cohomology theories. For rigid  cohomology this is a combination of functoriality and a change of coefficients. Namely, for any $k $-variety $Y $ consider the relative Frobenius $\phibar$ on $Y$, i.e., the morphism of $k $-varieties
\begin{equation}\label{relfrob}
  \phibar: Y\to Y^{(\pa)}, \text{ where } Y^{(\pa)} =  Y\times_{k,\ol{\s}} k
\,,
\end{equation}
obtained by raising to the $\pa$-th power on the structure sheaf. By functoriality, the $k $-morphism $\phibar $ of~\eqref{relfrob} induces a $K $-linear map
\begin{equation}
  \phibar^\ast:  \hrig^1( Y^{(\pa)}/ K) \to  \hrig^1(Y / K)\;. 
\end{equation}
By~\cite[Proposition~1.8]{Ber97}, the morphism $\sigma : K \to K $ induces an isomorphism
\begin{equation}
  \hrig^1(Y /K) \otimes_{K,\sigma } K  \to \hrig^1(  Y^{(\pa)} /K)\;,
\end{equation}
which is the same thing as a $\s $-semilinear isomorphism $ \hrig^1(Y /K) \to \hrig^1(  Y^{(\pa)} /K) $. Composing, we get the $\s $-semilinear Frobenius. For crystalline cohomology we use the same functoriality and change of coefficients approach, first taking coefficients in $W(k) $ and the restriction of $\sigma  $ to this ring, to get a $\sigma  $-semilinear map $\myphicr {\pa}: \hcr^1(C_k/W(k)) \to \hcr^1(C_k/W(k))  $. We note that this map may be associated directly to the absolute $\pa $-th power Frobenius of $Y $ and the compatible morphism $\sigma  $ on $R $. We then tensor up to $K $ with $\sigma  $ acting on $K $ to get
\begin{equation*}
\myphicr {\pa} \, \otimes \, \s : \hcr^1(C_k/W(k)) \otimes_{W(k)} K \to\hcr^1(C_k/W(k)) \otimes_{W(k)} K
.
\end{equation*}

In contrast to this inexplicit description of Frobenius on crystalline and rigid cohomologies,  the action of Frobenius on Monsky-Washnitzer cohomology is given directly by
the induced action of a certain $\s $-semilinear endomorphism $\phidagA$ (called a lift of Frobenius) of the
algebra $\Adag$. The explicit computation of this map is one of the key
ingredients of our algorithm. Overall, using rigid cohomology as an intermediate, we get a map of cohomology theories, commuting with a $\sigma  $-semilinear Frobenius;
or equivalently, a commutative diagram
\begin{equation} \label{robswish1}
\begin{split}
\xymatrix{
\hcr^1(C_k/W(k)) \otimes_{W(k)} K \ar[r] \ar[d]^-{\myphicr {\pa} \, \otimes \, \s } & \hmw^1(U/K) \ar[d]^-{\phidagA}
\\
\hcr^1(C_k/W(k)) \otimes_{W(k)} K \ar[r] & \hmw^1(U/K)
,}
\end{split}
\end{equation}
where we used the identication in~\eqref{eq:coker} in order to
let~$ \phidagA $ act on the Monsky-Washnitzer cohomology.
This diagram is the basis of our computation of the crystalline Frobenius.

Past methods for point counting using Monsky-Washnitzer cohomology,
starting with Kedlaya's algorithm~\cite{Ked01}  relied
on effective reduction -- identifying the image
of~$\omega $ in $ \Omega^1(\Adag / K)$ under~$ \phidagA $
in~\eqref{eq:coker} explicitly as a
linear combination of the elements of a fixed basis of $\hmw^1(U/K)$.
Another key component of our algorithm is to avoid this
reduction, and instead use the computation of a cup product pairing.

Poincar\'e duality (see~\cite[\href{https://stacks.math.columbia.edu/tag/0FW7}{Tag 0FW7}]{stacks-project})
tells us that the cup product in algebraic de Rham cohomology,
followed by the trace map, gives a non-degenerate pairing
\begin{equation}\label{eq:cppair}
\cup : \hdr^1(C_K /K) \times \hdr^1(C_K /K) \longrightarrow \hdr^2(C_K
/K) \xrightarrow{\tr} K
.
\end{equation}

The map $U \hookrightarrow C_k $ induces by functoriality a map in rigid cohomology, sitting in a short exact sequence,
\begin{equation}\label{shorti}
    0\to \hrig^1(C_k / K) \to \hrig^1(U / K) \xrightarrow{\myproj}  \hrig^0(Z/K)
,
\end{equation}
that identifies, as vector space with a Frobenius action,  $\hdr^1(C_K /K)$ with the kernel of $\pi$ in
\eqref{shorti}. We may therefore view the cup product as a non-degenerate pairing
on $\ker(\pi)$. We shall show how, starting from the description
\eqref{eq:coker}, we can explicitly do the following.
\begin{enumerate}
\item
Describe $\pi$, hence $\ker(\pi)$.

\item
Describe the cup product pairing~\eqref{eq:cppair} in terms of~$\ker(\pi)$.
\end{enumerate}
The space~$\ker(\pi)$ is thus the appropriate place to do point counting, as it, on the
one hand, has an explicit description of the lift of Frobenius, coming from
  the description of $\phidagA$ on $\Omega^1(\Adag/K)$, and has computable
  coordinates, coming from the cup producs with a fixed basis of
  $\hdr^1(C_K/K)$.

\subsection{Rigid analytic residues and cup products}\label{sec:rigcup}
We now turn to an explicit description of the map $\myproj$ in~\eqref{shorti}, and the cup
product pairing on its kernel. We shall use a geometric description, based on
Coleman's work~\cite{Col-de88,Col89}. We then translate this into the more algebraic
language that will be used in the rest of the paper.

To rely directly on Coleman's work, it is convenient to first change
scalars to the field $\Cp$ of ``complex $p$-adic numbers''. Recall
that this is the completion of an algebraic closure of~$\Qp$. Its
residue field is an algebraic closure $\fpb$ of the finite field~$ \F_p $
with $p$ elements. Thus, we shall work with a curve $C$ defined over $\Cp$
and only later assume that it is obtained by a base change from a
curve defined over $K$.

We shall use the theory of rigid analytic spaces, with~\cite{Fre-Put04}
as reference. Recall that a rigid analytic space $Y$ has rings of
functions $\O(Y)$ and differential forms~$\Omega^i(Y)$ with obvious
differentials. Recall further from~\cite[Example 4.3.3]{Fre-Put04} that an
algebraic variety $Y$ over our complete non-archimedean field
$K$ has an associated rigid analytic space $Y^{\an}$. We shall often abuse
notation when there is no danger of confusion to say that we consider $Y$
as a rigid analytic space. In particular, we consider~$C$ as a rigid analytic space over $\Cp$
by means of a fixed embedding of~$ K $ into~$ \Cp $.

We shall assume that the curve $C$ has good reduction. This means that it
has a smooth $\O_{\Cp}$-model. Denote by $Z$ the special fibre of this
model, which is a smooth~$\fpb$-variety (which in our application is obtained
from the closed fibre~$ C_k $ by extending the coefficients~$ k $
to~$ \fpb $). There is a reduction map (of
topological spaces with respect to the appropriate topologies),
$\reduction : C^\an \to Z$~\cite[p.100]{Fre-Put04}. For a closed point $q$ in~$ Z$, the
inverse image $\reduction^{-1}(q)$ is a rigid
analytic space $\D$.
Because $C$ has good reduction, the space $\D$ is
isomorphic, via a parameter $t$, to an open unit disc,
\begin{equation}
  t: \D \to \{z \text{ in } \Cp \text{ with } |z|<1\}\; .
\end{equation} 
Coleman calls $\D$ a residue disc.
Let $ S=\{q_1,\dots,q_m\} \subset Z $ be a finite non-empty set of
closed points,
$\D_i \subset C^\an $ the corresponding residue discs,
and fix corresponding parameters
\begin{equation}\label{tidef}
  t_i: \D_i \to \{z \text{ in } \Cp \text{ with } |z|<1\}
.
\end{equation} 
For each $0 \le r < 1$
we let $U_r$ be the rigid analytic
subspace of $C^\an$ obtained
by removing the subsets $\{x \text{ in } \D_i \text{ with } |t_i(x)|\le r\} $
(cf.~\cite[2.1]{Col-de88}). These~$U_r$ are examples of ``basic wide open
spaces'' in Coleman's terminology. We will write $U_1 $ for the space obtained by removing the \emph{open} discs of radius $1 $. This is what Coleman calls an underlying affinoid for the basic wide open $U_r $.

A rigid analytic space $Y$ over a field $K$ has associated to it de Rham
cohomology groups, $\hdre^\ast(Y/K)$~\cite{Kie67c}. These are given, just as in the algebraic setting,
as the hypercohomology of the complex of rigid differential forms.
If $X$ is a smooth algebraic variety, then by~\cite[Theorem~2.4]{Kie67c}
there is a natural isomorphism
\begin{equation}\label{eq:iso1}
  \hdr^\ast(X/K) \xrightarrow{\sim} \hdre^\ast(X^\an /K)
.
\end{equation}

On the other hand, for the space $U_r$, $0\le r < 1 $, we have explicitly, by the last line on page~225
of~\cite{Col89} that
\begin{equation} \label{hdreUr}
  \hdre^1(U_r / \Cp) = \frac{\Omega^1(U_r)}{\dd \O(U_r)}\;.
\end{equation}

Let $U_0$ be the space obtained from the rigid analytic space $C^\an$ by removing just the points
defined by~$t_i=0$ (in fact, we can remove a single arbitrary point in each disc
by using suitable~$ t_i $).
Note that $U_0$ is the analytification of an algebraic curve, which we also
denote by~$U_0$.
Coleman proves the following theorem.

\begin{theorem}[\bf Coleman]
  For any $0< r<1$ we have an isomorphism 
    \begin{equation*}
	\hdr^1(U_0 / \Cp) \xrightarrow{\sim} \hdre^1(U_0 / \Cp)
	\xrightarrow{\sim} \hdre^1(U_r / \Cp)
	\,.
    \end{equation*}
\end{theorem}

\begin{proof}
\cite[Theorem 4.2]{Col89} and its proof.
\end{proof}

We note that the first isomorphism is
  just~\eqref{eq:iso1}, and is given by sending an algebraic 1-form on
  $U_0$ to the same form viewed as an analytic form. Let us define
\begin{equation} \label{lambda-def}
\mynewmap :   \hdr^1(C/ \Cp) \to \hdre^1(U_r / \Cp)
\end{equation}
as the composition~$ \hdr^1(C/ \Cp) \to \hdr^1(U_0/ \Cp) \to \hdre^1(U_r / \Cp) $.
Note that this is the same as the composition~$ \hdr^1(C/ \Cp) \to \hdre^1(C/ \Cp) \to \hdre^1(U_r / \Cp) $.
Together with Corollary~\ref{rest} we deduce the following.

\begin{corollary}\label{sectorig}
    Let $\omega$ be a form of the second kind on  $C$ with all of its
    singularities outside $U_r$.
Then its class~$[\omega] $ in~$ \hdr^1(C / \Cp)$ is mapped under $\mynewmap$  to the class
of~$\omega|_{U_r}$ in~\eqref{hdreUr}.
\end{corollary}

Recall that an annulus is a rigid analytic space $\ee$, isomorphic, via a
parameter~$t$, to a space
of the form $\{z \text{ in } \Cp \text{ with } r<|z|<1\}$ for some $0\le r < 1 $.
The space of rigid analytic functions
on the annulus~$ \ee $ can then be described as
\begin{equation}
  \label{eq:annulus-fun}
  \left\{
\begin{aligned}
& \sum_{m\in \Z} a_m t^m\text{ with }a_m \text{ in } \Cp \text{ satisfying }
\\
&   \lim_{m\to -\infty} |a_m|  s^m = 0 \text{ for all } s>r \text{ and}
\\
& \lim_{m\to \infty} |a_m|  s^m = 0 \text{ for all } s<1
\end{aligned}
\right\}
\end{equation}
and 1-forms are just rigid functions multiplied by $\dd t$.
\begin{definition} \label{rigid-res-def}
Let~$\ee$ be an annulus with parameter~$t$.
For a rigid analytic 1-form~$ \omega = \sum_{m\in \Z} a_m t^m \dd t $ on~$ \ee $
we set $\res_{\ee} \omega $, the \emph{residue} of $ \omega $ on $\ee$ with respect to the
parameter $t$, to be~$ a_{-1}$.
\end{definition}

The set of all parameters on  an annulus $\ee$ breaks into two classes,
known as 
orientations~\cite[Lemma~2.1 and ensuing remarks]{Col89} such that the
residues with respect to
any two parameters are identical if they are in the same orientation
and differ by a sign otherwise.
An annulus with a choice of orientation is called an oriented annulus.

\begin{remark}\label{shrinkrm}
If an annulus $\ee$ with parameter~$t$ is defined by~$ r < |t| < 1 $,
and we define~$\ee' \subseteq \ee$ for $ r < r' < 1 $ by~$ r' < |t| < 1 $,
then clearly~$\res_\ee = \res_{\ee'}$, if we give both~$ \ee $
and~$\ee'$ the orientation induced by~$ t $.
\end{remark}

Let $t_i$ be the parameters introduced in~\eqref{tidef} on the rigid
analytic discs $\D_i$. For any $0<r<1$ and any $i$, the conditions $r< |t_i|<1$ define annuli~$\ee_i$
inside~$ \D_i $. We clearly have
\begin{equation*}
  U_r = U_1 \cup \bigcup_{i} \ee_i \;.
\end{equation*}
The annuli $\ee_i$ are
oriented by the parameters~$t_i$, and all parameters that are obtained
on the~$\ee_i$ as restrictions of parameters on~$\D_i$ give the same
orientation~\cite[Cor.~3.7a]{Col89}. The choice of $t_i$ is therefore
irrelevant for the residue.

\medskip

The analogue of the residue theorem holds (see \cite[Proposition~4.3]{Col89}).

\begin{theorem} \label{resthm}
  For $ 0 \le r < 1 $ and a rigid analytic 1-form $\omega$ on $U_r$ we have
  $\sum_i \res_{\ee_i} \omega = 0$.
\end{theorem}

\begin{prop}\label{resone}
  Let $\D$ be one of the residue discs~$ \D_i $ above and let $\ee = \ee_i $ be the
    corresponding annulus. Let $\omega$ be a meromorphic analytic form on
    $\D$ whose restriction to $\ee$ is analytic, and such that $\omega$ has a
    finite number of poles on $\D$.
Then 
    \begin{equation*}
      \res_\ee(\omega)= \sum_x \res_x \omega\;,
    \end{equation*} 
    where $x$ runs over all points of $\D$.
\end{prop}

\begin{proof}
    This is an easy consequence of the residue Theorem~\ref{resthm}. For
    the method of proof in a more complicated situation see the proof of
    Proposition~5.5 in \cite{Bes98b}
\end{proof}

Since the residue of an exact form is $0$, the residues factor through the
analytic de Rham cohomology of $U_r$, and summing over $i$ we get a map
\begin{equation*}
   \res : 
    \hdre^1(U_r / \Cp)
    \to  \oplus_{i=1}^m \Cp
\end{equation*}

\begin{theorem}\label{basicshort}
With~$ \mynewmap $ as in~\eqref{lambda-def}, there is an exact sequence
  \begin{equation*}
    0\to \hdr^1(C / \Cp) \xrightarrow{\mynewmap} \hdre^1(U_r / \Cp)
    \xrightarrow{\res} \oplus_{i=1}^m \Cp\;.
  \end{equation*}
\end{theorem}

\begin{proof}
  This is~\cite[Proposition~4.4]{Col89}.
\end{proof}

\begin{definition} \label{rigid-sec-kind}
(1)  A rigid analytic form $\omega$ on $U_r$ will be called of the second
  kind if~$\res_{\ee_i} \omega = 0 $ for every annulus
  $\ee_i$.

(2) We denote the image of $\mynewmap$ in Theorem~\ref{basicshort}
by~$\hdree^1(U_r / \Cp)$ and call it the subspace of  cohomology classes
of the second kind. 
\end{definition}

By Theorem~\ref{basicshort} the cohomology classes of the second kind in~$ \hdre^1(U_r / \Cp) $
are precisely those represented by
rigid analytic forms of the second kind. It
also follows from Proposition~\ref{resone} that an algebraic form of the second kind
on $C $ that is regular on $U_r$ restricts to a form of the second kind on~$U_r$.

\begin{definition}
  If $\omega,\eta$ are rigid analytic forms of the second kind on $U_r$,
  in the sense of Definition~\ref{rigid-sec-kind}(1), their cup product pairing is defined by
  \begin{equation} \label{serre-rigid-formula}
    \pair{\omega,\eta}_{U_r} = \sum_i \res_{\ee_i} \left( \eta \int \omega \right)
.
  \end{equation}
\end{definition}

It is easy to check that~$ \res_{\ee_i} ( \eta \int \omega ) = - \res_{\ee_i} ( \omega \int \eta ) $
for each~$ i $ (see~\eqref{cupreslocal}). It then follows from the rigid analytic residue theorem,
Theorem~\ref{resthm}, that, just like in the algebraic setting of Theorem~\ref{seckindcup},
the pairing~\eqref{serre-rigid-formula} is well-defined, alternating,
and factors via $\hdre^1(U_r / \Cp)$.
It is further clear from
Remark~\ref{shrinkrm} that if $U_{r'}$ is a smaller wide open space as above
then 
\begin{equation}\label{paircr}
  \pair{\omega|_{U_{r'}},\eta|_{U_{r'}}}_{U_{r'}} = \pair{\omega,\eta}_{U_r}
\,.
\end{equation}

The pairings in~\eqref{serre-rigid-formula} and~\eqref{eq:cppair} are
connected by the following result.

\begin{proposition}[{\cite[Proposition~4.5]{Col89}}]\label{cuppair}
  Let $\alpha_1 $ and $\alpha_2 $ be in $ \hdr^1(C/\Cp)$.
  If $\omega_1$ and $\omega_2$ are rigid analytic forms of the second kind
  on $U_r$ such
  that the class of $\omega_i$ in $\hdre^1(U_r/\Cp)$
is $\mynewmap(\alpha_i)$, then
$\tr(\alpha_1\cup \alpha_2) = \pair{\omega_1,\omega_2}_{U_r}$.
\end{proposition}

\subsection{Putting it all together}
We return to the setting of a smooth curve~$ C/R $ as at the
beginning of Section~\ref{sec:cup-res}.
Then the constructions of Section~\ref{sec:rigcup} make sense already over
some finite extension~$ L $ of~$ K $ such that each of the residue
discs corresponding to the ends has an~$ L $-rational point.
To simplify the discussion, suppose first that $K=L $.
In particular, the domains~$ U_{r} $ and the annuli $ \ee_i $
are defined over~$ K $, and we have residue maps $ \res_{\ee_i} : \Omega^1(\ee_i) \to K $  as in Definition~\ref{rigid-res-def}, a residue theorem as in Theorem~\ref{resthm}, a short exact sequence
  \begin{equation}\label{Kshort}
    0\to \hdr^1(C_K / K) \xrightarrow{\mynewmap} \hdre^1(U_r / K)
    \xrightarrow{\res} \oplus_{i=1}^m K
  \end{equation}
as in Theorem~\ref{basicshort}, with the map $ \mynewmap $ satisfying Corollary~\ref{sectorig} (over $ K $). Finally, we have a pairing
\begin{equation} \label{serre-over-L}
  \pair{~,~}_{U_r}: \ker(\res) \times \ker(\res) \to K,\quad
    \pair{\omega,\eta}_{U_r} = \sum_i \res_{\ee_i} \left( \eta \int \omega \right)
\,,
\end{equation}
satisfying the analogue of Proposition~\ref{cuppair}.
All of this is immediately proved by base changing to $ \Cp $ and applying the corresponding results there.

The connection with Monsky-Washnitzer cohomology is made by the formulae
\begin{equation}\label{relwithMW}
 \Adag = \varinjlim_{r\uparrow 1} \O(U_r),\quad 
   \Omega^1(\Adag /K) = \varinjlim_{r\uparrow 1} \Omega^1(U_r)
,
\end{equation}
which are proved in the course of proving Proposition 1.10 in~\cite{Ber97}.

\begin{definition}
(1) Define a map $ \mynewmap: \hdr^1(C_K /K) \to \hmw^1(U /K) $  as follows. Let~$ \omega $ be a form of the second kind on $ C_K $ that has no poles in $ U_1 $. Then there is some $ r<1 $ such that $ \omega  $ has no poles in $ U_r $ and we define $ \mynewmap ([\omega]) $ to be the class represented by the image of $ \omega|_{U_r} $ in $ \Omega^1(\Adag / K)  $.

(2) Define a map $ \res: \hmw^1(U /K) \to  \oplus_{i=1}^m K $  as follows. Suppose~$ x $ in~$\hmw^1(U /K) $
is represented by some $ \eta $ in~$ \Omega^1(U_r) $ for some $ r<1 $. Then let $ \res x = \res \eta  $.
\end{definition}

\begin{lemma}
    The maps $ \mynewmap $  and $ \res $  are well-defined and fit in a short exact sequence
  \begin{equation}\label{MWshort}
    0\to \hdr^1(C / K) \xrightarrow{\mynewmap}  \hmw^1(U /K)
    \xrightarrow{\res} \oplus_{i=1}^m K
  \end{equation}
\end{lemma}

\begin{proof}
The fact that $ \mynewmap $ is well-defined is clear. That $ \res $ is well-defined follows from the compatibility of
the residue maps as~$ r $ grows, for which we extend scalars to $ \Cp $ and use Remark~\ref{shrinkrm}.The short exact sequence~\eqref{MWshort} is now a consequence of the compatible short exact sequences~\eqref{Kshort}. 
\end{proof}

\begin{definition}
    The pairing $ \pair{~,~} $ on $ \ker(\res) : \hmw^1(U /K) \to  \oplus_{i=1}^m K $  is defined as follows.
    Represent~$ x $ and~$y $ in~$ \ker(\res) $ by forms $ \omega$ and~$\eta $ in~$ \Omega^1(U_r) $ for a sufficiently large $ r $, and set
\begin{equation} \label{newpair}
  \pair{x,y} = \pair{\omega ,\eta }_{U_r}
\end{equation}
\end{definition}

Finally, let us address the general case, where residue discs acquire rational points only over a finite extension $L $ of $K $. The domains $U_r $ can still be defined over $K $ (see the construction of the domains $V_{\lambda } $ in the proof of Proposition~1.10 in~\cite{Ber97}) and the relation~\eqref{relwithMW} with Monsky-Washnitzer cohomology still holds by the above reference to~\cite{Ber97}.
\begin{proposition}
  We have a short exact sequence
   \begin{equation}\label{MWshortK}
    0\to \hdr^1(C / K) \xrightarrow{\mynewmap}  \hmw^1(U /K)
    \xrightarrow{\res} V
  \end{equation}
  with $V $ a finite dimensional $K $-vector space, and a pairing
\begin{equation} \label{serre-over-K}
\pair{~,~} :  \ker(\res) \times \ker(\res) \to  K
\end{equation}
obtained by changing scalars to a finite Galois extension $L $ of $K $ where residue discs acquire rational points and computing the pairing~\eqref{newpair} over $L $.
\end{proposition}

\begin{proof}
  Galois equivariance will allow us to descend $ \oplus_{i=1}^m L $ to~$ V $,
in which the descended  residue map takes values. It also shows
that the pairing~\eqref{serre-over-K} indeed takes values in $K\subseteq L $.
\end{proof}

\begin{remark}\label{robswish2}
Note that the short exact sequence~\eqref{MWshortK} is just the sequence~\eqref{shorti} but with explicit maps. Thus, we have a commutative diagram
\begin{equation*}
\xymatrix{
\hcr^1(C_k/W(k)) \otimes_{W(k)} K \ar[r] \ar[d] & \hmw^1(\Abar/K)
\\
\hdr^1(C_K/K) \ar@{=}[r] & \hdr^1(C_K/K) \ar[u]_-{\mynewmap}
,
}
\end{equation*}
where the top horizontal map is the horizontal map in~\eqref{robswish1}.
Also note that we do not prove, nor do we need to prove, that the two sequences \eqref{MWshortK} and~\eqref{shorti}
are the same.
(The above diagram is  compatible with~\eqref{hcr-hrig-hdr}, which we also do not need to prove.) Finally, the extension of scalars to $L $ will be needed to compute the pairing, but we will not need to compute the residue map directly, as the elements we willl be using will come via $\mynewmap $ and the application of Frobenius, and will live in the kernel of $\res $ without explicitly computing it.
\end{remark}

\medskip

We shall later need the following compatibility result.

\begin{lemma} \label{compatible-pairings}
The pairing~\eqref{newpair} is well-defined. Moreover, for $ x' $ and~$y' $ in
$ \hdr^1(C /K) $ we have
\begin{equation}\label{cupMW}
   \tr(x'\cup y') = \pair{\lambda (x'),\lambda(y')}\;.
\end{equation}
\end{lemma}

\begin{proof}
The pairing is independent of $ r $ by~\eqref{paircr} and factors through cohomology by Theorem~\ref{resthm}. The formula~\eqref{cupMW} follows immediately
from Proposition~\ref{cuppair}.
\end{proof}

Suppose now that  $\omega_1,\ldots,\omega_{2g}$ in~$ \Omega^1(U_r) $
for some~$ 0 < r < 1 $ are forms such that their classes in~$ \hmw^1(U/K) $
form a basis of~$ \ker(\res) $ in~\eqref{MWshort}.
Let  $M_1$ be the matrix with entries~$ \pair{[\o_i], [\o_j]} $,
and let $ M $ be the matrix representation, acting on column vectors, of the operator $ \phidagA $ on $ \ker(\res) $
with respect to this basis. (Note that this operator is semilinear.)
Then the matrix with entries~$ \pair{ [ \omega_i] ,\phidagA ( [\omega_j])  }$ is~$M_2=M_1 M$. This suggest the following method for computing~$ M $.

\begin{method} \label{cupmethod}
(1)
Compute the matrix $M_1$ with entries $\pair{ [\omega_i] ,  [\omega_j] } $ using~\eqref{newpair}.

(2)
Compute an explicit lift~$ \phidagA $ of the~$ \pa $-th power
map.

(3)
Compute the matrix $M_2$ with entries $ \pair{[ \omega_i] ,\phidagA( [\omega_j] ) } $ using~\eqref{newpair}.

(4)
Compute the matrix $M=M_\s$ as $ M_1^{-1}M_2 $.
\end{method}

From~(3.2.6) on~\cite[p.583]{Ber74} (or \cite[(1.3.4)]{Ill94})
we know that the numerator of the zeta function of~$ C_k $ is~$ \det(1 - \myphicr {|k|} T) $,
and~$ \myphicr {|k|} = \left( \myphicr {\pa} \right)^l $ if~$ |k| = \pa^l $.
(Note that the formula for the denominator is based on the Lefschetz trace formula~\cite[Th\'eor\`eme~3.1.9]{Ber74},
which requires an endomorphism of~$ C_k $ as a~$ k $-scheme,
so we cannot allow an action of~$ \, \ol{\s} $.)
But instead of the action of~$ \myphicr {\pa} $ on~$ L_{\cry} = \hcr^1(C_k/W(k)) $,
we have an action of~$ \myphicr {|k|} \otimes \t $ on~$ L_{\cry} \otimes_{W(k)} K \simeq \hdr^1(C_K/K) $ for some~$ \t $
in~$ \Aut(K/\Qp) $ that reduces to the identity on~$ k $ and~$ W(k) $.
If~$ \t $ is the identity,\footnote{%
It seems not easy to recover even~$ L_{\cry} \otimes_{W(k)} K_0 \subseteq L_{\cry} \otimes_{W(k)} K $
for~$ K_0 $ the field of fractions of~$ W(k) $. E.g., if~$ K/K_0 $
is of degree~2, and we decompose~$ K $ into~$ K_0 \oplus K^\chi $ corresponding
to the characters of the Galois group~$ \{\id_K, \t \} $,
then~$ L_{\cry} \otimes_{W(k)} K $ decomposes into the~$ K_0 $-subspaces~$ L_{\cry} \otimes_{W(k)} K_0 $,
which is the one we want and has the correct action of~$ \myphicr {|k|} \otimes \t $, and~$ L_{\cry} \otimes_{W(k)} K^\chi $,
of the same~$ K_0 $-dimension but where the action is multiplied by~$ \chi(\t) = -1 $.}
then we can compute
the required determinant by going from a~$ W(k) $-basis of~$ L_{\cry} $
to a~$ K $-basis of~$ \hdr^1(C_K/K) $.

The zeta function can therefore now be obtained as follows.

\begin{method} \label{zeta-method}
Using Method~\ref{cupmethod}, compute the zeta function of $C_k$ as follows.
Let~$ \pa $ be the order of a subfield of~$ k $, so~$ |k| = \pa^l $ for~$ l $ the order of~$ \ol{\s} $.
Assume that~$ \s $ also has order~$ l $.

(1)
Compute the matrix $M = M_\s$ of~$\myphicr {\pa}  \otimes \s $.

(2)
Compute the matrix of the linear Frobenius as
\begin{equation*}
      M' = \s^{l-1}(M)\times \s^{l-2}(M)\times \cdots \times
      \s(M)\times M
.
\end{equation*}

(3)
Let $P_1(T)= \det(1-T M')$.
Its coefficients, a priori in $\Zp$, are integers.

(4)
Deduce the zeta function as $ Z(T) = \frac{P_1(T)}{(1-T)(1-|k| T)}$.
\end{method}

\begin{remark} \label{non-order}
If~$ e \ne 1 $, and~$ |k| = \pa^l $, then it can happen that~$ \s $ has order divisible
by, but not equal to, the order~$ l $ of~$ \ol{\s} $. In that
case, let~$ \t = \s^{1-l} $, so that~$ \ol{\t} = \ol{\s} $, and~$ \t \s^{l-1} = \id_K $.
In that case also~$ \t $ can be used in~\eqref{robswish1} instead of~$ \s $,
and we can obtain a lift of~$ \myphicr {|k|} \otimes \id_K $ by first using a~$ \t $-semilinear lift~$ \phidagAt $, followed
by~$ l-1 $ times a~$ \s $-semilinear lift~$ \phidagA $. This gives the matrix
\begin{equation*}
M' = \s^{l-1}(M_\t) \times \s^{l-2}(M_\s)\times \cdots \times \s(M_\s)\times M_\s
,
\end{equation*}
with~$ M_\t $ computed as in Method~\ref{cupmethod} for~$ \phidagA $.
Of course, if~$ \s^l = \id_K $ then~$ \t = \s $ and we can recover~$ M' $
as in Method~\ref{zeta-method}.
\end{remark}

In Section~\ref{algorithm} we shall work out into a detailed
algorithm how to use Methods~\ref{cupmethod} and~\ref{zeta-method}. In particular,
we shall use as a~$ K $-basis of~$ \ker(\res) $ in~\eqref{MWshort}
the image under~$ \mynewmap $ of an~$ R $-basis of~$ \hdr^1(C/R) \subseteq \hdr^1(C_K/K) $
as this will help us to avoid, or at least bound below, the valuations
of the entries of our matrices, as we shall discuss in Remarks~\ref{detrem}
and~\ref{h-remark}.
After obtaining more auxiliary results, in Section~\ref{algorithm-estimates},
we then give a version of this algorithm that uses finite precision version
and truncated expansions.

\begin{remark} \label{detrem}
If we use an~$ R $-basis of~$ \hdr^1(C/R) $ as basis for~$ \ker(\mynewmap) $
in~\eqref{MWshort}, then Lemma~\ref{compatible-pairings} implies that the entries of~$ M_1 $
in Method~\ref{cupmethod} are in~$ R $. Moreover, by Poincar\'e
duality, its determinant is in~$ R^* $.
\end{remark}

\section{The algorithm} \label{algorithm}

We now work out Methods~\ref{cupmethod}
and~\ref{zeta-method} into an algorithm, based on infinite precision,
and full expansions.
A version that uses finite precision and truncated expansions will be given in Section~\ref{algorithm-estimates}.

We shall formulate the algorithm below using the image in~\eqref{MWshort}
under~$ \mynewmap $ of an~$ R $-basis of $ \hdr^1(C/R) $ because then
the matrix~$ M_1 $ introduce in Method~\ref{cupmethod} has entries
in~$ R $ and determinant in~$ R^* $, which will simplify these estimates.
We shall discuss how to compute such a basis in practice in Section~\ref{de-rham}
in the general case, and in Section~\ref{smooth-planar} in the case of a smooth
planar curve.

The algorithm here also works if, instead, one uses a~$ K $-basis of~$ \hdr^1(C_K/K) $
but then the estimates for convergence in Section~\ref{algorithm-estimates} get more complicated
and a higher $p$-adic precision may be needed;
see Remark~\ref{ComputewithKbasis} for details.

The main setup for the algorithm, and some of the techniques
that will be used, were already described in Section~\ref{sec:cup-res}.
Others will be introduced and described in detail before we can
give the estimates in Section~\ref{algorithm-estimates}.
In order to help the reader, we here briefly describe the later assumptions
or constructions in the algorithm.

We assume that we are given an open affine part of~$ C $ as a
smooth global complete intersection (see Assumption~\ref{Aassumption}),
which contains the generic point of the closed fibre~$ C_k $.
Using the smoothness assumption, and extending~$ K $ if
necessary, for each end~$ \ee $ we can choose a suitable local parameter~$ t $,
and find local expansions of the $ x_i $ in terms of~$ t $.
From these we can compute the contribution of~$ \ee $ to the pairing
on~$ \hdr^1(C/R) $ by means of~\eqref{serre-rigid-formula} (see
Lemma~\ref{compatible-pairings}). Summing over the~$ \ee $ we can then compute
the matrix~$ M_1 $ in Step~2 of the algorithm.

For the matrix~$ M_2 $ we proceed as follows.
From the smoothness assumption we obtain certain matrices $ P $
and~$ \Delta_j $ that enable us to write down a system of equations~$ G(\S) = 0 $
(see~\eqref{globalG}) that determines by~\eqref{eq:phiact}
a (then) unique lift of Frobenius~$ \phidagA $ on~$ \Adag $,
obtained from such a lift~$ \phidag $ on~$ R\la \x \ra^\dagger $.
For each end~$ \ee $, using the local expansions we can transform
this into a local system~$ H(\S) = 0 $, which
determines the local expansions of the $ \phidagA(x_j) $
(see Theorem~\ref{locallift0} or Step~1(c) of the algorithm).
Solving~$ H(\S) = 0 $ through Newton iteration we can 
then compute the contribution of~$ \ee $ to the pairing~\eqref{serre-rigid-formula}
if we first apply~$ \phidagA $ to~$ \eta $.
Summing over the~$ \ee $ we can then compute
the matrix~$ M_2 $ in Step~2 of the algorithm.

In order to complete the algorithm we compute~$ M = M_1^{-1} M_2$ as in Method~\ref{cupmethod},
as well as~$ M' $ and~$ P_1(T) = \det(1-T M') $
as in Method~\ref{zeta-method}.

We emphasise that the system~$ G(\S) = 0 $ is never solved globally
in this approach, which would involve Newton iteration in a power series in several
variables. Apart from some calculations in a polynomial
ring, all of our calculations take place in (infinite but converging)
Laurent series in one local parameter instead.
Similarly, the matrices $ \Delta_j $ are actually never used
in the local calculations.

We make explicit how to calculate Method~\ref{zeta-method},
so we are assuming that~$ \s $ and~$ \ol{\s} $ have the same order~$ l $.

\begin{algorithm} \label{mainalgorithm}
Method~\ref{zeta-method} can be carried out as follows.
For simplicity, if~$ \o $ is a 1-form representing a class of~$ \hdr^1(C/R) $
or~$ \hdr^1(C_K/K) $, we simplify the notation by 
writing~$ [\o] $ for what is~$ \mynewmap([\o]) $ in~$ \hmw^1(U/K) $
in~\eqref{MWshort}.

\medskip

\noindent
{\bf Input}
\begin{itemize}
\item
A presentation of an $ R $-algebra $ A = R[\x]/(f) = R[x_1,\dots,x_n] / (f_2, \dots, f_n) $
as in Notation~\ref{globalnotation}
that satisfies 
Assumption~\ref{Aassumption}, and such that $ A $ corresponds to an open affine~$ X \subset C $
with~$ X \cap C_k $ non-empty.

\item
Matrices $ P $ and $ \Delta_{r+1},\dots,\Delta_n $ in $ \M^{n,n-1}(R[\x]) $ and
$ \M^{n-1}(R[\x]) $ respectively, such that
$ \Jac_{(f)}\times P \equiv \Id_{n-r} + \sum_{j=r+1}^n f_j\Delta_j $
modulo~$ \pi $, from which we obtain the system~$ G(\S) = 0 $
(see~\eqref{globalG}).

\item
Elements~$ \o_1,\dots,\o_{2g} $ in~$ \Omega_{A/R}^1 $ that give
a basis of~$ \hdr^1(C/R) $.
(Such elements can be computed using Proposition~\ref{compdr}
together with Algorithms~\ref{KtoRbasis}  and~\ref{KtoRbasisOmega}.)

\item 
For each end~$ \ee $ a suitable local parameter~$ t = t_\ee $ (see Section~\ref{expansions}),
together with the local expansions~$ \ex_\ee(x_i) $.
(In this step, we may have to use a finite extension of~$ K $.)
\end{itemize}

\medskip

\noindent
{\bf Step 1: local calculations.}

\noindent
For each end~$ \ee $ do the following, using the local expansions~$ \ex_\ee(x_1), \dots, \ex_\ee(x_n) $.

\begin{enumerate}
\item[(a)]
Compute all~$ \res_{\ee} (\o_j \int \o_i ) $
as in Definition~\ref{rigid-res-def}.

\item[(b)]
Compute~$ H(\S) $ (see Theorem~\ref{locallift0})
from~\eqref{globalG} by using the~$ \ex_\ee(x_i) $.

\item[(c)]
Using Newton iteration, compute the solution $ \soll $ in $ \Rtstar^n $
(see Notation~\ref{Ralphabetanotation}) of $ H(\S) = 0 $,
as in Corollary~\ref{locesti}, starting from~$ z = 0 $.

\item[(d)]
Compute all~$ \ex(\phidagA(x_i)) = x_i(t)^p + \sum_{j=2}^n \ex(\psi(P_{i,j})) \soll_j $
where~$\psi$ is as in~\ref{eq:phiact}.

\item[(e)]
Compute all $ \res_{\ee} ( \phidagA (\o_j) \int \o_i ) $
as in Definition~\ref{rigid-res-def}, using the results of~(d).
\end{enumerate}

\medskip

\noindent
{\bf Step 2: global calculations.}

\begin{enumerate}
\item[(a)]
Compute all $ \pair{ \o_i , \o_j } = \sum_\ee \res_{\ee} ( \o_j \int \o_i ) $
in~$ M_1 $ using the results of 1(a).

\item[(b)]
Compute all $ \pair{ \o_i , \phidagA \o_j } = \sum_\ee \res_{\ee} ( \phidagA\o_j \int \o_i ) $
in~$ M_2 $ using the results of 1(e).

\item[(c)]
Compute~$ M = M_1^{-1}M_2 $
and~$ M' = \s^{l-1}(M)\times \s^{l-2}(M)\times \cdots \times \s(M)\times M $.

\item[(d)]
Compute $ P_1(T) = \det(1 - T M') $.
\end{enumerate}

\medskip

\noindent
{\bf Output}
The numerator $ P_1(T) $ of the zeta function of $ C_k $.
\end{algorithm}

Here Steps 1(a) and 2(a) compute~$ M_1 $ because of~\ref{serre-rigid-formula}
and Lemma~\ref{compatible-pairings}, and its determinant is in~$ R^* $ by Remark~\ref{detrem}. 
Step~1(d) is justified by Theorem~\eqref{locallift0}, so that
Steps~1(e) and~2(b) compute~$ M_2 $.
Then~$ M $ is the matrix of~$ \myphicr {\pa} \otimes \s $ because of~\eqref{robswish1},
so that Steps~2(c) and~(d) are just parts~(2) and~(3) of Method~\ref{zeta-method}.

We treat how to write down a lift~$ \phidagA $ of the~$ \pa $-th power map
based on the system~$ G(\S) = 0 $ in Section~\ref{global-frobenius}.
But as evident from the algorithm, we only need to compute
each resulting local expansion. This we do in Section~\ref{local-frobenius},
based on each of the systems~$ H(\S) = 0 $,
after preparatory material on expansions in Section~\ref{expansions}.
We give examples of the global Frobenius and its local
expansions in Section~\ref{global-examples} and~\ref{local-examples}.
Those sections come with estimates that are needed for finite
precision calculations. Together with the results in Section~\ref{cup-product-estimates},
we can then combine everyting into a version of the algorithm
with finite precision and finite expansions in Section~\ref{algorithm-estimates}.

\section{The global lift of Frobenius} \label{global-frobenius}

Let $R$, $\pi$, $k$, $\s$ and $ \pa $ be as in Notation~\ref{basicnot}.
In this section we explain our strategy for computing a lift of
Frobenius on the dagger algebras~$ \Rxdag $ and $ \Adag $ over $R$ as introduced in Section~\ref{sec:intro},
which was inspired by the work of Arabia~\cite{Arabia01}.
Even though we ultimately apply this
only for curves, given the current limitation of our cup product
method for computing the resulting action on the cohomology, the method
of lifting applies (and we describe it here) in greater generality,
for any~$ A $ as in Assumption~\ref{Aassumption}
below.

By and large, this method was already developed in the master thesis of
F.-R.~Escriva. Later we discovered that another approach,
but with a less transparent presentation,
is contained in the unpublished PhD thesis of R.~Gerkmann~\cite{Ger03}.
In this thesis a lift of Frobenius to~$ \Adag $ instead 
of one to~$ \Rxdag $ that induces the one on~$ \Adag $ is discussed, and the computations to find it
are based on contractions, not on Newton iteration (where the precision doubles
each time). Of course, carrying out the Newton iteration in $ \Adag $
can reduce the complexity due to the relations involved,
but we only want to compute this lift explicitly at the ends of
a curve, where we work with (infinite) Laurent series in one
variable (see Sections~\ref{expansions} and~\ref{local-frobenius}).

Given~$\Adag= R \la x_1,\ldots,x_n \ra^\dagger /(f_{r+1},\ldots,f_n)$,
we want to lift the~$\pa$-power endomorphism $ \phibar $ of $ \Abar $ to a $\s$-semilinear endomorphism $ \phidagA $ of
$ \Adag $, induced by an $ \s $-semilinear endomorphism $ \phidag $ of
$ R \la x_1,\ldots,x_n \ra^\dagger $,
and obtain estimates on the coefficients of the images of the~$ x_i $
under the latter endomorphism.
But for the sake of exposition we first explain our approach
in the simplest possible case of one equation in two variables over $\Zp$
with $ \s $ the identity.
Also, we at first ignore the issue of overconvergence,
so look for maps $ \phihat $ and $ \phihatA $ on the complete
algebras.

Suppose then that $f(x,y)$ in $ \Zp[x,y]$ is such that the reduction
$\ol{f}(x,y)$ defines an non-singular curve in $\A_{\F_p}^2$. Our goal is to
lift the $ p $-th power Frobenius morphism to a morphism
of the affine curve defined by $f$, viewed as a rigid analytic variety.
With $f_x$ and $f_y$ the two partial derivatives of $f$,
the non-singularity of the curve defined by $\ol{f}$ means that
there exist polynomials $\ol{P}_1$, $\ol{P}_2$ and $\ol{\Delta}$ in
$ \F_p[x,y]$
such that
\begin{equation*}
 \ol{f}_x \ol{P}_1 + \ol{f}_y \ol{P}_2 = 1 + \ol{f}  \, \ol{\Delta}
\,.
\end{equation*}
We arbitrarily lift $\ol{P}_1$, $\ol{P}_2$ and $\ol{\Delta}$ to polynomials
$P_1$, $P_2$ and $\Delta$ in $\Zp[x,y]$, so that the congruence
\begin{equation} \label{congeq1}
  f_x P_1 + f_y P_2  \equiv 1 + f \Delta 
\end{equation}
holds modulo $ p $. We now seek our lift of Frobenius on $ \Zp \la x , y \ra  $ of the form
\begin{equation*}
  \phihat(x,y) = (x^p + P_1(x^p,y^p) s , y^p + P_2(x^p,y^p) s )
\end{equation*}
where $ s $ in $ p\Zp \la x,y \ra$ is chosen to solve the equation
\begin{equation*}
  f( (x^p + P_1(x^p,y^p) S , y^p + P_2(x^p,y^p) S ) - f(x,y)^p - f(x,y)^p \Delta(x^p,y^p) S = 0
\end{equation*}
in the variable $S$.
Clearly, if $s$  is a solution, then $ \phihat(f) $ is
divisible by $f$ (and even by $f^p$), so that $ \phihat $ induces a homomorphism
$ \phihatA $ of the algebra $ \Ahat = \Zp\la x, y \ra / (f) $
to itself. Furthermore, since by assumption the coefficients of
$s$ are divisible by $p$, we see that $\phihat(x,y) \equiv (x^p,y^p) $
modulo $ p$, so $ \phihat $ and $ \phihatA $ are indeed lifts
of the respective $ p $-th power Frobenius.

The above equation in $S$ has coefficients in 
$ \Zp[x,y] \subseteq \Zp\la x,y \ra$, and has~$0$ as a solution modulo~$p$. Its derivative with respect
to $S$ at $S=0$ is
\begin{equation*}
  f_x(x^p,y^p) P_1(x^p,y^p) + f_y(x^p,y^p) P_2(x^p,y^p)  - f(x^p,y^p)
  \Delta(x^p,y^p)
\,,
\end{equation*}
which, in light of~\eqref{congeq1}, reduces to~1 modulo~$p$. The
existence and uniqueness of the solution in $ p \Zp\la x, y \ra $
are thus guaranteed by Hensel's lemma, and it
can be computed efficiently using Newton iteration, starting from the
approximate solution~$0$.

We shall consider the following simple (and for point counting
rather uninteresting) example as a running example at various
points in this paper.

\begin{example} \label{runningexample1}
Consider $ f(x,y) = x^2 - y^2 -1 $ in $ \Zp[x,y] $ with $ p \ne 2 $.
Then we have~$ \ol{f}_x \cdot \ol{2}^{-1} x + \ol{f}_y \cdot \ol{2}^{-1} y = \ol{1} + \ol{1} \cdot \ol{f} $
in $ \F_p[x,y] $. Now we write down
\begin{equation} \label{runexeq}
\begin{aligned}
   G(S)
& =
   f( x^p + 2^{-1} x^p S, y^p + 2^{-1} y^p S) - f(x,y)^p - f(x,y)^p S 
\\
& =  
    4^{-1} (x^{2p} - y^{2p}) S^2 + (x^{2p} - y^{2p} -  f(x,y)^p ) S +  f(x^p,y^p) - f(x,y)^p 
\end{aligned}
\end{equation}
in $ \Zp[x,y][S] $. We solve this for the unique solution $ S = s $
in $ p \Zp\la x, y \ra $. Then $ \phihat(x,y) = (x^p+2^{-1}x^p s, y^p+2^{-1}y^p s) $
defines an endomorphism of $ \Zp\la x,y \ra $ that descends
to an endomorphism $ \phihatA $ of $ \Zp\la x,y \ra / (f) $
and reduces to the Frobenius morphism $ \phibar(x,y) = (x^p,y^p) $ modulo $ p $.
\end{example}

We now describe the general case, while still ignoring
overconvergence for the moment. 
Recall the shorthand $ R[\x] $ of Notation~\ref{global-notation2}.
We shall write $\M^{a,b} $ and $\M^{a}$ for matrices of sizes~$a\times b$
and $a\times a$ respectively, and for given $ f_{r+1}, \dots, f_n $ in~$ R[\x] $
we let~$\Jac_f $ in~$ \M^{n-r,n}(R[\x])$ be the resulting Jacobian matrix.
We want to lift the $\pa$-power endomorphism $ \phibar $ of $ \Abar = k[\x] / (\ol{f_{r+1}}, \dots, \ol{f_n} ) $
to a $\s$-semilinear endomorphism $ \phihatA $ of~$ \Rx/(f_{r+1},\ldots,f_n) $, induced by a $ \s $-semilinear endomorphism
$ \phihat $ of $ \Rx $, if $ A = R[\x]/(f_{r+1},\ldots,f_n) $ satisfies the following assumption.
In particular, by Remark~\ref{smoothlcirem} below,
this will apply (with $ r = 1 $) to a suitable Zariski open part~$ \Spec(A) $ of~$ C/R $.
See Section~\ref{loc-plane-curves} for a practical approach using localisation in the plane.

\begin{assumption} \label{Aassumption}
In $ R[\x] $, for some $ r = 0, \dots, n-1 $
we are given $ f_{r+1}, \dots, f_n $,
such that 
\begin{equation*}
 A = R[\x] / (f_{r+1}, \dots, f_n) 
\,,
\end{equation*}
and the ideal of~$ \Abar $ generated by the determinants
of the $(n-r) \times (n-r)$ minors of~$ \Jac_{\ol f} $ is~$ \Abar $.
\end{assumption}

\begin{remark} \label{Abarnonzero}
Assumption~\ref{Aassumption} does not imply that $ \Abar \ne 0 $ 
(it fails, for example, for~$ A = \Zp[x,y] / (px-1) \simeq \Qp[y] $).
Also, it does not imply that~$ A \otimes_R K $ is smooth
over~$ K $ (take, e.g., $ A = \Zp[x,y] / ( x (px-1)^2 ) $).
We shall make such additional assumptions in later sections,
where appropriate.
\end{remark}

Assumption~\ref{Aassumption} implies that
\begin{equation} \label{congeq2}
    \Jac_f \times P \equiv \Id_{n-r}+\sum_{j=r+1}^n f_j \Delta_j \textup{ modulo }\pi
\end{equation}
for some~$ P $ in $ \M^{n,n-r}(R[\x]) $ and $ \Delta_{r+1},\ldots,\Delta_n $ in $ \M^{n-r}(R[\x]) $,
which can be obtained as follows.
For each $ I \subseteq \{1,\dots,n\} $ with $ n-r $ elements,
let~$ d_I $ be the determinant of the corresponding $ (n-r) \times (n-r) $
of $ \Jac_f $. Putting the rows of the adjugate matrix of this
minor into the rows indexed by $ I $ of a matrix in $ \M^{n,n-r}(R[x]) $
gives a matrix $ P_I $ with~$ \Jac_f \times P_I = d_I \Id_{n-r} $.
By Assumption~\ref{Aassumption} there
are $ p_I $ and $ c_{r+1}, \dots, c_n $ in $ R[\x] $ with
$ \sum_I p_I d_I = 1 + c_{r+1} f_{r+1} + \dots + c_n f_n $ modulo~$ \pi $,
so that~\eqref{congeq2} holds with~$ P = \sum_I p_I P_I $
and $ \Delta_j = c_j \Id_{n-r} $.
However, the more general shape in~\eqref{congeq2} allows for
other choices of $ P $ and the $ \Delta_j $, which may result
in better convergence of the resulting solution of~\eqref{globalG}
or its local equivalents at the ends in Theorem~\ref{locallift0} below
(see Propositions~\ref{itprop} and~\ref{hensel2} together with
their corollaries and remarks for estimates on
the convergence).

Let $\psi$ be the $\s$-semilinear endomorphism of $ \Rx $ that sends each
$ x_i $ to $ x_i^{\pa} $, so that it maps
an element $g(\x)$ to $g^\s(\psi(\x))$, where the superscript
$ \s $ means we apply $ \s $ to the coefficients.
We shall look for a $ \s $-semilinear endomorphism $\phihat$ of $ \Rx $
that is defined by its action on the column vector of variables $\x$ as
\begin{equation} \label{eq:phiact}
  \phihat(\x) = \psi(\x)+\psi(P)\sol
\,,
\end{equation}
where $\sol$ is a suitable column vector in $\pi \Rx^{n-r}$.
In order to describe~$ \sol $, we let~$ \S $ be the column vector $ (S_{r+1},\dots,S_n) $
for variables~$ S_{r+1}, \dots, S_n $, and we let
$ G(\S) $ be the column vector~$ (G_{r+1}(\S),\dots, G_n(\S)) $
with entries in $ R[\x] [\S] $ given by
\begin{equation} \label{globalG}
  G(\S) = f^\s (\psi(\x)+\psi(P)\S) - f^{\pa} -
 \sum_{j=r+1}^n f_j^{\pa} \psi(\Delta_j) \S  
\,,
\end{equation}
where $ f^\s = (f_{r+1}^\s,\dots,f_n^\s) $
and $ f^{\pa} = (f_{r+1}^{\pa},\dots,f_n^{\pa}) $ as column vectors.
We then want $ \sol $ to satisfy $G(\sol)=0$.
Because $ \ol{\s} $ is raising to the $ \pa $-th power on $ k $,
one finds in a way similar to the case of one equation in two variables discussed
before that:
\begin{itemize}
\item
$G(0) \equiv 0 $ modulo $ \pi $;

\item
$ \Jac_G(0)  = \Jac_{f^\s}(\psi(\x)) \times \psi(P) - \sum_{j=r+1}^n f_j^{\pa} \psi(\Delta_j) $,
and this is equivalent to $ \Id_{n-r} $ modulo $ \pi $.
\end{itemize}
Therefore~\eqref{globalG} has a unique solution $ \sol $ in $ \pi \Rx^{n-r} $ by Hensel's lemma,  and it
can be approximated effectively using Newton iteration:~starting
with~$ z_0 = (0,\dots,0) $ as solution of $ G(\S) \equiv 0 $ modulo $ \pi $,
one puts $ z_{i+1} = z_i - \Jac_G(z_i)^{-1} G(z_i) $ for $ i \ge 0 $.

It is then clear from the construction that the resulting endomorphism $ \phihat $
of $ \Rx $ is $ \s $-semilinear,
maps the ideal $ (f_{r+1},\dots,f_n) $
into~$ (f_{r+1}^{\pa}, \dots,f_n^{\pa}) \subseteq (f_{r+1},\dots,f_n) $,
and
reduces to the~$ \pa $-power map~$ \phibar : k[\x] \to k[\x] $.
So we proved the following.

\begin{theorem}\label{globallift0}
  Let $f_{r+1}, \dots, f_n $ in $ R[x_1,\ldots, x_n]$
  with $ 0 \le r \le n-1 $ be given, and suppose
  that $A=R[\x]/(f_{r+1},\ldots,f_n)$ satisfies
  Assumption~\ref{Aassumption}.
  Fix $ P $ as well as $ \Delta_{r+1}, \dots, \Delta_n $
  as in~\eqref{congeq2}, and let $G(\S)$ be given by~\eqref{globalG}.
  Then the following hold.
\begin{enumerate}
\item
The $ \s $-semilinear endomorphism $ \phihat $ of $ \Rx $ given by~\eqref{eq:phiact} with~$ \sol $ the
unique solution in~$ \pi \Rx^{n-1} $ of~$ G(\sol) = 0 $
lifts the $ \pa $-power map on $ k[\x] $.

\item
It maps the ideal $ (f_{r+1},\dots,f_n) $ into
$ (f_{r+1}^{\pa}, \dots,f_n^{\pa}) \subseteq (f_{r+1},\dots,f_n) $,
hence induces a $ \s $-semilinear endomorphism $ \phihatA $ on
$ R\la \x \ra / (f_{r+1},\dots,f_n) $ that lifts the $ \pa $-power endomorphism on $ k[\x]/ (\ol{f_{r+1}}, \dots, \ol{f_n} ) $.
\end{enumerate}
\end{theorem}

\begin{remark}
The fact that the solution $ \sol $ is unique once $ P $ and $ \Delta_{r+1} , \dots , \Delta_n $ have been fixed
is crucial to our approach, as we only compute the local expansion
of~$ \sol $ at the ends of $ C $ (see Sections~\ref{expansions}
and~\ref{local-frobenius})
and not~$ \sol $ itself.
This has the advantage of working with Laurent expansions in
one variable instead of power series in several.
\end{remark}

\begin{remark} \label{liftcurve}
In Theorem~\ref{globallift0} one can sometimes
prescribe that $ \phihat(x_i) = x_i^{\pa} $ for several~$ i $. Suppose,
for notational simplicity, that in
the discussion on how to obtain~\eqref{congeq2} one only needs $ I \subseteq \{m+1, \dots, n \} $
for some positive integer $ m \le r - 1 $, i.e., $ \Abar $ is
generated by the classes of the determinants of the minors involving only the
last $ n - m $ columns of $ \Jac_f $.
Then one can can arrange for~$ P $ to have its first~$ m $ rows
identically~$ 0 $, so that $ \phihat(x_i) = x_i^{\pa} $
for $ i = 1,\dots, m $ by~\eqref{eq:phiact}.
\end{remark}

In order to use the lift $ \phihat $ of the $ \pa $-th power Frobenius effectively, we need to know
that it preserves overconvergence,
i.e., that it induces an endomorphism~$ \phidag $ of $ \Rxdag \subseteq \Rx $ and, hence, also
an endomorphism~$ \phidagA $ of $ \Adag \subseteq \Ahat $.
Moreover, we need to have explicit estimates for the rate of overconvergence,
as well as for the speed of convergence for the Newton iteration in Theorem~\ref{globallift0}.

We prove such statements in greater generality in Proposition~\ref{itprop} and Corollary~\ref{estimate-cor} below.
We shall explain in Remark~\ref{Giscovered} how these results apply to~\eqref{globalG}.
In particular, we can solve~\eqref{globalG} using Newton iteration,
starting with~$ z_0 $ having all its entries equal to~0, and
obtain explicit estimates on the coefficients of the resulting
solution.

In order to state the results we introduce some notation.

\begin{definition} \label{Wab-def}
For any $ \a $ in $ (\realpos)^n $
and arbitrary $ \b $ in $ \R^n $, we let
\begin{equation*}
 \W{\a,\b} =  \bigl \{ \textstyle{\sum}_{D \ge 0} \, a_D \x^D \text{ in } \Rx \text{ with } v(a_D) \ge \a^{-1} \cdot (D - \b) \bigr \}
 \,,
\end{equation*}
and for any $ c \ge 0 $ we let $ \VV{c} = \{ \sum_{D \ge 0} \, a_D \x^D \text{ in } \Rx \text{ with } v(a_D) \ge c \} $.
\end{definition}

Even though~$ \b $ is not determined uniquely by the number~$ \a^{-1} \cdot \b $
if $ n \ge 2 $, we prefer to write this number
in this way as it makes clear how it changes when scaling~$ \a $. Of course, one could take~$ \b = b (1,\dots,1) $ for a uniquely
determined real number~$ b $. Also, the reason to use $ \a^{-1} $
in the definition of but not the notation for~$ \W{\a, \b} $
is that this simplifies notation later on.

Each $ \W{\a, \b} \cap \VV{c} $ is $ p $-adically closed because convergence in~$ \Rx $ is in the coefficient of each monomial $ \x^D $.
We have~$ \W{\a, \b} \subseteq \W{\a, \tilde\b} $ if $ \a^{-1} \cdot \b \le \a^{-1} \cdot \tilde \b $,
as well as~$ \W{\a, \b} \subseteq \W{\tilde\a, \b} $ if $ \tilde\a = (1+\l) \a $ with $ \l \ge 0 $
because all elements considered are in~$ \Rx $.

The product in $ \Rx $ maps $ \W{\a,\b_1} \times \W{\a,\b_2} $ to $ \W{\a,\b_1 + \b_2} $,
and the units of the ring~$ \W{\a,0} $ consists of its elements with constant term in~$ R^* $.

For our estimates, we need some statements about inclusions for~$ \W{\a,\b} \cap \VV{c} $ with $ c > 0 $.
For this, we take~$ \sum_{D \ge 0} a_D \x^D $ in this intersection,
and let $ \tilde\b $ be arbitrary.
If $ \a^{-1} \cdot ( \b - \tilde\b) \ge 0 $, then
$ c^{-1} \a^{-1} \cdot (\b - \tilde\b) v(a_D) \ge \a^{-1} \cdot (\b - \tilde\b) $,
so that $ (1+\l) v(a_D) \ge \a^{-1} \cdot (D - \tilde\b) $ with~$ \l = c^{-1} \a^{-1} \cdot (\b - \tilde\b) $.
Therefore we have~$ \W{\a, \b} \cap \VV{c} \subseteq \W{\tilde \a, \tilde \b} $
with $ \tilde\a  = (1+\l) \a $.
(Note that one can either compute $ \l $ from $ \tilde\b $, or
find a~$ \tilde\b $ that matches a given $ \l \ge 0 $.)
If~$ \a^{-1} \cdot (\b - \tilde\b) \le 0 $ then one has $ \W{\a,\b} \subseteq \W{\a,\tilde\b} $
as discussed above. Hence, for any~$ \tilde\b $ we have
\begin{equation} \label{alphabetac}
 \W{\a, \b} \cap \VV{c} \subseteq \W{\tilde\a, \tilde\b}
\end{equation}
with $ \tilde \a - \a = \l \a $ when~$ \l \ge c^{-1} \maxp\{ \a^{-1} \cdot (\b - \tilde\b) \} $,
i.e., when $ \l \ge 0 $ and $ \a^{-1} \cdot \tilde\b \ge \a^{-1} \cdot \b - \l c $.
In particular, we have~$ \W{\a, \b} \cap \VV{c} \subseteq \W{\tilde \a, 0 } $
for $ \tilde\a = \a + \l \a $ with $ \l = \maxp\{  c^{-1} \a^{-1} \cdot \b \} $.
We observe that in all cases~$ \l \a $ is invariant under scaling
of~$ \a $. (It is in fact the identity~$ \tilde\a = \a +  \l \a  $
and this behaviour of $ \l \a $ that prompted us to
use $ \a^{-1} $ instead of $ \a $ in the definition of~$ \W{\a, \b} $.)

The following estimate for Newton iteration is not beautiful,
but can be sharp (see, e.g., Example~\ref{runningexample2}).
In fact, the estimate is essentially forced if one formulates it
in terms of the~$ \W{\a,\b} $, so one can hope it is sharp `in general'.
Corollary~\ref{estimate-cor} provides a simpler (but worse) estimate.

\begin{prop} \label{itprop}
Let $ G(\S) $ be a column vector~$ (G_{r+1}(\S),\dots, G_n(\S)) $
with entries in $ R[\x] [\S] $.
Suppose that~$ \a $ in $ (\realpos)^n $,
$ b > 0 $, $ c > 0 $, $ \b $, $ \cc $ in $ \R^n $, and $ z $ in~$ \Rx^{n-r} $, are
such that
\begin{enumerate}
\item
  $ z $ has entries in $ \W{\a,\b} \cap  \VV{b} $,
  
\item
  each $ G_j(z) $ is in  $ \W{\a, \cc} \cap \VV{c} $,

\item
  $ \Jac_G(z) $ has entries in $ \W{\a,0} $ and determinant in $ \W{\a, 0}^* $,

\item
 $ \a^{-1} \cdot \cc \le \a^{-1} \cdot \b $.
\end{enumerate}
Let $ \e = - \Jac_G(z)^{-1} G(z) $ and $ z' = z + \e $. 
Then~(1) through~(4)  hold for $ z' $ instead
of~$ z $ if we replace~$ b $ with $ b' = \min\{b,c\} $,
$ c $ with $ c' = 2 c $, $ \a $ with $ \a' = \a + \l \a $,
$ \b $ with~$ \b' $, and~$ \cc $ with $ \cc' $, where
$ \l $, $ \b' $ and $ \cc' $ are determined as follows.

Write each~$ G_j(\S) = \sum_{l=0}^N \sum_{|J| = l} G_{j,J}(\x) \S^J $
for some $ N \ge 2 $, and assume that each~$ G_{r+i,J}(\x) $ with~$ |J| \ge 2 $ is in $ \W{\a, \d_{r+i,J}} $
for some $ \d_{r+i, J} $ in $ \R^n $.
For~$ j = 1,\dots,n-r $, let~$ J^{(j)} $ be obtained from~$ J $
by subtracting~$ 1 $ in its $ j $-th position.
Then
\begin{equation} \label{lambdadef}
 \l = \max_{i=1, \dots, n-r \atop j=1, \dots, n-r} \maxp_{ |J| = 2,\dots,N \atop 0 < K \le J^{(j)} }
              \left\{ \a^{-1} \cdot \frac {\d_{r+i,J} + (|J|-|K|-1) \b + |K| \cc}{(|J|-|K|-1)b + |K| c }  \right\}
 \,,
\end{equation}
we choose~$ \b' $ such that~$ \a^{-1} \cdot \b' = \a^{-1} \cdot \b - \l b' $, and
$ \cc' $ such that~$ \a^{-1} \cdot \cc' $ equals
\begin{equation*}
 \max_{ { { i=1, \dots, n-r } \atop { |J| = 2,\dots,N } } \atop { { 0 \le K \le J } \atop { |K|\ge2 } } }
\bigl\{
 \a^{-1} \cdot \bigl( \d_{r+i,J} + (|J|-|K|) \b + |K| \cc \bigr) - \l \bigl( (|J|-|K|)b + |K| c \bigr)
\bigr\}
\,.
\end{equation*}
Moreover, if $ b' = b $ then further Newton iteration does not change~$ \a' $, $ \b' $ and $ b' $.
\end{prop}

We postpone the proof of this proposition as it is somewhat notational, and first derive a more explicit estimate and make some remarks.

\begin{corollary} \label{estimate-cor}
If there exists a $ z $ as in Proposition~\ref{itprop},
then there is a unique solution $ \sol $ of $ G(\S) = 0 $ with
entries in $ \pi \Rxdag $.
It can be computed using Newton iteration with~$ z_0 = z $, and
$ z_{m+1} = z_m + \e_m $ for $ m \ge 0 $ with $ \e_m = - \Jac_G(z_m)^{-1} G(z_m) $.

Then $ z_2 = z'' $ comes with~$ \a'' $, $ \b'' $ and $ b'' $
that no longer change under further iteration, and
all $ z_m $ with $ m \ge 2 $, as well as~$ \sol $, have entries in~$ \W{\a'', \b''} $.

As a simple estimate, if we let
\begin{equation*}
 M  =  \maxp_{ i = 1, \dots, n-r \atop  |J| = 2,\dots,N } \left\{ \a^{-1} \cdot ( \d_{r+i,J} + (|J|-1) \b )  \right\}
\,,
\end{equation*}
then all $ z_m $, as well as~$ \sol $, have entries in~$ \W{ (1+\mu) \a, \b } $
for $ \mu = 3 c^{-1} M / 2 $.
\end{corollary}

\begin{proof}
The existence and  uniqueness of a solution with entries in $ \pi \Rx $ and that
it can be computed using Newton iteration from $ z $ are standard.
That those entries are in~$ \pi \Rxdag $ follows from the given
estimates, which we shall now prove.

Applying the proposition to~$ z' = z_1 $ instead of~$ z $ we
find~$ b'' = \min\{b',c'\} = b' $, so that $ \a'' $, $ \b'' $
and $ b'' $ are stable under further iteration.

Replacing $ \a^{-1} \cdot \cc $ with $ \a^{-1} \cdot \b $ 
in~\eqref{lambdadef}, and replacing the denominator with~$ c $, shows that~$ \l \le M /c $.
Therefore~$ \a' = \a + \l \a $ with~$ \l \le M /c $.
We similarly have $ \a'' = \a' + \l' \a' = (1 + \l ) \a + \l' \a' $ with~$ \l' \le M' /c' $.
Using $ \a^{-1} \cdot \b' \le \a^{-1} \cdot \b $, and scaling~$ \a' $ to~$ \a $,
one sees that $ \l' \a' = \tilde\l \a $ with $ \tilde\l \le M/c' $.
Because $ c' = 2 c $ we see that $ \a'' = (1+\mu) $ with $ \mu \le 3 c^{-1} M / 2 $.
Since $ \a'' $and $ \b'' $ do not change under further iteration, and $ \a^{-1} \cdot \b'' \le  \a^{-1} \cdot \b' \le  \a^{-1} \cdot \b $,
all $ z_m $ are in $ \W{ \a'', \b } \subseteq \W{ (1 + \mu) \a , \b } $,
and the same holds for~$ \sol $ by completeness.
\end{proof}

Because Corollary~\ref{estimate-cor} applies to~\eqref{globalG}
for~$ z = 0 $, combining it with Theorem~\ref{globallift0} now gives one of our main goals of this section.

\begin{corollary} \label{phidagA-phihatA}
The map~$ \phihatA $ in Theorem~\ref{globallift0}(2), induced
by\eqref{eq:phiact}, gives a~$ \s $-semilinear
endomorphism~$ \phidagA $ of~$ \Adag $ that reduces to the~$ \pa $-th
power map on~$ \Abar $.
The rate of convergence of the underlying endomorphism~$ \phihat $
of~$ R\la \x \ra^\dagger $ can be explicitly estimated from the
system~\eqref{globalG} using Proposition~\ref{itprop}
and Corollary~\ref{estimate-cor}.
\end{corollary}

\begin{remark} \label{longremark}
(1)
If we replace~$ b $ with $ \min\{ b , c \}$ in the assumptions
of~Proposition~\ref{itprop} then we may change~$ \l $, hence also~$ \a' $
and~$ \b' $.  But~$ b' = b $ now, so that these~$ \a' $ and~$ \b' $ are already stable.
We did not compare them with the stable values~$ \a'' $ and~$ \b'' $
obtained in Corollary~\ref{estimate-cor} from the original~$ b $.

(2)
The formula in~\eqref{lambdadef} is natural for the proof of
Proposition~\ref{itprop} but can be simplified. For example, using $ j $
and $ J^{(j)} $ is redundant: one could simply impose $ 0 < K < J $ instead. Also, for fixed $ i $ and $ J $, the expression
involved is a fractional linear function in $ |K| $ for which
the maximum will be attained for either $ |K| = |J| - 1 $ or $ |K| = 1 $.
Similar remarks apply to the condition defining $ \a^{-1} \cdot \cc' $.

(3)
For $ z = 0 $ one can ignore all $ K \ne J^{(j)} $ in~\eqref{lambdadef}
and all~$ K \ne J $ in the condition defining~$ \a^{-1} \cdot \cc' $.
Of course, this does not apply to any consequences based on
the full definition of $ \l $, but one can compute the resulting~$ z' $,
$ \a' $, etc., and then apply the proposition again to those.
\end{remark}

\begin{proof}[Proof of Proposition~\ref{itprop}]
We first prove part~(3) for $ z' $ instead of~$ z $. 
From its definition we see~$ \e $ has entries in $\W{\a,\cc} \cap \VV{c} $.
Expanding with $ z' = z + \e $ gives
\begin{alignat*}{1}
  G_j(z') & = \sum_{l=2}^N \sum_{|J| = l} \sum_{0 \le K \le J \atop |K|\ge 2} \binom{J}{K} G_{j,J}(\x) z^{J-K}\e^K
\\
\Jac_G(z')_{i,j}
  & =
   \Jac_G(z)_{i,j}
  +
   \sum_{l=2}^N \sum_{|J|=l}\sum_{0 \le K \le J^{(j)} \atop |K|\ge1} J_j \binom {J^{(j)}} K G_{r+i,J}(\x) z^{J^{(j)}-K} \e^K
\end{alignat*}
where $ \binom JK = \binom{j_{r+1}}{k_{r+1}} \dots \binom{j_n}{k_n} $
for $ J = (j_{r+1},\dots,j_n)$ and $K = (k_{r+1},\dots,k_n)$,
we use multi-index notation in the exponents of $ z $ and~$ \e $,
and $ J_j $ is the $ j $-th coordinate of~$ J $.
Here every $ G_{r+i,J}(\x) z^{J^{(j)}-K} \e^K $ is in
\begin{equation*}
\W{\a, \d_{r+i,J} + (|J|-|K|-1) \b + |K| \cc }  \cap \VV{(|J|-|K|-1)b + |K|c}
\,.
\end{equation*}
Applying~\eqref{alphabetac} to this and using the definition
of~$ \l $, we see that $ \Jac_G(z') $ has entries in~$ \W{\a', 0} $.
Its determinant is a unit because it is congruent modulo~$ \pi $
to that of $ \Jac_G(z) $ as~$ \e $ is in $ \VV{c} $.

For part~(2), we note that, similarly, in the expansion of $ G_j(z') $, 
every~$ z^{J-K} \e^K $ is in
$ \W{\a, \d_{r+i,J} + (|J|-|K|) \b + |K| \cc }  \cap \VV{(|J|-|K|)b + |K|c} $.
This is in $ \VV{c'} $ because $ |K| \ge 2 $.
One sees the intersection is also in $ \W{\a', \cc'} $
by rewriting the condition for the inclusion given by~\eqref{alphabetac} with
the current $ \l \ge 0 $, and comparing with 
the definition for~$ \a^{-1} \cdot \cc' $.

For part~(1), we note that we already know that~$ \cc $ has entries in $ \W{\a,\cc} \cap \VV{c} $.
As~$ z $ has entries in $ \W{\a,\b} \cap \VV{b} $, $ \a^{-1} \cdot \cc \le \a^{-1} \cdot \b $,
and~$ b' = \min\{b,c\} $, we have that $ z' = z + \e $ has entries in~$ \W{\a, \b } \cap  \VV{b'} $.
From~\eqref{alphabetac} we see that this intersection is contained in $ \W{\a', \b' } $
because~$ \a^{-1} \cdot \b' = \a^{-1} \cdot \b - \l b' $ by
definition.

We now show that~$ \a^{-1} \cdot \cc' \le \a^{-1} \cdot \cc - \l c $.
From~\eqref{lambdadef} we have
\begin{equation*}
 \a^{-1} \cdot \bigl( \d_{r+i,J} + (|J|-|K|-1) \b + |K| \cc  \bigr) \le \l \bigl( (|J|-|K|-1)b + |K| c \bigr)
\end{equation*}
for all $ i = 1, \dots, n-r $, $ J \ge 0 $ with $ |J| =2,\dots,N $ and $ K >0 $
with $ K \le J^{(j)} $, where~$ j $ runs from $ 1 $ to $ n-r $.
But for $ K \le J^{(j)} $ with $ K > 0 $ we can
add~$ 1 $ to $ K $ in its~$ j $-th position
and obtain $ K' $ with $ |K'| = |K| + 1 \ge 2 $
and $ 0 \le K' \le J $. Conversely, any such~$ K' $
is obtained this way for some $ K $ and~$ j $.
So we can rewrite the above as
\begin{equation*}
\a^{-1} \cdot \bigl( \d_{r+i,J} + (|J|-|K'|) \b + |K'| \cc  \bigr) \le \l \bigl( (|J|-|K'|)b + |K'| c \bigr)
  + \a^{-1} \cdot \cc - \l c 
\end{equation*}
for all $ i = 1, \dots, n-r $, $ J \ge 0 $ with $ |J| =2,\dots,N $ and $ 0 \le K' \le J $
with~$ { |K'| \ge 2 } $.
Comparing this with our condition on~$ \a^{-1} \cdot \cc' $,
we see that~$ \a^{-1} \cdot \cc' \le \a^{-1} \cdot \cc - \l c $.
Because~$ \a^{-1} \cdot \cc \le \a^{-1} \cdot \b $ and $ b' \le c $,
we also have~$ \a^{-1} \cdot \cc' \le \a^{-1} \cdot \b' $.
Scaling $ \a $ to~$ \a' $ then proves part~(4).

Finally, using primes to indicate further Newton iteration,
we show that~$ \l' = 0 $ if~$ b' = b $, so that~$ \a'' = \a' $
and~$ \b'' = \b' $.
For this it is enough to show that the expressions in~\eqref{lambdadef}
(but for $ \l' $) are non-positive.
We scale from~$ \a' $ to~$ \a $ and consider
\begin{alignat*}{1}
& \a^{-1} \cdot \bigl( \d_{r+i,J} + (|J|-|K|-1) \b' + |K| \cc' \bigr)
\\
\le \,\,
 & 
\a^{-1} \cdot \bigl( \d_{r+i,J} + (|J|-|K|-1) \b + |K| \cc \bigr)
- \l \bigl( (|J|-|K|-1) b' + |K| c \bigr)
\,,
\end{alignat*}
where we used that~$ \a^{-1} \cdot \b' = \a^{-1} \cdot \b - \l b' $ 
and~$ \a^{-1} \cdot \cc' \le \a^{-1} \cdot \cc - \l c $.
Using that~$ b' = b $ and comparing with~\eqref{lambdadef} 
we see this last expression is non-positive, showing $ \l ' = 0 $
as desired.
Further iteration does not change $ b' $ because $ c' $ doubles
each time, hence the above applies at every step, so
both~$ \a' $ and $ \b' $ remain unchanged.
\end{proof}

\begin{remark} \label{finiteprecisionremark}
With our assumption on $ \Jac_G(S) $ in Proposition~\ref{itprop},
if $ z $ and $ \tilde z $ have entries in $ \pi \Rx $ and~$ c > 0 $, then $ \tilde z - z $ has entries in $ \VV{c} $
if and only if this holds for~$ G(\tilde z) - G( z) $.
In particular,~$ \sol - z $ has such entries if and only
if $ G(z) $ does.

Suppose that we want to compute $ \sol $ up to precision~$ \pi^M $, i.e., we want
to know its coefficients in $ R_M $ with $ R_M = R / (\pi^M) $,
which corresponds to knowing its class in~$ \Rx / \VV{c} $ for~$ c = M v(\pi) = M / e $.
Then we can start with $ z_0 = 0 $ and perform Newton iteration
where, at every step, Corollary~\ref{estimate-cor} can be used
to explicitly truncate the power series expansion of~$ \Jac_G(z_m)^{-1} $
to an element of $ R_M[\x] $. As the precision doubles at every
step, one needs at most $ \lceil \log_2(M) \rceil $ iterations,
and one can even use lower precision than~$ \pi^M $ at intermediate
steps.

This method also applies if~$ G(\S) $ itself is known only up
to precision~$ \pi^M $.
\end{remark}

\begin{remark} \label{Giscovered}
For any $ z $ in $ \pi R[\x]^{n-r} $, 
Proposition~\ref{itprop}, Corollary~\ref{estimate-cor} and Remarks~\ref{longremark}
and~\ref{finiteprecisionremark}
apply to~\eqref{globalG} with suitable $ \a $ in $ \realpos^n $.
In fact, because all entries of any given element of~$ \realpos^n $ are positive,
they will apply with~$ \a $ equal to a suitable scalar multiple of that
element.
Using this one can obtain various estimates on the Newton iterates
and~$ \sol $ that are based on the monomials occurring in the~$ G_j(\S) $.
\end{remark}

We now illustrate the estimates using our running example.

\begin{example} \label{runningexample2}
The one equation $ 0 = G(S) = G_2 S^2 + G_1 S + G_0 $
in Example~\ref{runningexample1}~has
\begin{alignat*}{1}
   G_2 & = 4^{-1} (x^{2p} - y^{2p})                  \\
   G_1 & = x^{2p} - y^{2p} -  (x^2 - y^2 -1)^p       \\
   G_0 & = (x^{2p}  - y^{2p} - 1) - (x^2 - y^2 -1)^p 
\,.
\end{alignat*}
We apply Proposition~\ref{itprop} with~$ z $ in $ \pi R $.
We can use~$ \a = (2p, 2p) $,
$ \b = \cc = (0, 0)  $, and $ b = c = 1 $.
With~$ \d_{2, 2} = (p, p) $ we find~$ \l = 1 $ and
\begin{equation*}
 (\a', \b', \cc', b', c') =  ( (4p, 4p), (-p, -p) , (-p, -p), 1, 2)
 \,.
\end{equation*}
Iterating now doubles the last entry each time, the others being
stable. So the solution~$  s = \sum_{i,j \ge 0} a_{i,j} x^i y^j $
is in $ \W{(4p, 4p), (-p,-p)} $, and~$ v(a_{i,j}) \ge \frac{i+j}{4p} + \frac12 $.
By contrast, Corollary~\ref{estimate-cor} leads to $ M = 1 $
and the estimate~$ v(a_{i,j}) \ge \frac{i+j}{5p}$.

The first estimate is sharp.
For this, note
$ s = (2 G_2)^{-1} (\sqrt{G_1^2 - 4 G_0 G_2} - \sqrt{G_1^2} ) $ 
with the roots in~$ 1 + p \Zp\la x, y \ra $.
As $ G_1^2 = 1 + p A $ and $ - 4 G_0 G_2 = p (x^{2p} - y^{2p}) B $ for~$ A $ and $ B $ in $ \Zp[x,y] $,
expanding the roots gives
\begin{alignat*}{1}
s & =
2 (x^{2p} - y^{2p})^{-1} \sum_{m \ge 1} p^m \binom {\frac12} m [(A + (x^{2p} - y^{2p}) B)^m - A^m ]
\\
& =
2 B \sum_{m \ge 1} p^m  \binom {\frac12} m  \sum_{l=0}^{m-1} \binom m l A^l ( (x^{2p} - y^{2p}) B)^{m-l-1} 
\,.
\end{alignat*}
The highest degree terms of $ A $ and $ B $ are
$ p^{-1} (x^{2p} - y^{2p} - (x^2 - y^2)^p )^2 = p x^{4p-4} y^4 + \dots $ and
$ p^{-1} ( (x^2 - y^2)^p - (x^{2p} - y^{2p}) ) = - x^{2p-2} y^2 + \dots $,
hence the summand for $ m $ has degree at most $ 4p m - 2p $.
In fact, the monomial in it of that degree with the highest power of~$ x $ is
\begin{equation*}
 (-1)^{m-1} 2 \binom {\frac12} m p^m  x^{(m-1)(4p-2) + 2p} y^{2(m-1)}
 \,.
\end{equation*}
This does not occur for lower $ m $, and if $ \binom {\frac12} m = (-1)^m (1 - 2m)^{-1} 4^{-m} \binom {2m} m $ is in~$ \Zp^* $,
then it occurs in $ s $ with a coefficient of valuation~$ m $.
This is the case, e.g., for all~$ m = p^k $ with $ k \ge 1 $ because
then $ \binom {2m} m \equiv 2 $ modulo~$ p $.
Keeping track of the precise grades we find~$ v(a_{i,j}) \ge \frac{i+j}{4p}+ \frac12 $
with equality for infinitely many $ a_{i,j} $.
\end{example}

\begin{example} \label{runningexample5}
We provide another example to be compared in Section~\ref{expansions}.
Consider the curve defined by $xy - 1$ in $\Zp[x,y]$ with $p \not= 2$.
The system $G(S)$ is then $\frac{1}{4}x^py^pS^2 + (x^py^p - (xy - 1)^p)S - (xy - 1)^p + x^py^p - 1$.
So that $G_2 = \frac{1}{4}x^py^p$, $G_1 = x^py^p - (xy - 1)^p = \sum_{k = 1}^p \binom{p}{k} (-1)^kx^ky^k$ and $G_0 = \sum_{k=1}^{p-1} \binom{p}{k} x^ky^k(-1)^k$.
Let us carry out the Newton iteration.
Start with $\a^{-1} = (\frac{1}{p-1}, \frac{1}{p-1}), \b = \gamma = (0,0)$, $b = c = 1$,
and $\d_{2,2} = (p, p)$ so that $\lambda = \frac{p}{p-1}$. Therefore, at the next stage we have
\begin{equation*}
\textstyle (\a', \b', \gamma', b', c') = \( (2p-1, 2p-1), (-\frac{p}{2}, -\frac{p}{2}), (-\frac{p}{2}, -\frac{p}{2}), 1, 2 \)
.
\end{equation*}
Thus the final solution $s = \sum_{i,j\ge 0} a_{i,j}x^iy^j$ lies in $W((2p-1, 2p-1), (\frac{p}{2}, \frac{p}{2}))$,
and $v(a_{i,j}) \ge \frac{i + j - p}{2p - 1}$.
\end{example}

We conclude this section by showing that our theorems apply
to suitable open parts of smooth, Noetherian schemes over~$ R $.

\begin{remark} \label{smoothlcirem}
Suppose that $ Y $ is a smooth, Noetherian scheme $ Y $ over $ R $ of
relative dimension $ r \ge 0 $. Then
there exists a Zariski open affine part of the form~$ \Spec(A) $ with
$ A $ satisfying Assumption~\ref{Aassumption} for some $ f_{r+1} , \dots , f_n $.  Moreover,
there exist matrices $ P $ in $ \M^{n,n-r}(R[\x]) $ and $ \Delta_{r+1},\ldots,\Delta_n$
in $ \M^{n-r}(R[\x]) $ such that
\begin{equation*}
    \Jac_f \times P = \Id_{n-r}+\sum_{j=r+1}^n f_j \Delta_j 
\end{equation*}
in $ \M^{n-r} (R[\x]) $.
(This is stronger than the consequence~\eqref{congeq2} of Assumption~\ref{Aassumption}.)
Moreover, if~$ Y_k $ is non-empty, then we can take~$ A $ such
that~$ \Abar $ is non-zero, thus also satisfying the condition
in Remark~\ref{Abarnonzero}.

In order to see this, let $ y $ be the generic point of the special fibre $ Y_k $
if this fibre is non-empty, and of $ Y_K = Y $ otherwise.
According to~\cite[Example~3.18]{Liu}, there is a Zariski open
neighbourhood of~$ y $ that is isomorphic to~$ R[x_1, \dots, x_n] / (f_{r+1}, \dots, f_n) $,
and by
\cite[\href{https://stacks.math.columbia.edu/tag/02H4}{Tag 02H4}]{stacks-project}
$ R[x_1, \dots, x_n] / (f_{r+1}, \dots, f_n) $ is a formally
smooth~$ R $-algebra. Then
according to \cite[Corollaire~(20.5.14)]{EGA4}, the sequence
\begin{equation*}
0 \to J/J^2 \xrightarrow{\dd} \Omega^1_{R[\x]/R}\otimes_{R[\x]} A \to \Omega^1_{A/R} \to 0
\end{equation*}
of~$ A $-modules is split exact.
Therefore, $ A\cdot\dd f_{r+1} \oplus \dots \oplus A\cdot\dd f_n \cong A^{n-r} $ 
is a direct summand of $ \Omega^1_{R[\x]/R}\otimes_{R[\x]} A\cong A^n $, and~$ \Jac_f $
has a right inverse $ P $ in~$ \M^{n,n-r}(A) $.
\end{remark}

\section{Examples of the global Frobenius} \label{global-examples}

In this section we make the construction of $ \phihat $ and~$ \phihatA $ in
Theorem~\ref{globallift0} and the resulting estimates from Proposition~\ref{itprop}
and Corollary~\ref{estimate-cor} more explicit for plane curves and their localisations,
for which we can take $ r=1 $ and $ n=2 $ or~3.

\subsection{Smooth plane curves} \label{planecurvesection}
Let us first treat the case of a smooth curve in $ R[x,y] $
defined by~$ f(x,y) $.
There then exist $ P_1 $, $ P_2 $ and $ \Delta $ in
$ R[x,y] $ such that from~\eqref{congeq2} we
have~$ \ol{P_1} \, \ol{f_x} + \ol{P_2} \, \ol{f_y} = 1 + \ol{f} \, \ol{\Delta} $
in $ k[x,y] $. Now~\eqref{globalG} becomes
\begin{alignat*}{1}
 G(S)
& =
f^\s \bigl( x^{\pa} +  \psi(P_1) S, y^{\pa} + \psi(P_2) S \bigr) - f^{\pa} - f^{\pa} \psi(\Delta) S 
\\
& =
     \psi(f) - f^{\pa}
   + \(f^\s_{x}\bigl(x^{\pa}, y^{\pa}\bigr) \psi(P_1) 
                   + f^\s_{y}\bigl(x^{\pa}, y^{\pa}\bigr)\psi(P_2) - f^{\pa} \psi(\Delta) \) S
\\
& \phantom{\,=\,}   + (\cdots) S^2 + \cdots
\,.
\end{alignat*}
With $ s $ the unique solution in $ \pi R\la x, y \ra ^\dagger $
of $ G(S) = 0 $,
the map~$ \phihat $ from $ R\la x, y \ra^\dagger $ to itself is given
by mapping $x$ to $x^{\pa} + P_1^\s(x^{\pa} , y^{\pa}) s $
and $y$ to $y^{\pa} + P_2^\s(x^{\pa} , y^{\pa}) s $, and this
induces the map $ \phihatA $ on~$\Adag = R\la x , y \ra^\dagger / (f)$.
If~$ s = \sum_{i,j\ge 0} a_{i,j} x^i y^j $ then we can estimate
$ v( a_{i,j} ) $ by applying Proposition~\ref{itprop} and/or
Corollary~\ref{estimate-cor} to $ G(S) $, starting with
$ z = 0 $. Note that we choose various~$ \a $, etc., based on the coefficients
in $ G(S) $.

\begin{example} \label{ellipticexample}
Consider the elliptic curve over $\Zp$ defined by $f(x,y) = y^2-x^3-1$
for~$ p \ne 2,3 $. Here
$ \frac13x f_x + \frac12 y f_y = 1 + f$ already in~$ \Zp[x,y] $.
Let us take $ \pa = p $.
As~$ \s $ is the identity,
we have to find the unique solution $ s $ in $ p\Zp\la x,y \ra^\dagger $
of~$ G(S) = 0 $ for
\begin{alignat*}{1}
 G(S)
& =
 f\left( x^p + \tfrac13 x^p S, y^p + \tfrac12 y^p S \right) - f^p-f^p S 
\\
& =
 -\tfrac{1}{27} x^{3p} S^3 + \left(\tfrac14 y^{2p} - \tfrac13 x^{3p}\right)
   S^2 + (1+f(x^p,y^p) - f^p) S + f(x^p,y^p)-f^p
\,.
\end{alignat*}
The map of $ \Zp \la x,y\ra^\dagger/(y^2-x^3-1) $ to
itself is induced by mapping $ x $  
to $ x^p + \frac13 x^p s $ and~$ y $ to~$ y^p + \frac12 y^p s $.
Starting in Proposition~\ref{itprop} with~$ z = 0 $, we can use~$ \a = (3p, 2p) $, $ \b = \cc = (0, 0) $, 
$ b = c = 1 $, and use~$ \d_{2,2} = \d_{2,3} =  \d $ for $ \d = \frac12 \a $
because $ \a^{-1} \cdot \d  = 1 $. Then~$ b' = 1 $, $ c' = 2 $,
$ \l = 1 $, $ \a' =  2 \a = (6p, 4p) $,
and~$ \a^{-1} \cdot \b' = \a^{-1} \cdot \b - \l b' = -1 $.
Then~$ \a' $ and $ \b' $ are stable because~$ b' = b $, and we
find~$ s = \sum_{i,j\ge0} a_{i,j} x^i y^j $ with
$ v(a_{i,j}) \ge \a'^{-1} \cdot (i, j) - \a'^{-1} \cdot \b' =  \frac{2i+3j}{12 p} + \frac12 $.
\end{example}

\subsection{Smooth localisations of plane curves} \label{loc-plane-curves}
Let $ f(x,y) $ define a (possibly singular) irreducible, reduced plane curve, and
suppose all of its singularities are contained in the curve defined by%
\footnote{The genus of the curve plays no role here, so we allow~$ g $ as notation for a polynomial.}~$ g $ in $ R[x,y] $
with non-zero reduction modulo~$ \pi $.
Then~$ g $ lies in the radical of the ideal $ (f,f_x,f_y) $ in
$ R[x,y] $. Replacing~$ g $ with a suitable power and subtracting
an element in $ (f) $ if necessary, we may assume
that~$ g = a f_x + b f_y $ for $ a $ and $ b $ in $ R[x,y] $.
Conversely, the curve defined by any such non-zero~$ g $ contains
the singularities of the curve defined by~$ f $.
We then consider the localisation $ A = R[x,y]/(f,zg - 1) $,
which satisfies Assumption~\ref{Aassumption}. Indeed, we have
$ \Jac 
=
\begin{pmatrix}
  f_{x} & f_{y} & 0       \\
  z g_x & z g_y & g \\
\end{pmatrix}
$,
so we can take
$ P =
\begin{pmatrix}
    za                          &  0  \\
    zb				&  0  \\
  - \bigl(ag_x + bg_y\bigr) z^3 &  z  \\
\end{pmatrix}
$,
$ \Delta_2 = 0 $
and
$ \Delta_3 = 
\begin{pmatrix}
   1                           & 0  \\
 - \bigl(ag_x + bg_y\bigr) z^2 & 1 \\
\end{pmatrix}
$.
With $ \x = (x,  y, z) $, the system~\eqref{globalG} corresponds to
\begin{alignat*}{1}
   G_2(S_2, S_3) 
   & = f^\s\(\psi(\x) + \psi(P)\S\) - f^{\pa} - \(zg-1\)^{\pa} S_2\\
& =
	f^\s\(\psi(x) + \psi(za)S_2, \psi(y) + \psi(zb)S_2\)
   - f^{\pa} 
   - \(z g - 1 \)^{\pa} S_2
\\
\intertext{and}
   G_3(S_2, S_3)
& =  
   f_3^\s\(\psi(\x) + \psi(P)\S\)
   - \(zg-1\)^{\pa} \( 1 - \psi(z^2 (ag_x + bg_y))S_2 + S_3 \)
\end{alignat*}
for $ f_3(x,y,z) = z g -1 $,
and we have to find the unique
$ s_2 $ and $ s_3 $ in~$ \pi R\la x,y,z \ra^\dagger $
with~$ G_2(s_2, s_3) = G_3(s_2, s_3) = 0 $.
Making $ G_3(S_2, S_3) $ more explicit is somewhat messy, but
it turns out we do not really need it.

Observe that $ G_2 $ is a polynomial in $S_2$ only, and using
Corollary~\ref{estimate-cor} with $ z = 0 $ one checks that~$ G_2(S_2) = 0 $ has
a unique solution in $ \pi R\la x,y,z\ra ^\dagger $.
This solution must be~$ s_2 $, and
the map $ \phihat $ from $ R \la x, y, z \ra ^\dagger $
to itself maps $ x $ to
$ x^{\pa} + \psi(za)s_2 $, and~$ y $ to~$ y^{\pa} + \psi(zb)s_2 $.

In order to determine $ \phihat(z) $ one could also compute~$ s_3 $,
but its image $ \phihatA(z) $ in $ \Adag = R\la x, y, z \ra ^\dagger / ( f, zg-1) $
is already determined by $ \phihatA(z) \phihatA(g) = 1 $ because
$ g $ is in $ R[x,y] $. More precisely, we have
\begin{equation*}
  \phihat(g) = g^\s\bigl(x^{\pa} + \psi(za)s_2, y^{\pa}+\psi(zb)s_2\bigr)
=
  g^{\pa} - F
=
  g^{\pa} (1 - z^{\pa} F)
\end{equation*}
with~$ F $
in~$ \pi R\la x,y,z\ra^\dagger $, so that~$ \phihatA(z) = \phihatA(g)^{-1} = z^{\pa} \bigl(1+\sum_{m=1}^\infty (z^{\pa}F)^m\bigr)$.

\begin{example} \label{raisex1example}
Let us take $ g = f_y $, under
the assumption that $ f_y $ is not identically~0 modulo~$ \pi $.
Then we can take $ a = 0 $ and $ b = 1 $,
so that $ A = R[x,y,z]/(f,zf_{y} - 1) $, and, letting~$ h = f_{y,y} $,
our matrices are
$ \Jac 
=
\begin{pmatrix}
  f_{x}     & f_{y} & 0       \\
  z f_{x,y} & z h   & f_{y} \\
\end{pmatrix}
$,
$ P = \begin{pmatrix}
    0    &  0    \\
    z    &  0    \\
  -z^3 h &  z  \\
\end{pmatrix}
$,
$ \Delta_2 = 0 $
and
$ \Delta_3 = 
\begin{pmatrix}
  1     & 0  \\
 -z^2 h & 1 \\
\end{pmatrix}
$.
We then obtain $ \phihat $ on $ R\la x, y, z \ra^\dagger $ with~$ \phihat(x) = x^{\pa} $,
as in~Remark~\ref{liftcurve}.
The first equation simplifies to
\begin{equation*}
   G_2(S_2) 
 =
   f^\s\(\psi(x), \psi(y)+\psi(z)S_2\)
   - f^{\pa} 
   - \(z f_y -1 \)^{\pa} S_2
\,,
\end{equation*}
which through its unique root $ s_2 $ in~$ \pi R \la x , y, z \ra ^\dagger $
determines $ \phihat(y) = y^{\pa} + z^{\pa} s_2 $ and~$ \phihatA(y) $.
\end{example}

\begin{example} \label{kedlayaexample}
Let us apply Example~\ref{raisex1example} to $ y^2 - Q(x) $,
where $ p \ne 2 $, $ Q(x) $ in $ R[x] $ is
of degree $ 2g+1 $, and its reduction in $ k[x] $ has degree
$ 2g+1 $ and no multiple roots. (In other words, if we take $ R = W(k) $,
the Witt vectors of $ k $, then we are in the
situation studied in~\cite{Ked01}.)
Inverting $ 2 y $, we obtain an open part corresponding to
$ R[x,y,z]/(y^2-Q(x), 2yz-1) $.  
We have 
$ \Jac = \begin{pmatrix}
 -Q'(x) & 2y &  0 \\
  0     & 2z & 2y \\
\end{pmatrix}$,
so we can take
$ P =
\begin{pmatrix}
  0    & 0   \\
  z    & 0   \\
 -2z^3 & z \\
\end{pmatrix}$,
$ \Delta_2 $ the zero matrix, and
$ \Delta_3 = 
\begin{pmatrix}
  1    & 0  \\
 -2z^2 & 1  \\
\end{pmatrix}
$.
Then $ s_2 $ is the unique root in $ \pi R\la x, y, z \ra ^\dagger $
of
\begin{equation*}
G_2(S_2) =  \bigl(y^{\pa} + z^{\pa} S_2 \bigr)^2 - Q^\s(x^{\pa}) - \(y^2 - Q(x)\)^{\pa} - \(2yz-1\)^{\pa} S_2
\,,
\end{equation*}
and $ \phihat(y) = y^{\pa} + z^{\pa} s_2 $.
Then $ \phihatA(y) $ in $ \Adag = R\la x, y, z \ra^\dagger / (f, 2yz-1 ) $
is the unique element congruent to $ y^{\pa} $ modulo $ \pi $
that satisfies~$ \phihatA(y)^2 - Q^\s(x^{\pa}) = 0 $, hence it must coincide
with the explicit formula given in \cite{Ked01} when $ \pa = p $
and~$ R = W(k) $.
\end{example}

\begin{example} \label{hyperellipticp=2}
If~$ k $ has characteristic~2, then,
according to Proposition~7.4.24 and Remark~7.4.25 of \cite{Liu},
a (hyper)elliptic curve of genus~$ g \ge 1 $ with a separable map to~$ \P_k^1 $ can be given by
\begin{equation*}
f =y^2 + Q(x) y + P(x)
\end{equation*}
with~$ Q(x) $ and~$ P(x) $ in~$ k[x] $ satisfying
\begin{equation*}
2 g + 1 \le \max\{ 2 \deg(Q) , \deg(P) \} \le 2 g + 2 
\end{equation*}
and~$ \gcd(Q,  P'^2 + P Q'^2 ) = 1 $.
If~$ U $ and~$ V $ in~$ k[x] $ satisfy~$ Q U + (P'^2 + P Q'^2) V = 1 $,
then~$ P_1 f_x + P_2 f_y = 1 $ in~$ k[x,y] / (f) $
for~$ P_1 = (P' - Q' y) V $
and~$ P_2 = U - Q'^2 y V $.
We can lift~$ P $, $ Q $, $ U $ and~$ V $ to polynomials in~$ R[x] $ of the same degrees,
and using these lifts instead obtain a curve over~$ R $ that
reduces to the one we started with over~$ k $, and have~$ P_1 f_x + P_2 fy \equiv 1 $
modulo~2 in~$ R[x,y]/ (f) $.
\end{example}

\section{Expansions} \label{expansions}

We now return to our curve~$ C/R $ as in Section~\ref{sec:intro},
and describe what happens at each end (the
much simpler case of a full residue disc
is discussed in Remark~\ref{full-disc-remark}).
We use notation as in Notations~\ref{basicnot} and~\ref{global-notation2}.
In particular, we have~$ X = \Spec(A) $ for $ A = R[\x]/(f_2,\dots,f_n) $ as in Assumption~\ref{Aassumption},
with~$ \Ahat = \Rx / (f_2,\dots,f_n) $ and~$ \Adag = \Rxdag / (f_2,\dots,f_n) $.
We assume that $ \Abar \ne 0 $, so that~$ \Spec(\Abar) $
contains the generic point~$ \eta $ of~$ C_k $. This will imply
that the local expansion of an element in~$ A $ has coefficients with
non-negative valuation.

In order to compute the pairing~\eqref{serre-over-K} for 1-forms~$ \eta $
and~$ \o $ of the second kind, defined over~$ K $, we can enlarge~$ K $
to~$ L $ so that all residue discs corresponding to the ends
have an~$ L $-rational point, and apply~\eqref{serre-over-L}. But in order to compute the contribution
of a specific end~$ \ee $, corresponding to a residue disc~$ \D $, it suffices to enlarge~$ K $ to a~$ \tK $
such that~$ \D $ has a~$ \tK $-rational point~$ Q $.
Keeping~$ [\tK : K] $ small will speed up calculations.
However, how the poles of the 1-forms involved lie with respect
to~$ Q $ also matters; see the calculations in the proof of Lemma~\ref{expansion-lemma}
below, as well as Example~\ref{bad-pole-example}.

Having fixed~$ Q $ and~$ \tK $, let~$ \tR $ be its valuation ring, and
view~$ Q $ as a section of~$ C_{\tR} \to \tR $.
Because we can choose the parameter~$ t $ on~$ \D $ defined over~$ \tK $,
we see that~$ \res_{\ee} (\eta \int\o) $ is in~$ \tK $.
Let~$ q $ be the reduction of~$ Q $ in the special fibre of~$ C_{\tR} $,
and~$ G_q $ the subgroup of~$ G = \Gal(\ol{K}/K) $ of elements that fix~$ q $.
Then~$ G_q $ preserves~$ \D $ and~$ \ee $, and although
it may change~$ Q $ and~$ t$,
we have that~$ \res_{\ee} (\eta \int\o) $ is in~$ \tK^{G_q} $ because it is independent of the parameter,
as remarked after Definition~\ref{rigid-res-def}.
The equivariance of~$ \res_\ee $ with respect to the
action of~$ G $ then implies that the contribution to~\eqref{serre-over-L} of the~$ \ee_i $ corresponding
to the~$ G $-orbit of~$ q $ is
equal to the trace from $ \tK^{G_q} $ to~$ K $ of~$ \res_{\ee} \eta \int \omega $.

In Theorem~\ref{globallift0} and Corollary~\ref{phidagA-phihatA}, working over
the original~$ R $, we constructed a~$ \s $-semilinear
endomorphism~$ \phihatA $ of~$ \Ahat $ that lifts the $ \pa $-power endomorphism on~$ \Abar $,
and induces an endomorphism~$ \phidagA $ of~$ \Ahat $.
In order to calculate~$ \res_\ee \phidagA(\eta) \int \omega  $,
after choosing a parameter~$ t $ for~$ \D $,
we could calculate the expansion in~$ t $ of $ \phidagA(\eta) $ as follows.
We compute the $ \phidagA(x_i) $ in~$ \Adag $, which we can plug
into~$ \eta $ in order to find~$ \phidagA(\eta) $. If we compute the expansions in~$ t $ of the~$ x_i $, we
can then substitute those into~$ \phidagA(\eta) $.

However, we use a different approach.
Instead of computing the $ \phidagA(x_i) $ and then their expansions
on $ \ee $, we expand on $ \ee $ the $ x_i $ in the defining equation~\eqref{globalG}
and solve the resulting local equation.  This way we can obtain the expansions of the~$ \phidagA(x_i) $
directly, without the need of substituting expansions into expansions.
Another advantage is that we work with expressions that
contain only the local parameter, not all the variables $ x_i $.
Also, in practice the local expansions can converge on a larger annulus
than one might expect from the behaviour of the global~$ \phidagA(x_i) $
(see Examples~\ref{runningexample3},\ref{runningexample4} and~\ref{weierstrass-example}).

The drawback is that we have to solve the local equation
for each of the ends (or at least for representatives of Galois
orbits, as discussed earlier in this section), so if there are many of those, it may be
better to compute the $ \phidagA(x_i) $ globally first and then substitute
expansions of the $ x_i $.  In order to maintain this flexibility,
in Proposition~\ref{expest} we also discuss how the estimates
on the coefficients in the global~$ \phidagA(x_i) $ translate into
estimates on the coefficients in their local expansions.

For these expansions, we have enlarged~$ K $ to~$ \tK $,
and work on the curve~$ C_{\tR} $ over the valuation ring~$ \tR $
of~$ \tK $. But in order to avoid cumbersome notation, we simply
write~$ K $ for~$ \tK $ and~$ R $ for~$ \tK $, and simplify
the notation accordingly. (Note~$ \s $ in Notation~\ref{basicnot} was
used to write down the system~\ref{globalG} over the
original~$ R $, and is used in the action of~$ \phidagA $, but plays no role in
the expansions here.)

In order to formulate our local estimates, we  shall introduce various rings and other
structures.
For this, we fix a local equation~$ t $ of $ Q $ on $ C $ (as scheme), which we shall
also view as a parameter on $ \D $ and $ \ee $, and use it to
make the restriction map from~$ \Adag $ to rigid analytic
functions $ U_r $, as described in~\eqref{eq:annulus-fun}, with~$ r \uparrow 1 $, explicit on~$ \ee $.

Let $ \OCq $ be the local ring on~$ C $ of the reduction $ q $ of $ Q $.
Note that $ \OCq/(t) \iso R $,
and the completion of $ \OCq $ with respect to $ (t) $ is
isomorphic to~$ R[[t]] $.  We shall refer to the resulting
map $ \OCq \to R[[t]] $,
or any of the analogues described below, as the expansion map
at $ Q $.

In order to describe these, let
\begin{equation*}
   \Rthat
=
 \biggl\{ \sum_{m\in\Z} a_m t^m \textup{ with all } a_m \textup{ in } R \textup{ and } \lim_{m\to-\infty}a_m=0 \biggr\}
\end{equation*} 
be the $ \pi $-adic completion of $ \Rt $.  Then 
any element in $ \Rthat $
that is not in~$ \pi \Rthat $ is in $ \Rthat^* $: we can write it
as $ t^d f - \pi g $ with $ f $ in $ R[[t]]^* $ and $ g $ in $ \Rthat $,
which has inverse $ t^{-d} f^{-1} ( 1 + \sum_{m \ge 1} (\pi t^{-d} f^{-1} g)^m ) $.

\begin{lemma} \label{expansion-lemma}
Let~$ \eta $ be the generic point of~$ C_k $ inside~$ C $,
and~$ \OCgen $ the corresponding local ring.
Then any element of~$ \OCq \setminus \pi \OCq $ is mapped
to a unit of~$ \Rthat $. 
In particular, the expansion map extends to an expansion map~$ \ex : \OCgen \to \Rthat $,
and it is defined on~$ A $.
\end{lemma}

\begin{proof}
For an element~$ a $ of~$ \OCq \setminus \pi \OCq $, let~$ d \ge 0 $
be the order at~$ q $
of its image in the local ring~$ \O_{C_k,q} = \OCq/\pi \OCq $.
Then~$ a = t^d b- \pi c $ for some~$ b $ and~$ c $ in~$ \OCq $
with~$ b(q) \ne 0 $, i.e., $ b $ in~$ \OCq^* $. We have just
seen above that this implies that the image of~$ a $ under the expansion map is in~$ \Rthat^* $.
The local ring~$ \OCgen $ is the valuation ring in the function
field~$ K(C_K) $ of~$ C_K $ for the valuation corresponding to~$ \pi $.
Viewing this function field as the field of fractions of~$ \OCq $,
and using that~$ \OCq $ is a unique factorisation domain, we
see that the expansion map extends to~$ \OCgen $.
Finally, as mentioned at the beginning of this section, our
assumption that~$ \Abar \ne 0 $ implies
that~$ A \subseteq \OCgen $ in~$ K(C_K) $, and we have the
expansion map at~$ Q $ on~$ A $.
\end{proof}

\begin{remark}
If~$ \Abar = 0 $ then, because~$ \OCgen $ is a unique factorisation domain,
for any~$ a $ in~$ A $ there exists~$ m \ge 0$, depending on~$ a $,
such that~$ \pi^m a $ is in~$ \OCgen $, and this element has
an expansion in $ \Rthat $. So in this case, the expansion
map on~$ A $ would take values in~$ \Rthat \otimes_R K $.
\end{remark}

Observe that an element of~$ \OCgen $ that lies in~$ t^{-d} \OCq $ for some
integer~$ d $ has an expansion in~$ \Rt $.
This applies in particular to an element~$ a $ of~$ A $ that
in~$ \D $ has a pole only along $ Q $.
Because in many applications the expansions of the~$ x_j $ in~$ A $ will be 
in~$ \Rt $, we include Proposition~\ref{expest}(1),
which deals with this specific case.

\begin{remark} \label{general-expansion}
For the description of the expansion map at~$ Q $ we did not
use that~$ C_k $ is smooth. Instead, we could have worked on a
regular model~$ C $ over~$ R $ of~$ C_K/K $ with a section~$ Q : R \to C $ (where~$ K $ is now possibly
larger than the original~$ K $).
Indeed, by~\cite[Corollary~9.1.32]{Liu}, the section~$ Q $ hits
exactly one irreducible component of the special fibre~$ C_k $,
this irreducible component has multiplicity~1 in the fibre~$ C_k $
(so is locally at~$ q $ defined by~$ \pi $), and is smooth over~$ k $ at
the intersection point~$ q $.
The same argument as above then gives that the expansion map to~$ \Rthat $
exists on $ \OCgen $ with~$ \eta $ the
generic point of the irreducible component.
The expansion map on~$ A $ takes values in~$ \Rthat $
if $ A \subseteq \OCgen $ inside~$ K(C_K) $, and
in~$ \Rthat \otimes_R K $ otherwise.
(Note that this condition on~$ A $ can be checked only after
extending the original~$ K $ and~$ R $ because the regular model
depends on~$ R $.)
\end{remark}

Because $ \Rthat $ is $ \pi $-adically complete, the expansion map 
$ A \to \Rthat $ at~$ Q $ extends to a map~$ \Ahat \to \Rthat $,
which induces a map
\begin{equation} \label{expansionmap}
 \ex : \Rx \to \Rthat
\,.
\end{equation}
Abusing notation, we shall denote all those expansions maps from $ \Rx $, $ \Rxdag $, $ A $, $ \Ahat $ and~$ \Adag $
to~$ \Rthat $ by~$ \ex $.
In order to describe the image of $ \Rxdag $ under this map,
and to be able to describe our estimates, we
introduce various subsets of $ \Rthat $, as well as a subring~$ \Rtstar $.  We shall show in
Proposition~\ref{expest} that $ \xi $ maps $ \Rxdag $ into $ \Rtstar $
and provide a description for bounds on the coefficients involved.

\begin{notation} \label{Ralphabetanotation}
(1) For any real numbers $\a$ and $\b$ with $\a>0$ we let
\begin{alignat*}{1}
\Rtab {\a} {\b}
& =
\biggl\{\sum_{m\in\Z}{a_m t^m}  \text{ in } \Rthat \text{ with } v(a_m)\ge  -\a^{-1} (m-\b) \biggr\}
,
\\
\Rtstar & = \bigcup_{\a,\b} \Rtab {\a} {\b}
.
\end{alignat*}

(2)
We let~$\Rtab {0} {\beta} = \cap _{\a > 0} \Rtab {\a} {\b} $,
so it is the subset of $\Rtstar $ with elements of the form~$\sum_{m \ge \beta} a_mt^m$.

(3)
For any~$ c \ge 0 $, we put~$V_c = \{ \sum_{m\in\Z}{a_m t^m} \text{ in } \Rthat \text{ with } v(a_m) \ge c \}$.
\end{notation}

\begin{figure}[h]
\begin{center}
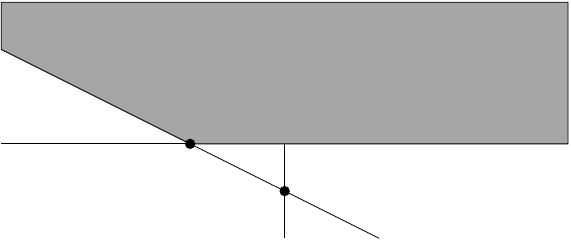
\label{Rabfigure}
\end{center}
\caption{}
\end{figure}

Each element of~$ \Rtab {\a} {\b} $ is a rigid function 
as described in~\eqref{eq:annulus-fun} on (a possibly
narrower) $ \ee $.
One can visualise the conditions on the $ a_m $ by drawing the
region in the plane in which the pairs $ (m,v(a_m)) $ for non-zero $ a_m $
can lie, as in Figure~\ref{Rabfigure}.
This is convenient because the multiplication $ a_m t^m a_n' t^n $ in~$ \Rthat $
corresponds to the addition of vectors~$ (m, v(a_m)) + (n, v(a_n') ) $
in the plane.

The $ \Rtab {\a} {\b} $ are our local analogues of the $ \W{\a,\b} $
of Definition~\ref{Wab-def}, but they use only one
variable, and the convergence condition is for negative powers
of that variable.
They show up quite naturally. For example,
we have seen above that an element of~$ \OCq \setminus \pi \OCq $
can be written as~$ t^d b - \pi c $ with~$ b $ in $ \OCq^* $ and~$ c $ in~$ \OCq $.
Its expansion in~$ \Rthat $ has inverse~$ t^{-d} \xi(b)^{-1} (1 + \sum_{m \ge 1} (\pi t^{-d} \xi(b)^{-1} \xi(c))^m ) $,
which is in~$ \Rtab {d/v(\pi)} {-d} $ because $ \xi(b) $ is in~$ R[[t]]^* $ and~$ \xi(c) $
in~$ R[[t]] $.

\medskip

We record the following properties.

\begin{lemma} \label{rtabprop}
  The subsets above have the following properties.
  \begin{enumerate}
  \item
  $ \Rtab {\a} {\b_1} \subseteq \Rtab {\a} {\b_2} $ for~$ \b_1 \ge \b_2 $, $ \Rtab {\a_1} {\b} \subseteq \Rtab {\a_2} {\b} $
  for~$ \a_1 \le \a_2 $.

  \item The elements in $ \Rtab {\a} {\b} $ converge for $ p^{-\a^{-1}} <
    |t| < 1 $.

  \item \label{rtabprop2} Multiplication in $ \Rthat $ induces
a map
\begin{equation*}
\Rtab {\a_1} {\b_1} \times \Rtab {\a_2} {\b_2}  \to \Rtab  {\max\{\a_1,\a_2\}} {\b_1+\b_2}
.
\end{equation*}

  \item $ \Rtstar $ is a subring of $ \Rthat $, as are the $\Rtab {\a} 0 $.

  \item $ \Rtab {\a} {\b} $ is an ideal of $\Rtab {\a'} 0 $ if $ \b \ge 0 $ and $ \a \le \a' $.
  
  \item $ \Rtab {\a} 0 $ is $ \pi $-adically complete.

  \item The units of $ \Rtab {\a} 0 $ are those $ \sum_m a_m t^m $ with $ a_0 $ in $ R^* $.

  \item \label{rtabprop7} If~$ c > 0 $ then $\Rtab {\a} {\b} \cap V_c \subseteq \Rtab {\ta} {\tb} $ if $\tilde{\a} \ge \max\{\a,\a + c^{-1}(\tilde{\b} - \b)\}$, i.e.,
if~$\tilde{\b} \le \b + (\tilde{\a} - \a)c$ and~$\tilde{\a} \ge \a$.

  \item If $x $ is in $ \Rtab {\a} {\b} \cap V_c$, then $x^n $ is in $ \Rtab {\a} {n\b} \cap V_{nc}$.
  \end{enumerate}
\end{lemma}

\begin{proof}
All are easy to prove. For part~\ref{rtabprop7}, the area corresponding
to $ \Rtab {\a} {\b} \cap V_c $ has a vertex at~$ (\b,0) + c (-\a,1) = (\b - c \a , c) = (\b + (\tilde\a - \a) c,0) + c (-\tilde\a,1) $,
and this area is contained in the one corresponding to~$ \Rtab {\tilde\a} {\tilde\b} $
under the given assumptions on~$ \tilde\a $ and $ \tilde\b $ (see Figure~\ref{Rabcfigure}).
\end{proof}

\begin{figure}[t]
\begin{center}
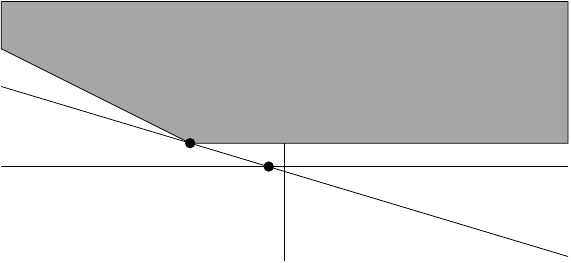
\caption{}
\label{Rabcfigure}
\end{center}
\end{figure}

\begin{remark} \label{expansion-slopes}
(1)
We can now make the statement in Lemma~\ref{expansion-lemma}
more precise for each~$ \ex(x_i) $ (but the same method works
for any element of~$ K(C_K) $ that is generically defined on~$ C_k $).
Suppose that a Zariski open~$ \Spec(B) \subseteq C $ contains~$ \Spec(\O_{C,q}) $,
with~$ B $ a quotient of a polynomial
ring~$ R[y_1, \dots, y_l] $. Then~$ \ex(x_i) $ is in~$ \Rtab {\a} {\b} $ with~$ \a $
and~$ \b $ determined as follows.
In~$ K(C_K) $, write~$ x_i = \frac{b_1} {b_2} $ for some~$ b_1 $ and~$ b_2 \ne 0 $ in~$ B $.
If~$ \ol{b_2} $ is zero in~$ B/(\pi) $, then we may take out
a factor~$ \pi $ (see Algorithm~\ref{KtoRbasis}(2) how to do this algorithmically).
Because~$ C_k \cap \Spec(B)$ is non-empty, the generic point~$ \eta $ of~$ C_k $ is in~$ \Spec(B) $,
so that~$ x_i $ is generically defined on~$ C_k $. Therefore
any such factor~$ \pi $ in~$ b_2 $ must match one in~$ b_1 $,
and iterating this we may assume that the reduction of~$ \ol{b_2} $ in~$ B/(\pi) $
is non-zero.
Then the expansion of~$ b_2 $ is in~$ t^{\b_2} \Rtab {\a} 0 $ with~$\b_2 $
the number of zeroes of~$ b_2$ in this residue disc, i.e., the
degree of the distinguished polynomial of~$ \ex(b_2) $,
and~$ \a = - c^{-1} $
for~$ c $ the smallest of the negative slopes of the Newton polygon
of~$ \ex(b_2) $. (If the lowest non-zero term in~$ \ex(b_2) $ has
degree~$ \b_2 $ then we can take~$ \a = 0 $ as in Notation~\ref{Ralphabetanotation}(2).)
Similarly~$ \ex(b_1) $ is in~$t^{\b_1} R[[t]] $ with~$ \b_1 $ the number of zeroes of~$ b_1 $ in this
residue disc. According to Lemma~\ref{rtabprop}(3) and~(7),
we then have~$ \xi(x_i) = \xi(b_1) \xi(b_2)^{-1} $ in~$ \Rtab {\a} {\b_1-\b_2} $.

Note that a calculation of~$ \ex(b_2) $ with coefficients in~$ R/(\pi^m) $
either determines~$ c $ or gives a negative upper bound for it, hence
a positive upper bound for~$ \a = - c^{-1} $.
If there is a term~$ a_m t^m $ with~$ m < \b_2 $ in this expansion
for which the point~$ (m, v(a_m)) $ is on or below the line segment connecting~$ (\b_2, 0) $
and~$ (0, (m+1) e^{-1} ) $, then we can determine~$ c $. If this
is not the case, we can conclude that~$ c < - (m+1) e^{-1} \b_1 $,
hence~$  \a > e \b_1 (m+1) ^{-1} $.

(2)
If~$ q $ is in~$ \Spec(A) $, then we can take~$ B = A $, $ b_2 = 1 $, and
find (again) that~$ \xi(x_i) $ in~$ R[[t]] $.
It is easy to see that here the expansion of each element of~$ \Ahat $
is in~$ R[[t]] $.
\end{remark}

We now fulfill an earlier promise, and show that~$ \Rxdag $ maps to $ \Rtstar $
under the expansion map~$ \xi $ in~\eqref{expansionmap}.
In particular, each element in $ \Adag $ is mapped to a rigid
function on (a possibly narrower) $ \ee $, which can be described
more precisely using the information from
Remark~\ref{expansion-slopes}.

\begin{prop} \label{expest}
The expansion map $ \ex $ maps $ \Rxdag $ to $ \Rtstar $.
More precisely, 
suppose~$ g(\x) = \sum_D a_D \x^D $
is in~$ \Rxdag $, with $ v(a_D) \ge \a^{-1} \cdot (D - \b) $ 
for some $ \a = (\a_1, \dots, \a_n) $ in $ (\realpos)^n $ and $ \b $ in $ \R^n $.
Denote $ \a' = (\a_1',\dots,\a_n') $  where $\frac{1}{\a_i'} = \ord_t(x_i)$  (if $\ord_t(x_i) = 0$ just take $\a_i'$ negatively large enough).
Then $\ex(g(\x))$ lies in $\Rtab {\cc} {\d} $ with $\cc$ positive, where $\cc$ and $\d$ can be chosen as follows.
\begin{enumerate}
\item If all $\ex(x_i)$ are in $\Rt$, then choose $\cc$ such that $ \frac{1}{\a_i} \ge -\frac{1}{\cc} \ord_t(x_i)$, and  let $\delta = \frac{1}{\cc} \a'^{-1} \cdot \beta$.

\item If each $\ex(x_i)$ is in some $\Rtab {\tilde{\a}} {\tilde{\b_i}} $, then just take $\cc$ such that for each $i$, 
$\cc^{-1} \ge \a_i' ( \tilde{\a} (\a_i'^{-1} - \tilde{\b_i} ) - \a_i^{-1})$, and take $\d $ so that $ - \a^{-1} \cdot \b \ge \cc^{-1} \d$.
\end{enumerate}
\end{prop}

\begin{proof}
(1)
Note that each $ \ex(x^D) $ is in $ t^{{\a'}^{-1} \cdot D} R[[t]] $.
We have $ v(a_D) \ge \frac{1}{\a} \cdot (D - \b)  \ge - \frac{1}{\cc} ({\a'}^{-1} \cdot D - \d)  $, with $\d = \cc {\a}^{-1} \cdot \b$,
so that $ a_D t^{{\a'}^{-1} \cdot D} $ is in $ \Rtab {\cc} {\d} $.
The same then holds for each $ \ex(a_D x^D ) $, hence for $ \ex(g(x)) $.

(2)
By Lemma~\ref{rtabprop}\eqref{rtabprop2} we see that $\ex(\x^D)$ lies in $\Rtab {\tilde{\a}} {D \cdot \tilde{\b}} $. If we write $\ex(\x^D) = \sum_j b_jt^j$ with $j \ge D \cdot \ord_t(\x)$ and $v(b_j) \ge -\tilde{\a}^{-1}(j - D \cdot \b)$, we see in order that $a_D\ex(\x^D)$ lies in $\Rtab {\cc} {\d} $, it is implied by 
$$ \a^{-1} \cdot (D - \b) - \tilde{\a}^{-1} (D \cdot \ord_t(\x) - D \cdot \tilde{\b}) \ge -\cc^{-1} (D \cdot \ord_t(\x) - \d) .$$ 
Therefore, $\cc$ should be chosen to satisfy $\cc^{-1} \ge \a_i' ( \tilde{\a} (\a_i'^{-1} - \tilde{\b_i} ) - \a_i^{-1})$, 
and $\d$ to satisfy $ - \a^{-1} \cdot \b \ge \cc^{-1} \d$.
\end{proof}

\begin{example} \label{runningexample3}
In Example~\ref{runningexample2}, for the curve defined by $ x^2 - y^2 -1 $
in $ \Zp[x,y] $ with $ p \ne 2 $, we had an element $ \sum_{i,j} a_{i,j} x^i y^j $
of $ \Zp\la x, y \ra ^\dagger $ for which~$ v(a_{i,j}) \ge \frac{i+j}{4p} + \frac12 $,
so that  $\a^{-1} = (\frac{1}{4p}, \frac{1}{4p})$ and $\b = (-p,-p)$.
Both $ x $ and $ y $ have poles of order 1 at each of the two
points at infinity, so that $\a_1' = \a_2' = -1$. If, for each of those ends, we centre~$ t $ at such a point~$ Q $,
then Proposition~\ref{expest}(1) applies, and the largest $ \cc $ we can take
is~$ 4p $, so that $ \ex(\sum_{i,j} a_{i,j} x^i y^j ) $
is in $ \Rtab 4p 2p $ for $ R = \Zp $.
\end{example}

\begin{example} \label{bad-pole-example}
Let us consider the curve defined by $xy - 1$ in $\Zp[x,y]$ with $p \not= 2$.
As we have already shown in Example~\ref{runningexample5},
the final solution $s = \sum_{i,j\ge 0} a_{i,j}x^iy^j$ satisfies $v(a_{i,j}) \ge \frac{i + j - p}{2p - 1}$.
We can take $\a^{-1} = (\frac{1}{2p - 1}, \frac{1}{2p - 1})$ and $\b = (-\frac{p}{2}, -\frac{p}{2})$.
There are two missing points at infinity: $[1:0:0]$ and $[0:1:0]$.
If we choose the local parameter as $t = \frac{1}{x}$, centered at the point $[1:0:0]$,
then $x$ has a pole of order one at this point and $y$ has a zero of order one at this point.
In this case, we can take $\a'^{-1} = (-1, -\e)$ where $\e > 0$ is close to $0$.
Now $x = \frac{1}{t} \in \Rtab {0} {-1} $ and $y = \frac{1}{x} = t \in \Rtab {0} {1} $.
According to Proposition~\ref{expest}(1) we can take $\cc$ as $2p - 1$ and $\d = \frac{p(1 + \e)}{2(2p - 1)}$.

On the other hand, if we centre~$ t $ at another point~$ Q $,
then we get a more complicated expansion.
For example, if we choose another local parameter $t = \frac{1}{x} + p$,
then $x = \frac{1}{t - p} = \sum_{n \ge 1} \frac{p^n}{t^{n+1}} \in \Rtab {1} {-1} $, and
$y = \frac{1}{x} = t - p \in \Rtab {1} {1} $. 
By Proposition~\ref{expest}(2), one can take $\cc^{-1} = 1 + \frac{2p}{\e(2p - 1)}$, which could be arbitrarily large.
\end{example}

\section{The local Frobenius} \label{local-frobenius}

We keep the setup as in Section~\ref{expansions}, and
fix an end~$ \ee $ together with a parameter~$ t $. We recall
that the original~$ K $ may
have been replaced with a finite extension of it.

Because the expansion map~\eqref{expansionmap} is a ring homomorphism, it will transform
the system~\eqref{globalG} with its solution~$ \sol $
into a system~$ H(\S) $ with a solution~$ \soll = \ex(\sol) $.
This determines the expansion
\begin{equation} \label{local-frob-expansion}
\ex(\phidagA(\x)) = \ex(\psi(\x)) + \ex(\psi(P)) \soll 
,
\end{equation}
and 
Corollary~\ref{estimate-cor} together with Proposition~\ref{expest} provide estimates on the coefficients
in~$ \soll $ in terms of the $ \Rtab {\a} {\b} $.
This method is essentially global.

However, unless there are many ends,
it should be much more efficient
if we can compute~$ \soll $ directly from~$ H(\S) $, which is
a local method.
For example, the system~\eqref{globalG} becomes simpler after choosing a local parameter.
That this can be done is the content of Theorem~\ref{locallift0},
and in Proposition~\ref{hensel2} and its corollaries we then discuss estimates on the coefficients in the expansions
obtained this way.
As mentioned at the beginning of Section~\ref{expansions}, the estimates provided by this local method can be quite different from (and are usually better)
than the ones obtained from the global method
(see Examples~\ref{runningexample3}, \ref{runningexample4} and~\ref{weierstrass-example}).

\begin{theorem} \label{locallift0}
Let $ P $ and $ \Delta_2,\dots,\Delta_n $ and
$ G(\S) $ be as in Theorem~\ref{globallift0},
and let~$ \phihatA $ be the resulting $ \s $-linear endomorphism
of $ \Adag $.
Then the expansions of the~$ \phihatA(x_i) $ can be computed directly.
More precisely, if we let~$ H(\S) $ be obtained from~\eqref{globalG}
by applying $ \ex $ to the coefficients, then~$ H(\S) = 0 $ has a unique
solution $ \soll $ in $ (\pi \Rthat)^{n-1} $, which equals $ \xi(\sol) $ with
$ \sol $ as in Theorem~\ref{globallift0},
and~$ \ex(\phidagA(\x)) = \ex(\psi(\x)) + \ex(\psi(P)) \soll $.
\end{theorem}

\begin{proof}
Recall that
$ \phihatA $ is induced by the $ \s $-linear endomorphism~$ \phihat $
of $ \Rxdag $ mapping $ g(\x) $ to $ g^\s(\psi(\x) + \psi(P) \sol ) $,
with $ \sol $ the unique solution in $(\pi \Rxdag)^{n-1}$ of
$ G(\S) = 0 $. 
So $ \ex(\phidagA(\x)) = \ex(\psi(\x)) + \ex(\psi(P)) \ex(\sol) $.
Because $ \ex(\sol) $ is in $ (\pi \Rtstar )^{n-1} $
by Proposition~\ref{expest}, it suffices to show that  
$ H(\S) = 0 $ has a unique solution $ \soll $ in
$ (\pi \Rthat)^{n-1} $.
By that proposition the coefficients in $ H(\S) $
are in $ \Rtstar $, and $ H(\S) $
has inherited the following properties from $ G(\S) $:
\begin{itemize}
\item
$ H(0) \equiv 0 $ modulo $ \pi \Rtstar $;

\item
$ \Jac_H(0)  = \ex\( \Jac_G(0) \)$,
hence $\Jac_H(0) \equiv\Id_{n-r} $  modulo $ \pi \Rtstar $.
\end{itemize}
By Hensel's lemma for the $ \pi $-adically complete ring
$ \Rthat $, there is a unique $ \soll $ with coefficients in
$ \pi \Rthat $. But $ \xi(\sol) $ provides such a solution by Proposition~\ref{expest},
so they must be equal. The final formula follows by applying
$ \xi $ to~\eqref{eq:phiact}.
\end{proof}

\begin{remark}\label{locnewtonrem}
Applying $ \ex $ to the coefficients of~\eqref{globalG}
maps $ f^{\pa} $ and the $ f_j^{\pa} $ to~0, so that
$ H(\S) = f^\s (\ex(\psi(\x)) + \ex(\psi(P))\S ) $ is determined
by~$ f $ and~$ P $ alone.
\end{remark}

\begin{remark} \label{H-disc-remark}
In Theorem~\ref{locallift0} we consider an end~$ \ee $, i.e.,
the reduction~$ q $ of the section~$ Q $ defined by~$ t=0 $ for the parameter~$ t $ on
the corresponding residue disc
is not in~$ \Spec(\Abar) \subseteq C_k $. The corresponding
statement for~$ q $ in~$ \Spec(\Abar) $ is that~$ H(\S) $ has coefficients
in~$ R[[t]] $ and has a unique solution~$ \soll $ in~$ (\pi R[[t]])^{n-1} $,
which equals~$ \ex(\s) $, so that~$ \ex(\phidagA(\x)) = \ex(\psi(\x)) + \ex(\psi(P)) \soll $
in~$ R[[t]]^n $.
\end{remark}

In order to give estimates on the coefficients involved in the
solution $ \soll $ described in Theorem~\ref{locallift0} by means of Newton iteration, 
we prove the local analogues of Proposition~\ref{itprop} and related results.
Quite analogous to what we did in Section~\ref{global-frobenius},
we want to start the
Newton iteration for solving the system in Theorem~\ref{locallift0}
with~$ z $ having entries 0.

\begin{prop} \label{hensel2}
Let $ H_{r+1}(\S),\dots,H_n(\S) $ be in $ \Rtstar [\S] $
of total maximal degree $ N \ge 1 $, obtained as in \ref{locallift0}.
Suppose there are~$z$ in $ \Rtstar^{n-r}$, $ \mu, \nu, \rho$ in~$\R$ with $ \mu >0 $, and $b, c >0$ such that:
\begin{enumerate}

\item
$z$ has entries in $ \Rtab {\mu} {\nu} \cap V(b)$;

\item
each $H_j(z)$ lies in $\Rtab {\mu} {\rho} \cap V(c)$;

\item
the entries of $\Jac_H(z)$ are in $\Rtab {\mu} {0} $ and $\det(\Jac_H(z)) $ is in $ \Rtab {\mu} {0} ^*$;

\item
$\rho \le \nu$.
\end{enumerate}
Let $\e = -\Jac_H(z)^{-1} H(z)$ and $z'= z + \e$. Then (1) through~(4)
hold for $z'$ instead of~$z$ if we replace $\mu$ with $\mu'$, 
$b$ with $b' = \min\{b,c\}$,
$ c $ with~$ c' = 2c $, 
$ \nu$  with~$\nu' = \nu + (\mu' - \mu)b'$, 
and~$ \rho $ with~$ \rho' = \rho + (\mu' - \mu) \min\{ |I - J| b + |J|c\}$.
Here $\mu'$ can be determined as follows.

Write $H_j(\S) = \sum_{l=0}^N \sum_{|I| = l}H_{j,I}(t)\S^I$ for some $N \ge 2$.
Assume that each~$H_{j,I}(t)$ with~$ l = |I| \ge 2 $ lies in $\Rtab {\mu} {\b_l} $ for some $\b_l$ in $\mathbb R$.
Then
\begin{equation*}
\mu' = \mu +
\max_{2 \le l \le N \atop j=1, \dots, n-r} \maxp_{ |I| = l \atop 0 < K \le I - e_j }
\left\{ -\frac{\b_l + (|I| - |K| - 1)\nu + |K| \rho}{ (|I| - |K| - 1)b + |K|c}\right\}
\,,
\end{equation*}
where $e_j$ means the index whose $j$-th component is~1 and the other components are~0.
Moreover, if $b'= b$ then further Newton iteration does not change $\mu', \nu', \rho'$.
\end{prop}

\begin{proof}
Expanding $ H_j(z')$ and $\Jac_H(z')$ with $z' = z + \e$ gives
\begin{equation*}
H_j(z') = \sum_{l \ge 2} \sum_{|I| = l} \sum_{0 < J < I, |J| \ge 2} \binom{I}{J} H_{j, I}(t) z^{I-J}\e^J,
\end{equation*}
\begin{equation*}
\Jac_H(z')_{i,j} = \Jac_H(z)_{i,j} + \sum_{l \ge 2} \sum_{|I| = l} H_{r+i,I}(t) i_{r+j} \sum_{0 < K \le I - e_j} \binom{I-e_j}{K} z^{I -e_j-K}\e^K
.
\end{equation*}
By the assumptions, together with parts~(3) and~(9) of Lemma~\ref{rtabprop}, each term $H_{r+i,I}(t) z^{I - e_j -K}\e^K$ lies in
\begin{equation*}
\Rtab {\mu} {\b_l + |I-K-e_j|\nu + |K|\rho} \cap V(|I-K-e_j|b + |K|c)
.
\end{equation*}
If we set 
\begin{equation*}
\mu' = \mu +
\max_{2 \le l \le N \atop j=1, \dots, n-r} \maxp_{ |I| = l \atop 0 < K \le I - e_j }
\left\{ -\frac{\b_l +  (|I| - |K| - 1)\nu + |K| \rho}{ (|I| - |K| - 1)b + |K|c}\right\}
\end{equation*}
then from Lemma~\ref{rtabprop}(8) we see that $\Jac_H(z')$ lies in $\Rtab {\mu'} {0} ^{*}$. Thus (3) is proven.

For (1), we note that we already have $z$ in $\Rtab {\mu} {\nu} \cap V(b)$ and $\e$ in $\Rtab {\mu} {\rho} \cap V(c)$,
so it follows that $z'$ lies in $\Rtab {\mu} {\min \{\nu, \rho\}} \cap V(\min \{b,c\})$.
Because we already defined~$\mu'$ above, by Lemma~\ref{rtabprop}(8)
we know that $z'$ lies in $\Rtab {\mu'} {\nu'} \cap V(b')$ where $b' = \min \{b,c\}$ and $\nu' = \nu + (\mu' - \mu)b$.

For part (2) the proof is similar. In the expansion of $H_j(z')$, each term $z^{I-J}\e^J$ is in 
$\Rtab {\mu} {|I-J|\nu + |J|\rho} \cap V(|I-J|b + |J|c)$. This is already in $V(c')$ since $c' = 2c$ and only 
terms of degree at least~2 show up in the expansion of $H_j(z')$.
Again, by Lemma~\ref{rtabprop}(8), $H_j(z')$ is in $\Rtab {\mu'} {\rho'} $ 
with $\rho' = \rho + (\mu' - \mu)\min\{|I-J|b + |J|c\}$.

Finally notice that if~$b' = b$, then 
$$ \mu'' - \mu' = \maxp \left\{-\frac{\b_l +  (|I| - |K| - 1)\nu' + |K|\rho'}{ (|I| - |K| - 1)b' + |K|c'}\right\} \le 0\,,$$
which is equivalent to saying
$ \min\{\b_l + |I-K-1|\nu' + |K|\rho'\} \ge 0.$
Now combining the definition of $\mu' - \mu$ we see also
\begin{alignat*}{1}
& \phantom{\,\,>\,\,} \b_l +  (|I| - |K| - 1)\nu' + |K|\rho'
\\
& \ge \b_l +  (|I| - |K| - 1)\nu + |K|\rho + (\mu' - \mu)( (|I| - |K| - 1)b' + |K|c) 
\end{alignat*}
is non-negative since $b' = b$. 
This means the Newton iteration will not change the parameters $\mu', \nu', \rho'$ after
two steps.
\end{proof}

\begin{corollary} \label{locesti}
If there exists a $z$ as above, then there is a unique solution $\soll$ of $H(\S) = 0$ with coefficients in $\pi \Rtstar$ obtained by Newton iteration. And as a simple estimate, if we let
\begin{equation*}
M =
\max_{2 \le l \le N \atop j=1, \dots, n-r} \maxp_{ |I| = l \atop 0 < K \le I - e_j }
\{-(\b_l +  (|I| - |K| - 1)\nu + |K|\rho) \}
,
\end{equation*}
then the final solution lies in $\Rtab {\mu + \frac{2M}{c}} {\b} $,
where~$ \b = \min_{2 \le l \le N} \{\b_l\} $.
\end{corollary}

\begin{proof}
We can easily apply the Newton iteration to get the unique final solution. Starting from $z_0 = z$, by the formula
\begin{equation*}
\mu' - \mu =
\max_{2 \le l \le N \atop j=1, \dots, n-r} \maxp_{ |I| = l \atop 0 < K \le I - e_j }
\{ -\frac{\b_l +  (|I| - |K| - 1) \nu + |K| \rho}{ (|I| - |K| - 1)b + |K|c}\},
\end{equation*}
we see that the sequences $\{\mu\}, \{\nu\}$ and $\{\rho\}$ are non-decreasing. Then
$$ \mu' - \mu \le \maxp \left\{ -\frac{\b_l +  (|I| - |K| - 1)\nu + |K|\rho}{c} \right\} \le \frac{M}{c}.$$
And in the further step we have $\mu'' - \mu' \le \frac{\b_l +  (|I| - |K| - 1)\nu' + |K|\rho'}{c'}  \le \frac{M}{2c}$ so that the limit of the sequence $\{\mu\}$ is no larger than $\mu + \frac{2M}{c}$.
\end{proof}

\begin{remark}
(1)
Since $\{\nu\}$ and $\{\rho\}$ are non-decreasing sequences, it
is easy to see they are bounded from left. They are also bounded from right by their defining equations: if $\nu$ and $\rho$ are too large, then the nominator $ \b_l + |I-K-1|\nu + |K|\rho$ is positive, so that $\mu' - \mu = 0$. Hence all the indices here form convergent sequences.

(2)
Proposition~\ref{hensel2} states that the parameters that descibes where the solution of $H(\S) = 0$ lies stabilize after 2 steps.
It does not mean that it only takes two steps to finish the Newton iteration.
Indeed, if we want to know the final solution $\sol$ of the equation $H(\S) = 0$ up to precision $N$ (see Definition~\ref{p-precision}),
we need to take at least $\lfloor \log_2 N\rfloor + 1$ steps.
\end{remark}

\begin{example} \label{runningexample4}
Let us return to Example~\ref{runningexample1} and obtain local
estimates.
In Example~\ref{runningexample3}
we derived local estimates from the global one in
Example~\ref{runningexample2}.
Recall that $ p \ne 2 $.
There are two points at infinity, $ [1,1,0] $ and $ [1,-1,0] $. 
With local parameter $ t = 1/x $ we find the expansions $ x(t) = t^{-1} $
and $ y(t) = \pm t^{-1} \sqrt{1- t^2} = \pm t^{-1} (1 - \frac12 t^2 - \frac18 t^4 -\frac{1}{16} t^6 - \cdots  ) $.
Using those in~\eqref{runexeq} we obtain
\begin{equation*}
   H(S)
 =  
   4^{-1} A(t) S^2 + A(t) S  - A(t) + 1
\end{equation*}
with $ A(t) = x(t)^{2p} - y(t)^{2p} = x(t)^{2p} - (x(t)^2-1)^p = p t^{-2p+2} + \cdots $
in $ 1 + p t^{-2p+2} R[[t]] $.

Then we can take $H_0(t)$ in $ \Rtab {2p-2} {0} \cap V(1), H_1(t)$ in $ \Rtab {2p-2} {0} ,$ and $H_2(t)$ in~$\Rtab {2p-2} {0} $. 
Starting with $z = 0$ in $ \Rtab {2p-2} {0} \cap V(1)$ and $H(0) = 1 - A(t)$ in $ \Rtab {2p-2} {0} \cap V(1)$ (so that $b = c = 1, \nu = \rho = 0$), we find:
$\mu' - \mu = 0, b' = \min\{b,c\} = 1, c'= 2, \nu' - \nu = \mu' - \mu = 0, \rho' - \rho = 2(\mu' - \mu) = 0$.
Moreover, since $b' = b$ we see that $\mu'' = \mu'$, i.e., the slope remains
the same, hence the final solution lies in~$\Rtab {2p-2} {0} \cap V(1)$.
\end{example}

We end this section by discussing how to get the desired Frobenius ``up to certain precision".

\begin{definition} \label{p-precision}
We shall say that we know $\a$ in $K$ up to precision $N$ if
we know~$\tilde \a $ in $K$ with $v(\a - \tilde \a) > N$.
We shall say that we know~$f(t)$ in $\Rtstar$ up to precision~$ N $
if we know each coefficient of $f(t)$ up to precision $N$.
\end{definition}

Note that if $\a$ lies in $R$, then to know $\a$ up to precision $N$,
is equivalent to knowing its class in the quotient ring~$ \RN R N = R/(\pi^{N'})$, where $N'$ is the smallest integer such that $N' v(\pi) > N$.

\begin{notation} \label{RtabNdef}
We let $ \RN R N ((t)) $ be the set of formal Laurent series with coefficients in~$ \RN R N $,
and let $\RtabN {\a} {\b} {N} $ denote the image of~$\Rtab {\a} {\b} \subseteq \Rtstar $
in~$ \RN R N ((t)) $.
\end{notation}

Now, to say that we know the Frobenius up to precision $N$ amounts to the same thing 
as knowing the solution $\soll$ of the system $H(\S)$, which has
entries in~$ \Rtstar $,
up to precision $N$, i.e., knowing the classes of the entries of $\soll$ in some $\RtabN {\a} {\b} {N} $.
The sets $\RtabN {\a} {\b} {N} $ behave as the sets~$\Rtab {\a } {\b} $, so we have the following lemma.

\begin{lemma} \label{RtabN}
The sets $\RtabN {\a} {\b} {N} $ have the following properties.
\begin{enumerate}
\item $\RtabN {\a} {\b_1} {N} \subseteq \RtabN {\a} {\b_2} {N} $ for $\b_1 \ge \b_2$, $\RtabN {\a_1} {\b} {N} \subseteq \RtabN {\a_2} {\b} {N} $ for $\a_1 \ge \a_2$.

\item $ \RN R N ((t)) = \cup_{\a, \b} \RtabN {\a} {\b} {N} $.

\item Multiplication in $\RN R N ((t))$ induces a map
\begin{equation*}
\RtabN {\a_1} {\b_1} {N} \times \RtabN {\a_2} {\b_2} {N} \to \RtabN {\max\{\a_1, \a_2\}} { \b_1 + \b_2} {N}
.
\end{equation*}

\item If~$ 0 < c <N $ and~$\tilde{\a} \ge \max\{\a,\a + c^{-1}(\tilde{\b} - \b)\}$ then $\RtabN {\a} {\b} {N} \cap V_c \subseteq \RtabN {\ta} {\tb} {N} $.

\item  $\RtabN {\a} {\b} {N} \cap V_c \subseteq \RtabN {\ta} {\tb} {N} $ if~$ 0 < c <N $ and~$\tilde{\a} \ge \max\{\a,\a + c^{-1}(\tilde{\b} - \b)\}$.

\item  If $x $ is in $ \RtabN {\a} {\b} {N} \cap V_c$ with $nc < N$ then $x^n $ is in $ \RtabN {\a} {n\b} {N} \cap V_{nc}$.
\end{enumerate}
\end{lemma}

The reader should view the above as an analogue of Lemma~\ref{rtabprop}.
As a consequence, the addition and multiplication are compatible
with working up to precision~$N$.
This means that, in order to find the Frobenius up to precision~$N$, 
we just need to find everything up to that precision. Thus we have the followings.

\begin{proposition} \label{NewtonprecN}
In order to know~$ \ex(\phi(\x)) $ up to precision~$ N $, it
suffices to compute the system~$ H(\S) $ from~\eqref{globalG}
with coefficients in~$ \RN R N ((t)) $, and apply Newton iteration,
starting with~$ z $ having all entries~0.
Moreover, then Proposition~\ref{hensel2} applies with
each~$ \Rtab {\cc} {\d} $ replaced with~$ \RtabN {\cc} {\d} N $.
In particular, starting with~$ z_0 = 0 $, $ \soll $ coincides
with~$ z_n $ up to precision~$ N $ for~$ n > \log_2 (eN) + 1 $,
where~$ e $ is the ramification index for the current~$ K $.
\end{proposition}

\begin{proof}
This follows from~\eqref{local-frob-expansion}, Proposition~\ref{locallift0}, 
Proposition~\ref{hensel2} and Lemma~\ref{NewtonprecN}, taking
into account that~$ c > 0 $ in the proposition implies~$ c \ge e^{-1} $
in the current~$ K $.
\end{proof}

We state the following lemma for the sake of completeness.

\begin{lemma} \label{lemmaprecision}
Let $\a$ be in $\pi^AR$ and $\b$ in $\pi^B R$, for~$A$ and $B$ in~$ \Z $. Then in order to know~$\a\b$ up to precision $N$,
it suffices to know~$\a$ up to precision~$N_1$ and $\b$ up to precision $N_2$, for~$N_1$ and $N_2$ with~$N_1 + B \ge N, N_2 + A \ge N$ and $N_1 + N_2 \ge N$.
\end{lemma}

\begin{proof}
To know $\a$ up to precision $N_1$ means we have~$\tilde\a$ in $K$ and~$\delta_1$ in $R$
with~$\a = \tilde\a + \delta_1$ and $v(\delta_1) > N_1$, similarly for $\b = \tilde\b + \delta_2$.
Then $\a\b = (\tilde\a + \delta_1)(\tilde\b + \delta_2)$ is known up to precision $N$ if $v(\delta_1) + B > N$,
$v(\delta_2) + A > N$ and $v(\delta_1) + v(\delta_2) > N$.
\end{proof}

\begin{remark} \label{full-disc-remark}
(1)
In the current and previous sections we worked
at one of the ends~$ \ee $, but one can of course also compute
the local expansion of the Frobenius on a residue disc~$ \D $ contained
in~$ \Spec(A) $. Enlarging~$ K $ if necessary, we may assume
there is a point~$ Q : R \to C $ with~$ Q $ in~$ \D $,
corresponding to a parameter~$ t $ on~$ \D $. This gives
an expansion map~$ \ex = \ex_Q : \Adag \to R[[t]] $.
The system~$ H(\S) $ be obtained from~\eqref{globalG}
by applying $ \ex $ to the coefficients now has coefficients
in~$ R[[t]] $, the unique~$ \soll $ in~$ (\pi R[[t]])^{n-1} $ with $ H(\S) = 0 $ equals $ \xi(\sol) $ with
$ \sol $ as in Theorem~\ref{globallift0},
and~$ \ex(\phidagA(\x)) = \ex(\psi(\x)) + \ex(\psi(P)) \soll $
has entries in~$ R[[t]] $.
If we want to compute such expansions up to degree~$ L $
(see Definition~\ref{t-precision})
then it suffices to start with computing~$ \ex(\x) $ up to~$ t^L $.
The corresponding statements hold if we replace~$ R $ with
some~$ \RN R N $. We observe that in both cases there is no need
to precompute how many Newton iterations or multiplications are
to be used.

(2)
For future applications, we also consider~$ \ex(f^{-1}) = \ex(f)^{-1} $
for~$ f $ in~$ A $. If~$ f $ has no zero in~$ \D $ then~$ \ex(f) $
is in~$ R[[t]]^* $, and the situation is clear. 
If~$ f $ has precisely~$ d \ge 1 $
zeroes (counting multiplicities), then~$ \ex(f) $ is
in~$ R[[t]] \cap \Rtab {\a} d $ for some~$ \a \le v(\pi)/d $,
and~$ \ex(f^{-1}) $ is in~$ t^{-d} \Rtab {\a} 0 = \Rtab {\a} {-d} $.
More explicitly, we can write~$ \ex(f) = t^d u - g $ with~$ u $
in~$ R[[t]]^* $, and~$ g $ in~$ R[t] $ of degree less than~$ d $
as well as in~$ \Rtab {\a} d $.
Then
\begin{equation} \label{exfinv}
\ex(f^{-1}) = t^{-d} u^{-1} (1 - g t^{-d} u^{-1})^{-1} =
t^{-d} u^{-1}  \Bigl( 1 + \sum_{m = 1}^\infty g^m t^{-md} u^{-m} \Bigr)
,
\end{equation}
where~$ g t^{-d} u^{-1} $ is in~$ \Rtab {\a} 0 $.

Now suppose we want to compute~\eqref{exfinv} with coefficients
in~$ \RN R N $, and up to a certain degree, assuming that~$ u $ is known up to degree~$ L $ in~$ \RN R N [[t]] $.
Because~$ g $ is of degree less than~$ d $, we have~$ g t^{-d} u^{-1} $
in~$ V_{\a^{-1}} $, so that in~\eqref{exfinv}
we only need to sum up to~$ m = \lfloor \a N \rfloor $.
Then the result in~$ \RN R N [[t]] $ is known up to degree~$ L - (\lfloor \a N \rfloor +1) d $.

There is an alternative estimate, which can work well for~$ d > 1 $.
Using that~$ g $ is in~$ V_c $ for some~$ c > 0 $,
we only need to sum in~\eqref{exfinv} up to~$ m = \lfloor N c^{-1} \rfloor $,
and we know the result in~$ \RN R N [[t]] $ up to degree~$ L - (\lfloor N c^{-1} \rfloor + 1 ) d $.
\end{remark}

\section{Examples of the local Frobenius} \label{local-examples}

In this section we revisit some of the examples in Section~\ref{global-examples}.
In particular, we investigate the case of an hyperelliptic curve as
in Example~\ref{kedlayaexample}, leaving out either the point
at infinity, or all the Weierstrass points.
We work out those cases mostly as an illustration of the differences
between leaving out as few points as possible, or opting for
localizing but imposing $ \phi(x) = x^{\pa} $.

The reader should bear in mind that for those curves,
using the closed formula as in \cite{Ked01}
for $ \phi $ with $ \phi(x) = x^{\pa} $, one can certainly get more precise information
about the expansions $ \ex(\phi(y)) $ than by our general methods.
Also, due to the low degree in $ y $ of the defining equation,
the problem of computing $ \phi(y) $ or its expansion
for this $ \phi(x) $
is of a rather different nature than in the case of a more general
curve.

\begin{example} \label{kedlayaexample2}
Let the notation and assumptions be as in Example~\ref{kedlayaexample}.
Let $ A = R[x,y]/(f) $,
so that we leave out only the point at infinity.
Since $ \ol Q $ has no multiple roots, there exist polynomials $ \ol{a}(x) $ of
degree at most $ 2g $ and $ \ol{b}(x) $ of degree at most $ 2g-1 $
in $ k[x] $ such that $ -\ol{a}(x) \ol{Q'}(x) + \ol b(x) \ol Q(x) = 1  $ in $ k[x] $.
Then
\begin{equation*}
	\ol a~\ol{f_x} + 2^{-1}y\ol b~\ol{f_y} = 1 + \ol b~\ol f
\end{equation*}
in $ k[x,y] $.
We lift $ \ol{a} $ and $ \ol{b} $ to $ a $ and $ b $ of degree
at most $ 2g $ and $ 2g-1 $ in $ R[x] $, so that
$ a $ and $ 2^{-1} y b $ have a pole at infinity of order at most $ 4g $ and $ 6g-1 $ respectively.

Using those, the $ H(S) $ defined in Theorem~\ref{locallift0} becomes
\begin{alignat*}{1}
	H(S) &= \bigl(y(t)^{\pa} + \tfrac{1}{2} y(t)^{\pa} b^\s (x(t)^{\pa} ) S \bigr)^2 -
			Q^{\s}\bigl(x(t)^{\pa} + a^\s(x(t)^{\pa})S\bigr)\\
         &= \sum_{l=0}^{2g+1} H_l S^l\,.
\end{alignat*}
Choosing a parameter $ t $ centred at the missing point,
the expansions $ x(t) $ and $ y(t) $ are in $ t^{-2}R[[t]] $ and 
$ t^{-2g-1}R[[t]] $ respectively. 
So $ H_0 $ is in $ \pi t^{-2\pa(2g+1)}R[[t]] $,
$ H_1 $ is in
$ R^* + \pi t^{-8\pa g}R[[t]] $, and 
$ H_l $ is in 
$ t^{-2\pa(l(2g-1) + 2g+1)}R[[t]] = t^{-2\pa((l-1)(2g-1)+4g)}R[[t]] $
for $ l = 2,\cdots,2g+1 $.

From the above we can assume:
\begin{itemize}
 \item $H_0(t) $ is in~$ \Rtab {8p'ge} {2p'(2g-1)} \cap V_{1/e}$;

 \item $H_1(t) $ is in~$ \Rtab {8p'ge} {0} ^{*}$;

 \item $H_l(t) $ is in~$ \Rtab {8p'ge} {-2p'((2g-1)(l-1) + 4g)} $ for $2 \le l \le 2g + 1$.
\end{itemize}
We can start with $z = 0$ in~$ \Rtab {8p'g} {0} \cap V_{1/e}$ and $H(0)~$ in $\Rtab {8p'g} {2p'(2g-1)} \cap V_{1/e}$,
so that $b = c = \frac{1}{e}$,  $ \nu = 0$, and~$ \omega = 2p'(2g-1)$.
Then
\begin{alignat*}{1}
\mu' - \mu & = \maxp_{2\le l \le 2g+1, 0<k \le l-1} \{ -\frac{-2p'((2g-1)(l-1) + 4g) + 2kp'(2g-1)}{(l-1)\frac{1}{e}}\} = 8p'ge
,
\\
b' & = b = \frac{1}{e}, c' = 2c; \nu' - \nu = (\mu' - \mu)b' = 8p'g
,
\\
\omega' - \omega & = \min_{2 \le l \le 2g+1} \{(l-1) \frac{1}{e}\} (\mu' - \mu) = 8p'g
,
\end{alignat*}
hence the final solution lies in~$\Rtab {16p'ge} {2p'(2g+3)} \cap V_{1/e}$.
\end{example}

\begin{example}
Let us now invert $ 2y $ as in Example~\ref{kedlayaexample},
so that we work
with the open affine corresponding to $ R[x,y,z]/(y^2 - Q(x), 2yz - 1) $
and the missing points are the $ 2g+2 $ Weierstrass points.
The vector $ H(\S) $ defined in Theorem~\ref{locallift0} has entries
\begin{alignat*}{1}
	H_2(\S) & = \bigl(y(t)^{\pa} + z(t)^{\pa} S_2\bigr)^2 - Q^\s(x(t)^{\pa}) 
				= H_{2,0} + H_{2,1} S_2 + H_{2,2} S_2^2 \\
\intertext{and}
	H_3(\S) & = 2 \bigl(y(t)^{\pa} + z(t)^{\pa} S_2\bigr) \bigl( z(t)^{\pa}-
				2z(t)^{3\pa}S_2+z(t)^{\pa} S_3 \bigr)-1
\,.
\end{alignat*}

As in Example~\ref{raisex1example}, 
the first condition involves only $ S_2 $,
and by Lemma~\ref{hensel2}, $ H_2(S_2) = 0 $ has a unique solution $ \ts_2 $
in $ \pi\Rtstar $. To give estimates on $ \ts_2 $, we need to study three
distinct cases.

Case 1: the missing point is the point at infinity. Then the expansions 
$ \ex(x) $, $ \ex(y) $ and $ \ex(z) $ are in $ t^{-2}R[[t]] $,
$ t^{-2g-1}R[[t]] $, and~$ t^{2g+1}R[[t]] $, respectively.

In this case, we have assumptions \footnote{We recall that the notation~$ \Rtab {0} {\b} $ was defined in Notation~\ref{Ralphabetanotation}.}:

\begin{itemize}
 \item $H_{2,0}(t) $ is in~$ \Rtab {0} {-2p'(2g+1)} \cap V_{1/e}$;

 \item $H_{2,1}(t) $ is in~$ R^{*}$;

 \item $H_{2,2}(t) $ is in~$ \Rtab {0} {2p'(2g+1)} $.
\end{itemize}

Starting with $z = 0$ we may set $\mu = \nu = 0$ and $\omega = -2p'(2g+1)$, then $M = 0$ so that the final solution $s_2$ is in~$ \Rtab {0} {-2p'(2g+1)} $. 

Case 2: the missing point is a Weierstrass point $ (a,0) $.
Here we can choose $ y $ to be the local parameter~$ t $, and the expansions~$ \ex(x) $
and~$ \ex(z) $ are in~$ R[[t]] $ and~$ t^{-1}R[[t]] $, respectively.

In this case, the assumptions are:
\begin{itemize}
 \item $H_{2,0}(t) $ is in~$ \Rtab {0} {0} \cap V_{1/e}$;

 \item $H_{2,1}(t) $ is in~$ R^{*}$;

 \item $H_{2,2}(t) $ is in~$ \Rtab {0} {-2p'} $.
\end{itemize}

Again $\mu = \nu = \omega = 0$ and $M = 2p'$, so that the final solution lies in $\Rtab {2p'/e} {-2p'} $. Note that in this case the new estimate is the same as the old one.
\end{example}

\begin{example} \label{weierstrass-example}
In~$ \Z_3[x,y] $, let~$ f = y^2 - h(x)$ for $ h(x) = x^3 - x $.
As in Example~\ref{kedlayaexample}, $ A = \Z_3[x,y,z] / (f, yz - 1) $ describes
an affine part of an elliptic curve over~$ \Z_3 $, and we can
find~$ \phidagA $ with~$ \phidagA (x) = x^3 $ and~$ \phidagA(y) = y^3 \sqrt{1+a} $
for~$ a = \frac{h(x^3) - h(x)^3}{h(x)^3} = 3 x^2 (1-2x^2+x^4)^{-1} $.
Then on the residue disc containing~$ (0,0) $,
with parameter~$ t = y $, we have that~$ a(t) $ is in~$ 3 t^4 \Z_3[[t]] $,
and
\begin{equation*}
\phidagA(y)(t) = t^3 + \frac32 t^7 + \frac{39}8 t^{11} + \frac{315}{16} t^{15} + \cdots \text{ in } \Z_3[[t]],
.
\end{equation*}
\end{example}

\begin{remark}
We discuss smooth planar curves in more detail in Example~\ref{gen-planar-ex}.
\end{remark}
\section{Computing cup products up to a given precision} \label{cup-product-estimates}

In this section we explain how to use the methods that were discussed in
Sections~\ref{expansions} and~\ref{local-frobenius} to get the cup products required
in Method~\ref{cupmethod}, up to a given precision.
More explicitly, fix an end $\ee$ and a local parameter $t$ for the corresponding residue disc~$\D$.
Then the expansion at~$ \ee $ of~$f$ in~$ \Adag $ is of the form $\sum c_nt^n$ in~$\Rtstar $,
and that of~$\o$ in~$ \Omega^1(\Adag/R) $
of the form $\sum c_n't^n \dd t$ in $\Rtstar \, \dd t$.

In Definition~\ref{p-precision} we defined the precision for the coefficients $c_n$,
but we still need the following definition since we are dealing with infinite series in~$t$
and~$ t^{-1} $.

\begin{definition} \label{t-precision}
We shall say that we know~$\sum c_nt^n$ in~$ \Rtstar $ in degrees~$L_0$ to~$L_1$
(up to precision~$ N $)
if we know the coefficients $c_n$ for $L_0 \le n \le L_1$ (up
to precision~$ N $).
\end{definition}

Thus the task can be stated more precisely.
Given~2-forms $\o_1, \dots, \o_n$ that form a basis of~$ \hdr^1(C/R) $,
we need to know the matrix $M_1$ with entries $\sum_\ee \res_{\ee} (\o_j \int \o_i) $ up to precision $N_1$,
and the matrix $M_2$ with entries $\sum_\ee \res_{\ee} ( \phihatA(\o_j) \int \o_i ) $ up to precision~$N_2$.
The required precisions $N_1$ and $N_2$ will be computed in Section~\ref{algorithm-estimates}.

We assume the~1-forms are given as $\o = \sum_{i=1}^n f_i(\x) \, \dd x_i$,
and the first thing is to expand all the forms in terms of~$t$.

\begin{proposition} \label{expdiff}
The expansion map~$\ex: \Rxdag \to \Rtstar $ in~\eqref{expansionmap} induces an expansion map of forms
$\ex: \oplus_i \Rxdag \dd x_i \to \Rtstar \,\dd t$ by
sending $\sum_i f_i(\x) \, \dd x_i$ to~$\sum_i \ex(f_i) \, \dd \ex(x_i)$,
which descends to an expansion map on~$ \Omega_A^\dagger$.
Moreover, if $\ex(x_i)$ is in~$\Rtab {\a} {\b} $,
then $\ex (\dd x_i)$ is in $\Rtab {\a} {\b-1} \, \dd t$.
\end{proposition}

\begin{proof}
Suppose $\ex(x_i) = \sum a_nt^n$, then $\dd \ex(x_i) = \sum na_nt^{n-1} \, \dd t$,
with the coefficient~$ n a_n $ of~$ t^{n-1}$ satisfying~$v(na_n) \ge v(a_n) \ge -\frac{1}{\a} ((n-1) - (\b -1))$.
\end{proof}

Consider two expanded forms~$ \o = \sump {m} {} a_{m-1} t^{m-1} \, \dd t $
and~$ \eta = \sump {m} {} b_{m-1} t^{m-1} \, \dd t $ with trivial residue,
where the prime indicates we leave out the term with $ m = 0 $.
We want to compute
\begin{equation} \label{cupreslocal}
  \res_\ee \Bigl( \eta \int \o \Bigr) = \sump {m} {} \frac{a_{m-1} b_{-m-1}}{m}
\end{equation}
as in~\eqref{serre-rigid-formula}, but 
up to a given precision~$ N \ge 0 $.
For this, we estimate how many terms
we need in the expansions of $ \eta $ and~$ \o $, and up to which
precision.

\begin{lemma} \label{estimates-lemma}
For~$ \eta $ and $ \o $ as above, suppose that~$ \o $ is in~$\Rtab {\a_1} {\b_1} \, \dd t $
and~$ \eta $ is in~$ \Rtab {\a_2} {\b_2} \, \dd t $.
Then
\begin{equation*}
 \res_\ee \Bigl( \eta \int \o \Bigr)
 =
 \sump {m=L_0} {L_1} \frac{\tilde a_{m-1} \tilde b_{-m-1}}{m}
\end{equation*}
up to precision~$ N \ge 0 $, 
where each $ \tilde a_i $ equals $ a_i $, and
each $ \tilde b_i $ equals $ b_i $, up to precision~$ N + \e $,
for $ L_0 $, $ L_1 $ and $ \e $ that are determined as follows.
\begin{itemize}
\item
Find an integer $L_0 \le 0 $  such that his
$-\a_1^{-1} (m - 1 - \b_1) > N + \log_p|m| $ for all $m < L_0$,
and an integer $L_1 \ge 0 $ such that
$-\a_2^{-1} (-m - 1 - \b_2) > N + \log_p|m| $ for all $m > L_1$.

\item
If $ L_0 < 0 $ and
\begin{equation*}
\qquad\quad
\maxp\{- \a_1^{-1} (m-1-\b_1)\} + \maxp\{- \a_2^{-1} (-m-1-\b_2)\}
 > N + v(m)
\end{equation*}
for $ m = L_0 $, replace $ L_0 $ with $ L_0 + 1 $. Repeat this
until the test fails.

\item
Treat $ L_1 $ similarly, but in this case replace $ L_1 $ with $ L_1 - 1 $.

\item
Take $ \e = \maxp\{v(L_0),\dots,v(-1), v(1), \dots, v(L_1)\} $.
\end{itemize}
\end{lemma}

\begin{proof}
In computing $ \res_\ee \bigl( \eta \int \o \bigr) = \sump {m} {} \frac{a_{m-1} b_{-m-1}}{m} $
up to precision $ N $ we can ignore terms with $ v(a_{m-1}) + v( b_{-m-1}) > N + v (m) $.  
Since all~$ b_i $ are in~$ R $, for $ m < L_0 $ one has 
$ v(a_{m-1}) + v(b_{-m-1}) \ge -\a_1^{-1} (m - 1 - \b_1) > N + \log_p|m| \ge N + v(m)$.
For~$ m > L_1 $ one similarly has
$ v(a_{m-1}) + v(b_{-m-1}) > N + v(m) $.

Thus, $ \res_\ee \bigl( \eta \int \o \bigr) $ is equal to $ \sump {m=L_0} {L_1} \frac{a_{m-1} b_{-m-1}}{m} $
up to precision~$ N $. We can then still discard the term for
any remaining~$ m $ if the bounds on the coefficients of an element of
$ \Rtab {\a_1} {\b_1} $ and on those of an element of
$ \Rtab {\a_2} {\b_2} $ gives us that $ v(a_{m-1}) + v(b_{-m-1}) > N + v(m) $.
This is done (if possible) in the second step
for $ m = L_0, L_0 + 1, \dots $, and adjusts~$ L_0 $ accordingly. The third step does this for
$ m = L_1, L_1 - 1, \dots $, and adjusts~$ L_1 $.

For~$ a $ and $ b $ in~$ R $, and
$ \tilde a $ and $ \tilde b $ in $ K $ with
$ v(\tilde a - a) > N + \e $ and $ v(\tilde b - b) > N + \e $
for~$ N + \e \ge 0 $, we
have $ \tilde a $ and $ \tilde b $
in $ R $. Then~$\tilde a \tilde b - a b = (\tilde a -  a) \tilde b - a (\tilde b - b) $
implies that~$ v(\tilde a \tilde b - a b ) > N + \e$.
Therefore~$ \res_\ee \bigl( \eta \int \o \bigr) $
up to precision $ N $ also equals~$ \sump {m=L_0} {L_1} \frac{\tilde a_{m-1} \tilde b_{-m-1}}{m} $
by our choice of~$ \e $.
\end{proof}

Then we see from Lemma~\ref{estimates-lemma} that, 
given $\o_i = \sump {m} { } a_{i, m-1} t^{m-1} \, \dd t $ lying in the set $\Rtab {\a_i} {\b_i} \, \dd t $,
in order to compute $\res_{\ee} ( \o_j \int \o_i) $ up to precision $N_1$, 
we need to know the class of $\o_i$ in the quotient $\RtabN {\a_i} {\b_i} {N_1 + \e} \, \dd t$
in degrees~$L_{i,0}$ to $L_{i,1}$.
However, to compute $\res_{\ee} (\phi(\o_j) \int \o_i) $,
we need to know a priori the set $\Rtab {\cc} {\d} \, \dd t $ in
which every~$\phi(\o_i)$ lies.
In order to have such an estimate,
at first we estimate the subsets $\Rtab {\a} {\b_j} $ where the homogeneous terms of $H(\S)$ lie.
Then we use Proposition~\ref{hensel2} to obtain an~$\Rtab {\mu} {\nu} $
in which~$\sol$ lies.
Now from~$\ex(\phi(\x)) = \ex(\psi(\x)) + \ex(\psi(P))\soll$
as in~\eqref{local-frob-expansion},
together with Lemma~\ref{rtabprop} and Proposition~\ref{expdiff},
we can find an~$\Rtab {\cc} {\d} \, \dd t$ in which every~$\phi(\o_i)$ lies.

Because we know that~$\phi(\o_j) $ is in $ \Rtab {\cc_j} {\d_j} \, \dd t$, 
in order to find $\res_{\ee} (\phi(\o_j)\int \o_i) $ up to precision $N_2$,
we need to find the class of $\phi(\o_j)$ in $\RtabN {\cc_j} {\d_j} {N_2 + \e} \, \dd t $, 
in degrees~$L_{j,0}$ to $L_{j,1}$.
By lemma~\ref{NewtonprecN}, to find $\phi(\o)$ up to precision $N_2 + \e$, 
we need to find everything, in particular, to run the Newton iteration steps, up to the same precision.

We still have the problem how many terms in $t$ of $\sol$ are needed.
Let us reduce it to the following proposition.

\begin{proposition} \label{Prodprec}
Suppose $A_i(t) = \sum_m a_{i,m}t^m$ is in $ \Rtab {\a_i} {\b_i} $ for $1 \le i \le l$.
Then in order to know the $k$-th coefficient of $ \prod_{i=1}^l A_i(t) $ up to precision $N$,
it suffices to know the coefficients of each~$ A_i(t) $ 
up to degree $L_i$, up to precision~$ N $,
where $L_i = k + (\max_{j \not= i} \a_j) N - (\sum_{j \not= i} \b_j) $.
\end{proposition}

\begin{proof}
It suffices to prove this for~$ i = 1 $. Let~$ A(t) = A_1(t) = \sum_m a_m t^m $
and let~$ B(t) = A_2(t) \dots A_l(t) = \sum_{n} b_n t^n $, so
that~$ B(t) $ is in~$ \Rtab {\a'} {\b'} $ with~$ \a' = \max\{\a_2,\dots,\a_l\} $
and~$ \b' = \b_2 + \dots + \b_l $ by Lemma~\ref{rtabprop}.
The coefficient of~$ t^k $ in~$ A(t) B(t) $ is~$ \sum_{m} a_{m} b_{k-m}$, where
$ v(a_{m} b_{k-m}) \ge v(b_{k-m}) > N $ if~$ - \a'^{-1} (k-m - \b') > N $, i.e., if~$ k + \a' N - \b' < m $.
So we only need to know $ a_m $ for~$ m \le k + \a' N - \b' $.
Because all coefficients involved are in~$ R $, we only need to know the
classes of those~$ a_k $ in $ \RN R N $.
\end{proof}

\begin{remark}
Given $ \o $ and~$ \eta $ in Proposition~\ref{estimates-lemma},
in order to compute~$ \res_\ee \bigl( \o \int \eta \bigl) $ up
to a specified precision, at each end we have to compute the expansion
of~$\phi(\eta)$ up to a certain precision~$ N $, and up to a given
degree.

If~$ \eta = \sum_i f_i(\x) \, \dd x_i $, then
using Propositions~\ref{expdiff} and~\ref{Prodprec}, we see that
we have to compute each~$ \ex(\phihatA(x_i)) $ up to precision~$ N $,
and up to some degree~$ L_i $.
Applying Proposition~\ref{Prodprec} again to~\eqref{local-frob-expansion},
we see that we have to find~$ \soll $ up to precision~$N$ and up to 
some degree~$ L $.

By Proposition~\ref{NewtonprecN} we can  carry out the Newton iteration
with coefficients in~$ \RN R N $, and $ \soll $ is~$ z_k $ up
to precision~$ N $ if~$ k > \log_2(eN)  + 1 $.
Thus it remains to determine up to which degree in~$ t $ the coefficients in the system $H(\S)$
needs to be known in order to get~$ z_k $ up to degree~$ L $.

From the data of the system $H(\S)$ for~$ z = 0 $,
we can determine the parameters~$\mu_i, \nu_i, \o_i$ for which $z_i$ is in~$ \Rtab {\mu_i} {\nu_i} $,
and~$H(z_i)$ is in~$ \Rtab {\mu_i} {\o_i} $.
From~$z_{i+1} = z_i - \Jac_H^{-1}(z_i) H(z_i)$ we can determine how many terms are needed for $H(z_{k-1})$
by using Proposition~\ref{Prodprec}.
Then we can decide inductively how many terms are needed for~$H(\S)$.
\end{remark}

\section{Bounding denominators in the residues at the ends.} \label{eisp}

In this section we bound from below the valuations of the elements
in~$ M_2 $ in Algorithm~\ref{mainalgorithm}, when the basis of~$ \hdr^1(C/R) $
is obtained from Proposition~\ref{compdr} for~$ D $ with~$ |D| \cap \Spec(A) = \emptyset $.
We do this by proving that the valuation of the contribution
of each end to the pairing in~\eqref{serre-over-L} is bounded below by~$ - \lceil \log_p(e/p) \rceil $.

If~$ e < p $ then~$ \hcr^1(C_k/W(k)) \otimes_{W(k)} R = \hdr^1(C/R) $
in~\eqref{hcr-hrig-hdr} by \cite[Theorem~2.8.1]{Ber-Ogu83}. It
follows that this last group is mapped to itself under~$ \myphicr {\pa} \otimes \s $ under that
condition, but our results below show this also holds if~$ e = p $.
Moreover, even if~$ e < p $ then it is not obvious that the local contributions
in~\ref{serre-over-L}, if we use a basis as above, have non-negative valuations.
For a further discussion we refer to Remark~\ref{h-remark}.

\begin{definition}
We say that~$ \o = \sum_n a_n t^n \dd t $ in~$ \Rtstar \dd t $
has integrable polar part if~$ a_{-1} = 0 $ and~$ (n+1)^{-1} a_n $ is in~$ R $
for all~$ n < -1 $.
\end{definition}

Note that~$ \o $ has integrable polar part precisely when it
is of the form~$ \dd h + \eta $ with~$ h $ in~$ \Rtstar $ and~$ \eta $ in~$ R[[t]] \dd t $.
This also shows that the notion is independent of the choice of the
parameter~$ t $, as another choice is of form~$ t' = t u $ for~$ u $ in~$ R[[t]] $.

\begin{lemma} \label{polar-lemma}
(1)
If~$ \o $ and~$ \eta $ have integrable polar parts,
then~$ \res_\ee \Bigl( \eta \int \o \Bigr) $ is in~$ R $.

(2)
If in a Zariski neighbourhood~$ U $ in~$ C $ of~$ q $ in the closed
fibre, $ \o = \dd f + \eta $, with~$ \eta $ in~$ H^0(U, \Omega_{U/R}^1 ) $ and~$ f $
a rational function that is generically defined on~$ C_k $, then~$ \ex(\o) $
has locally integrable polar part. Here~$ \ex $ is the expansion
map as in~\ref{expansionmap}.
\end{lemma}

\begin{proof}
(1)
Writing~$ \o = \sump {m} {} a_{m-1} t^{m-1} \dd t $ and~$ \eta = \sump {m} {} b_{m-1} t^{m-1} \dd t $
with all~$ a_{m-1} $ and~$ b_{m-1} $ in~$ R $, in the formula~$ \sump {m} {} \frac{a_{m-1} b_{-m-1}}{m} $
for the residue given (cf.~\eqref{cupreslocal}) every term is in~$ R $
because~$ \frac1m a_{m-1} $ is in~$ R $
for~$ m < 0 $ and~$ - \frac1m b_{-m-1} $ is in~$ R $
for~$ m > 0 $.

(2)
The polar part of~$ \ex(\o) $ is that of~$ \dd \ex(f) $, and~$ \ex(f) $ is in~$ \Rtstar $.
\end{proof}

\begin{proposition} \label{ep-prop}
Let~$ h = \max\{0, h' \} $,
where~$ h' = e j - p^j $ for~$ j = \lceil \log_p(e/(p-1)) \rceil $.
Define~$ \phi : \Rtstar \, \dd t \to \Rtstar \, \dd t  $ by replacing~$ t $
with~$ f + \pi g $ for some~$ f $ in~$ t R[[t]] $ and~$ g $ in~$ \Rtstar $.
If~$ \o $ in~$ \Rtstar \dd t $ has integrable polar part,
then~$ \pi^h \phi(\o) $ has integrable polar part.
\end{proposition}

\begin{proof}
Write~$ \eta = \sump m {} a_{m-1} t^{m-1} \dd t $, so
\begin{equation*}
\phi(\eta) = \sump m {} a_{m-1} \phi(t)^{m-1} \dd \phi(t) = \sump m {} m^{-1} a_{m-1} \dd \phi(t)^m 
.
\end{equation*}
For~$ m < 0 $, each~$ m^{-1} a_{m-1} $ is in~$ R $ by assumption,
so~$ a_m \phi(t)^m \dd \phi(t)$ has integrable polar part, and
the same holds if we multiply it by~$ \pi^h $ because~$ h \ge 0 $.
Now it suffices to show that the polar part of~$ m^{-1} \pi^{h'} \phi(t)^m $ for~$ m > 0 $
has coefficients in~$ R $. That follows from the next lemma.
\end{proof}

In the next result we let~$ v_\pi = e v_p $, so~$ v_\pi(\pi) = 1 $.

\begin{lemma}\label{betterval}
Let $v_{\pi } $ of a polynomial be the minimal valuation of its coefficients.
Then,
for all~$ m > 0 $, we have~$  v_{\pi } \left( (x+\pi y)^m - x^m \right) \ge v_{\pi }(m) - h' $.
\end{lemma}

To prove this we first prove what is probably a well-known formula about the valuations of binomial coefficients.
The situation is much better than our bound gives because it is often negative, and all binomial coefficients are integers.

\begin{lemma}\label{binomlem}
  We have $v_p(\binom{m}{k}) \ge v_p(m)-v_p(k) $ for~$k  = 1, \dots, m $.
\end{lemma}

\begin{proof}
Recall Legendre's formula $v_p(n!) = \sum_{i=1}^{\infty} \lfloor\frac{n}{p^i}\rfloor $. Now that~$ \lfloor x+y\rfloor- \lfloor x\rfloor-\lfloor y\rfloor $
for real numbers~$ x $ and~$ y $ is either $0 $ or $1 $. In particular, it is~1 if $x+y $ is an integer but $x $ is not. Taking $x_i = k / p^i $ and $y_i = (m-k) / p^i $ we see that the expression
\begin{equation}
  v_p\left( \binom{m}{k} \right) =  \sum_{i=1}^{\infty} \bigl( \lfloor x_i+y_i\rfloor- \lfloor x_i\rfloor-\lfloor y_i\rfloor \bigr)
\end{equation}
is a sum of $0 $s and $1 $s and we get a $1 $ for every $i $ such that $p^i \mid m $ and~$p^i \nmid k $.
\end{proof}

\begin{proof}[Proof of Lemma~\ref{betterval}]
By Lemma~\ref{binomlem} we have
\begin{align*}
  v_{\pi } \left( (x+\pi y)^m - x^m \right) &= v_{\pi } \left( \sum_{i=1}^{m} \binom{m}{i} (\pi y)^i x^{m-i} \right)\\  &\ge \min_{i=1}^m \left\{ i+ e (v_p(m) -v_p(i)) \right\}\\ &= v_{\pi }(m) +  \min_{i=1}^m \left\{ i- e v_p(i) \right\}\\ &\ge v_{\pi }(m) +  \min_{i\ge1} \left\{ i- e v_p(i) \right\}
.
\end{align*}
 This minimum values of~$ i- e v_p(i) $ for~$ i \ge 1 $
is attained for some~$i= p^l $, so it suffices to minimise~$ f(l) = p^l - el $
for~$ l \ge 0 $. We have $f(l+1)-f(l)= p^l(p-1)-e $, an increasing
function on~$ \R $ that is zero for~$ l = \log_p(e/(p-1)) > -1 $.
\end{proof}

\begin{remark}\label{hincreasing}
(1)
The bound of Lemma~\ref{betterval} is best possible. In order
to see this, note that the minimal value in its proof was obtained at $i=p^l $.
Using Kummer's theorem one sees that the inequality in Lemma~\ref{binomlem} is an equality when~$k=p^l $ and~$m=p^n $ with $n\ge l $

(2)
If, in the proof above, we write $f_e(l) $ to stress the dependence on $e $, then clearly $f_e(l) $ is non-increasing in $e $. Therefore, $-h'= - h_e' $,
which is just the minimal value of $f_e(l) $ over all $l\ge 0 $, is also non-increasing in $e $, so $h_e' $ itself is non-decreasing in $e $.
In fact, it is increasing once $j\ge 1 $, i.e., starting at $e=p-1 $.
So, as an easy estimate we have $h'=h > 0 $ for~$e \ge p+1 $.
\end{remark}

We want to apply Proposition~\ref{ep-prop} to~$ \ex( \phihatA(\o)) $ for~$ \o $
in~$ H^0(C, \I(D)) $ as in Proposition~\ref{compdr}, 
under the assumption that~$ |D| \cap \Spec(A) = \emptyset $.
But, as mentioned at the beginning of Section~\ref{expansions}, the
local parameter~$ t $ is perhaps only defined after
replacing~$ K $ with a finite extension~$ L $ (which we suppressed
from the notation above), and so~$ \phihatA(t) $ is not defined.
We therefore enlarge~$ K $ to a suitable~$ L $,
with valuation ring~$ R_L $, in such a way that~$ \phihatA $
extends to~$ A \otimes_R R_L $, and is defined on a suitable~$ t $.

\begin{lemma}
Let~$ L/K $ be finite and unramified with residue field extension~$ l/k $. Then there is a unique~$ \t $
in~$ \Aut(L/\Qp) $ such that~$ \t_{|K} = \s $ and~$ \ol{\t} $ in~$ \Gal(l/\F_p) $
is given by raising to the power~$ \pa $.
\end{lemma}

\begin{proof}
If~$ \t_1 $ and~$ \t_2 $ both satisfy the conditions, then in~$ \Aut(K/\Qp) $ both restrict to~$ \s $
on~$ K $, and~$ \t_1 \t_2^{-1} $ is in~$ \Gal(L/K) $. This group
is isomorphic to~$ \Gal(l/k) $ because~$ L/K $ is unramified, so~$ \ol{\t_1} \, \ol{\t_2}^{-1} = \id_l $,
and~$ \t_1 = \t_2 $.
In order to show a suitable~$ \t $ exists, let~$ E \subseteq K $
be the fixed field of~$ \la \s \ra $. Then~$ L/E $ is finite
Galois, so there is a~$ \rho $ in~$ \Gal(L/E) $ such that~$ \rho_{|K} = \s $.
Let~$ F_{\pa} $ in~$ \Gal(l/\F_p) $ denote the map given by raising
to the power~$ \pa $. Then~$ \ol{\rho}_{|k} = \ol{\s} = F_{\pa|k} $,
hence~$ \ol{\rho}^{-1} F_{\pa} $ is in~$ \Gal(l/k) $.
We can then multiply~$ \rho $ by an element in~$ \Gal(L/K) \simeq \Gal(l/k) $
in order to obtain a~$ \t $ of~$ \Gal(L/E) $ satisfying both conditions.
\end{proof}

For~$ L/K $ finite and unramified, and~$ \t $ as in the lemma,
we can extend~$ \phihat $ from an~$ \s $-semilinear endomorphism of~$ \Rx $
to a~$ \t $~semilinear endomorphism~$ \phihatt $ of~$ \Rtildex $
by using~\eqref{eq:phiact} but with~$ \s $ replaced with~$ \t $
in the definition of~$ \psi $.
It induces a map~$ \phi_{\t,\hat A_{\tR}} $ on~$ \Adag_{\tR} = \Adag \otimes_R \tR $
that restricts to~$ \phihatA $ on~$ \Adag $, and reduces to raising
to the power~$ \pa $ on~$ \Adag_{\tR} / (\pi) = \Abar \otimes_k l $.
Moreover, because of the uniqueness of~$ \t $ for a given~$ L $,
these extensions are compatible if we replace~$ L $ with a larger
unramified finite extension of~$ K $.

By choosing a suitable~$ L/K $ that is finite and unramified,
we can make sure one (or all) of the residue discs corresponding
to~$ \ee $ contain and~$ L $-rational point. It gives rise to
a section~$ \tR \to C_{\tR} $ that enables us to choose a parameter~$ t $
for this disc, as element in~$ \O_{C_{\tR}, q} $ where~$ q $
in~$ C_l $ is the closed point in the image of~$ Q $.

In order to avoid cumbersome notation, below we again write~$ R $, $ K $,
$ \s $, and~$ C $ for~$ \tR $, $ L $, $ \t $ and~$ C_{\tR} $.

Now that we have~$ t $ in~$ \O_{\eta,C} = A_{(\pi)} $, where~$ \eta $
is the generic point of the closed fibre~$ C_k $, we have to
give meaning to~$ \phihatA(t) $ and its expansion in~$ \Rtstar $.
For this, consider the commutative diagram
\begin{equation*}
\xymatrix{
A \ar[r] \ar[d] & \Adag \ar[d]
\\
A_{(\pi)} \ar[r] & \Adag_{(\pi)}
.}
\end{equation*}
Around~\eqref{expansionmap} we have defined the expansion
map~$ \ex $ from~$ A $, $ \Adag $ and~$ A_{(\pi)} = \O_{\eta,C} $
to~$ \Rtstar $, where we used Proposition~\ref{expest} to replace~$ \Rthat $
with~$ \Rtstar $ as the target.
We can get all of those by combining the above diagram with an expansion map~$ \ex : \Adag_{(\pi)} \to \Rtstar $,
which we obtain from the one on~$ \Adag $, as follows. We have~$ \Adag = A + \pi \Adag $,
so if~$ a $ in~$ \Adag $ is not in~$ \pi \Adag $, then~$ a = c + \pi b $
with~$ b $ in~$ \Adag $ and~$ c $ in~$ A \setminus \pi A $.
We saw during the construction of the expansion map on~$ A_{(\pi)} $
that then~$ \ex(c) $ is in~$ \Rtstar^* $. The same then holds
for~$ \ex(a) $, so that~$ \ex : \Adag \to \Rtstar $ extends uniquely
to an expansion map~$ \ex : \Adag_{(\pi)} \to \Rtstar $.

Similarly, we saw in Corollary~\ref{phidagA-phihatA} that~$ \phidagA : A \to \Adag $
induces a map~$ \Adag \to \Adag $ that induces the~$ \pa $-th
power map on~$ \Abar = A/(\pi) = \Adag/ (\pi) $. So if~$ a $
is in~$ \Adag \setminus \pi \Adag $ then~$ \phidagA $ is in the
same set, hence~$ \phidagA $ induces a map~$ \phidagA : \Adag_{(\pi)} \to \Adag_{(\pi)} $
that reduces to the~$ \pa $-th power map on~$ \Frac(\Abar) = k(C_k) $.

We now have the following. Note that we can view~$ t $ here as in~$ A_{(\pi)} $,
or in~$ \Adag_{(\pi)} $.

\begin{proposition} \label{phidagAt-prop}
Let the notation be as above. Then we have~$ \phidagA(t) = t^{\pa} + \pi a $ for
some~$ a $ in~$ \Adag_{(\pi)} $.
\end{proposition}

\begin{lemma} \label{phit-lemma}
Let~$ \phit : \Rtstar \to \Rtstar $ be the homomorphism obtained
by letting~$ \s $ act on the coefficients and replacing~$ t $ with~$ \tilde t = \ex(\phidagA(t)) = t^{\pa} + \pi \ex(a) $,
with~$ a $ as in the previous proposition. Then~$ \ex \circ \phidagA $
and~$ \phit \circ \ex $ induce the same homomorphism from~$ A_{(\pi)} $ to~$ \Rtstar $.
\end{lemma}

\begin{proof}
Let~$ b $ be in~$ \O_{q,C} $ and, for any~$ N \ge 0 $ write~$ b = r_0 + r_1 t + \dots + b_N t^N + t^{N+1} c $
with the~$ r_i $ in~$ R $ and~$ c $ in~$ \O_{q,C} $. Then
\begin{equation*}
\ex \circ \phidagA(b) =  \s(r_0) + \s(r_1) \tilde t + \dots + \s(b_N) \tilde t^N + \tilde t^{N+1} \ex \circ \phidagA(c) 
.
\end{equation*}
Letting~$ N $ grow we see that the right-hand side
converges coefficientwise to~$ \phit \circ \ex (b) $.
Combining Lemma~\ref{expansion-lemma} with Proposition~\ref{expest},
we find that if~$ b $ is not in~$ \pi \O_{q,C} $
then~$ \ex(b) $ is in~$ \Rtstar $, so the same holds for~$ \phit \ex (b) = \ex \circ \phidagA (b) $.
There the two ring homomorphisms extend to the same one from~$ \left( \O_{q,C} \right)_{(\pi)} = A_{(\pi)} $
to~$ \Rtstar $.
\end{proof}

Because applying~$ \s $ to the coefficients in~$ \Rtstar \dd t $ has
no influence on the notion of having integrable polar part, and
Lemma~\ref{phit-lemma} extends to~$ \Omega_{A_{(\pi)}/R}^1 $,
using Lemma~\ref{polar-lemma} as well as Propositions~\ref{ep-prop} and~\ref{phidagAt-prop},
we obtain the following. Note that here~$ R $ is the original~$ R $,
not the valuation ring in the extension(s)~$ L $ of~$ K $ used above
in which the mentioned contributions of the various ends lie.

\begin{corollary} \label{h-cor}
Let~$ \o $ and~$ \eta $ be in~$ H^0(C, \I(D)) $ as in~Proposition~\ref{compdr},
and assume that~$ |D| \cap \Spec(A) = \emptyset $. Let~$ h $
be as in Proposition~\ref{ep-prop}.
Then the contribution from each end~$ \ee $ to~$ \pi^h \pair{ \o, \phidagA(\eta) } $
in~\eqref{serre-over-L} has non-negative
valuation. In particular, if we use such forms to obtain a basis
of~$ \hdr^1(C/R) $, then the entries of~$ \pi^h M_2 $ in Algorithm~\ref{mainalgorithm}
are in~$ R $.
\end{corollary}

Combining this corollary with Remark~\ref{hincreasing} we obtain
the following.

\begin{corollary} \label{stable-cor}
If~$  e \le p $ then~$ \hdr^1(C/R) $ is stable under the action of~$ \myphicr {\pa} \otimes \s $
in~\eqref{robswish1}, under the identifications in~\eqref{hcr-hrig-hdr}.
\end{corollary}

\begin{remark}
If~$ e > p $ then~$ p $ is small in practice, and the additional
precision in Algorithm~\ref{algorithm} (see Section~\ref{algorithm-estimates})
for such~$ p $ is not an obstable.
\end{remark}

\begin{remark} \label{h-remark}
The isomorphisms in~\eqref{hcr-hrig-hdr}
gives rise to two potentially different~$ R $-structures on $ \hdr^1(C_K/K) $:
the image~$ L_\dr $ of~$ \hdr^1(C/R) $, and the
image~$ L_\cry $ of~$ \hcr^1(C_k/W(k)) \otimes_{W(k)} R $,
with the last one stable under the action of~$ \myphicr {\pa} \otimes \s $.
By~\cite[Theorem~2.8.1]{Ber-Ogu83}
we have $ p^j L_\dr \subseteq L_\cry $
and~$ p^j L_\cry \subseteq L_\dr $ for~$ j =  \lceil \log_p(e/(p-1)) \rceil $
as in Proposition~\ref{ep-prop}.
In particular, for~$ e = 1, \dots, p-1 $, we have~$ L_\cry = L_\dr $,
(as already proved in~\cite[Theorem V.2.3.3]{Ber71}),
which implies that~$ L_\dr $ is stable under~$ \phidagA $.

Now let~$ \o $ and~$ \eta $ be as in Corollary~\ref{h-cor}. Then~$ p^j \eta $,
hence also~$ \phidagA (p^j \eta) $, is in~$ L_\cry $,
and~$  p^j \phidagA (p^j \eta) $ is in~$ L_\dr $ again.
Therefore~$ p^{2 j} \pair{ \o, \phidagA(\eta) } $ is in~$ R $
by Remark~\ref{detrem}.
However, this does not say anything about the contributions to
the pairing at each end.
Also, because~$ v_p(p^{2j}) = 2 j > j - e^{-1} p^j = v_p(\pi^{h'}) $
for all~$ j \ge 0 $, the result in Corollary~\ref{h-cor} is always
stronger if~$ h' \ge 0 $, i.e., for~$ e \ge p $ (see Remark~\ref{hincreasing}(2)).
For~$ e = 1, \dots, p-1 $ we have~$ h' = -1 $ but~$ h = j = 0 $,
and the two estimates coincide.
\end{remark}

In \cite[Remark~2.10]{Ber-Ogu83}
examples are given with~$  e > p $ where~$ L_\cry $ and~$ L_\dr $
as in Remark~\ref{h-remark} do not coincide.
Based on Corollary~\ref{stable-cor}, we pose the following.

\begin{question}
Is~$ \hdr^1(C_k/R) = \hcr^1(C_k/W(k)) \otimes_{W(k)} R $ in~\eqref{hcr-hrig-hdr} if~$ e = p $?
\end{question}

\section{Estimates for the algorithm} \label{algorithm-estimates}

In this section we provide estimates for performing Algorithm~\ref{mainalgorithm}.
As indicated there, we assume given a suitable~$ R $-basis of~$ \hdr^1(C/R) $,
but in Remark~\ref{ComputewithKbasis} we shall discuss how the estimates have to be modified
if we use a~$ K $-basis of~$ \hdr^1(C_K/K) $.

As mentioned just before Method~\ref{zeta-method}, the zeta function of $ C_k $
is obtained as~$ Z(T) = \frac{P_1(T)}{(1-T)(1-qT)} $, where~$ P_1(T) = a_0+\dots+a_{2g}T^{2g} $ in~$ \Z[T] $
equals~$ \det(1-T M') $, and~$ q $ is the cardinality of~$ k $.%
\footnote{The reader beware that in this section and Section~\ref{complexity} we have~$ q = |k| $, but in other sections it denotes a closed point of~$ C_k $.}%
Then, according to~\cite[p.507]{Weil49} or \cite[Theorem~12.6]{Mil80}, we have~$ a_0=1 $ and $ a_{2g-i}=q^{g-i}a_i $ for $i=0,\dots,g$,
so that we only need to know $ a_1, \dots, a_g $.
In fact, if
\begin{equation*}
P_1(T) = a_0+\dots+a_{2g}T^{2g} = \det(1 - T M' ) = \prod_{i = 1}^{2g} (1 -\a_i T)
\end{equation*}
in~$ \mathbb{C}[T]$
then~$\Tr(M'^j) = \sum_{i=1}^{2g} \a_i^j$,
and~$ j a_{j} = - \sum_{i = 1}^j a_{j - i} \Tr(M'^i)$ for~$ j \ge 1 $
because of the Newton identities. With~$ a_0 = 1 $, and all~$ |\a_j| = q^{\frac 12} $,
we find that~$ \Tr(M'^j) $ for~$ j \ge 0 $ is an integer of absolute
value at most~$ 2 g q^{\frac j2} $.
Therefore we can obtain~$ P_1(T) $ if we compute the integers~$\Tr(M'^j)$ for~$ j =  1, \dots, g $
up to precision~$ \tilde N $
where $\tilde N$ is the largest integer less than or equal to $\log_p 2\lfloor2gq^{\frac{g}{2}}\rfloor$.

By Remark~\ref{detrem} and Corollary~\ref{h-cor}, the entries of~$ M_1 $
and~$ M_1^{-1} $ are in~$ R $, and those of~$ M_2 $ are in~$ \pi^{-h} R $
with~$ h $ as in Proposition~\ref{ep-prop}, so that~$ M = M_1^{-1} M_2 $
has entries in $ \pi^{-h}R $.
Since~$ M' = \sigma^{l-1}(M) \times \dots \times \sigma(M) \times M $,
it suffices to know the entries of $M$ up to precision~$ N = \tilde N + g(l-1)he^{-1}$.
Using again that~$M = M_1^{-1}M_2$, we see from Lemma~\ref{lemmaprecision}
that for this it suffices to determine~$M_1$ up to precision $N_1' = N + h e^{-1}$ and $M_2$ up to precision $N_2' = N$.
Finally, using Lemma~\ref{estimates-lemma}, we know that in order to know the entries of $M_1$
up to precision $N_1'$, we need to know the forms $\o_i$ up to precision $N_1 = N_1' + \e_1$
where $\e_1$ is obtained from the same lemma.
Similarly, we need to know the forms $\o_j$ and $\phi(\o_j)$ up to precision $N_2 = N_2' +\e_2$.

\begin{remark}
The matrix~$ M $ is also important for Coleman integration
on~$ C_K $. We note that the above estimates still hold if we start
with a (potentially higher) desired precision~$ \tilde N $ for the entries of~$ M' $.
\end{remark}

The following algorithm shows how to get the zeta function by computing the matrices $M_1$ and $M_2$ up to high enough precision.
In it, we abbreviate~$ \phihatA $ to~$ \phi $.

The algorithm has many steps, but almost all of them are involved
estimates on the required precision and number of terms in the
expansions, before we actually compute the expansions and plug
them into the actual computation.

\begin{assumption}
We assume that the form~$ \o_1, \dots, \o_{2g} $ in the input
of the algorithm have integrable polar part, as discussed in
Section~\ref{de-rham}, i.e., we assume they lie in a suitable~$ H^0(\I(D)) $.
\end{assumption}

As mentioned in Lemma~\ref{polar-lemma},
this implies that all residues~$ \res_\ee(\o_j (\int \o_i) ) $, as given in~\eqref{cupreslocal},
have non-negative valuation.

\begin{algorithm} \label{est-algorithm}
Computation of the matrices $M_1$ and~$ M_2 $.

\medskip

\noindent
{\bf Input}
\begin{itemize}
\item The required precision~$ N = \lfloor \log_p 2\lfloor2gq^{\frac{g}{2}}\rfloor \rfloor  + g(l-1)h e^{-1} $ for the entries of~$ M $,
and the required precision $N_1'  = N + h e^{-1} $ for those
in~$M_1$, and $N_2' = N $ for those in~$M_2$.

\item The input of Algorithm~\ref{mainalgorithm}.

\item The matrix~$ P $ in~\eqref{congeq2} and the resulting system $G(\S)$ in~\ref{globalG}.
\end{itemize}

\medskip
\noindent
Then for each end $\ee$ with local parameter $t = t_\ee$ do the following.

\medskip

\noindent
{\bf Step 1: computing the local contribution to the entries of $M_1$}

\begin{enumerate}
\item
Compute the parameters $\lambda_i, \theta_i$ such that the expansion $\ex(x_i)$ lies in $\Rtab {\lambda_i} {\theta_i} $
(see Remark~\ref{expansion-slopes}).

\item
Use Lemma~\ref{rtabprop} and the~$ \lambda_i $ as well as the~$ \theta_i $
to find $ \a_j $ in $ \realpos $ and $ \b_j $ in $ \R $
such that $ \o_j $ has its expansion in $ \Rtab {\a_j} {\b_j} \, \dd t $.

\item Use Lemma~\ref{estimates-lemma} to determine the precision $N_1 = N_1' + \e_1$ of the~$\ex(\o_i)$
needed to compute the~$ \res_{\ee} (\o_j \int \o_i) $ up to precision~$ N_1' $,
as well so the degrees $L_i$ for each~$\o_i$ in $t$ required for this.

\item Use Lemma~\ref{Prodprec} to determine the required number of terms~$ K_j $ in the expansion of each~$ x_j $
up to precision~$ N_1 $ in order to compute the~$ \ex(\o_i) $
in the previous step.

\item
Use Newton iteration to compute the~$ \ex(x_j) $ up to the required
precision and terms, using Remark~\ref{expansion-slopes}.

\item Compute the~$\o_i = \sum_m \ol{a}_{i,m}t^m \, \dd t $ with coefficients in~$\RN R {N_1} $, up to degree $L_i$.

\item Compute the value of each of $\res_{\ee} (\o_j \int \o_i) $
as $\sump m {} \frac{\ol{a}_{i, m-1}\ol{a}_{j, -m-1}}{m}$.
\end{enumerate}

\medskip

\noindent
{\bf Step 2: computing the local contribution to the entries of $M_2$}

\medskip

\noindent{\it A: preliminary estimates}
\begin{enumerate}
\item 
Based on Step~1(1),
use Proposition~\ref{expest} to estimate the parameters $\tilde\a$ and~$\tilde \b_j$ such that the homogeneous terms of degree $j$ of $H(\S)$ are in $\Rtab {\tilde \a} {\tilde \b_j} $.

\item Use Proposition~\ref{hensel2} to determine the parameters $\tilde \cc$ in $ \realpos $ and $\tilde \d$ in $\R$ such that the components of $\soll$,
and of all the Newton iterates starting from~$ z_0 = 0 $, lie in $\Rtab {\tilde \cc} {\tilde \d} $.

\item Use the equality 
$ \ex(\phi(x_i)) = x_i(t)^p + \sum_{j=r+1}^n \ex(\psi(P_{i,j})) \tilde s_i $ to determine 
$ \tilde\lambda_i $ in $ \realpos $ and $ \tilde \theta_i $ in $ \R $ such that each
$ \phi(x_i) $ has its expansion in $ \Rtab {\tilde\lambda_i} {\tilde\theta_i} $.

\item Use Lemma~\ref{rtabprop} and~$ \tilde \lambda_i$ as well as $\tilde \theta_i$
to find $\cc_i$ in $\realpos$ and $\d$ in $\R$ such that each $\phi(\o_i)$ has its expansion in 
$\Rtab {\cc_i} {\d_i} \, \dd t$.

\item Use this, and the result in Step~1(3) in Lemma~\ref{estimates-lemma} to
determine the precision $N_2 = N_2' + \e_2$ of the~$\ex(\o_i)$
and~$\ex(\phi(\o_i))$
required for the calculation of the~$\res_\ee (\phi(\o_j) \int \o_i) $
up to precision~$ N_2' $, 
as well as the degrees $\tL_i$ for each~$\o_i$ in $t$,
and~$L_i'$ for each~$\phi(\o_i)$, required for this.
Later steps use coefficients in~$ \RN R N_2 $.
\end{enumerate}

\noindent{\it B: estimates for the Newton iteration}
\begin{enumerate}
\item
Compute the number of steps $k = [\log_2 N_2] + 1$ of Newton iteration needed for computing~$ \soll $
up to precision~$ N_2 $.

\item
Use the output of A(2) and Proposition~\ref{Prodprec} to determine the
number~$ \bL $ of terms in~$ t $ that
are needed in~$ \soll $ in order to compute the~$ \ex(\phi(x_j)) $
from~\eqref{eq:phiact}.
Here, and in later steps, increase each~$ K_j $ if necessary.

\item Using that~$z_{i+1} = z_i - \Jac_H^{-1}(z_i) H(z_i)$ and Proposition~\ref{Prodprec},
starting with~$z_k = \soll $, determine the degree $\bL_i$ in $t$ needed for computing $z_i$
for~$ i= k, \dots, 1 $.

\item Using that $z_1 = -\Jac_H^{-1}(0) H(0)$, Proposition~\ref{Prodprec}, and~$ \bL_1 $, determine the degree~$\bK$ 
in~$ t $ for the coefficients in~$ H(\S)$ that is needed to compute~$\soll$
up to degree~$ \bL $.
\end{enumerate}

\medskip

\noindent
{\it C: actual calculations}
\begin{enumerate}
\item Determine the~$ \ex(x_j) $ up to degree~$ K_j $.
(This may require recalculating with higher precision and/or
more terms than in Step~1(5) by more Newton iterations.)

\item Compute each $\o_i = \sum \ol{a}_{i,m}t^m \, \dd t$ up to degree $\tL_i$.

\item Apply the expansion map $\ex$ to the system $G(\S)$ as in~\ref{globalG} to get the system $H(\S)$, where each term of $H(\S)$ is computed up to degree $\bK$ as a series in~$t$.

\item Use Newton iteration to compute the solution $\soll$ of $H(\S) = 0$ in~$ k $ steps, starting from~$ z = 0 $.

\item Find each~$\phi(\o_i) = \sum  \ol{c}_{j,m} t^m \, \dd t $ up to degree $L_i'$.

\item
Compute $ \res_\ee (\phi(\o_j) \int \o_i) $ up to precision $N_2'$
as $ \sump m {} \frac{\ol{a}_{i,m-1} \ol{c}_{j,-m-1}}{m} $. 
(It may have negative valuation.)
\end{enumerate}

\medskip

\noindent
{\bf Output:}
the entry at position~$(i,j)$ of $M_1$ as~$ \sum_\ee \res_\ee (\o_j \int \o_i) $,
and that of~$ M_2$ as~$ \sum_\ee \res_\ee (\phi(\o_j) \int \o_i) $.
\end{algorithm}

Finally, we are able to compute the zeta function.

\begin{algorithm} \label{computezeta}

\noindent
{\bf Input:} the output of the above two algorithms.

\begin{enumerate}
\item
Compute the product $ M_1^{-1}M_2 $ corresponding to the matrix $ M $ of the
action of the $\s$-linear Frobenius, up to precision $ N $.

\item 
With $ q = p^l $, compute 
$ M' = \s^{l-1}(M)\times \s^{l-2}(M)\times \cdots \times \s(M)\times M $, 
the matrix of the action of the linear Frobenius, up to precision $ N' $.

\item
Lift the coefficients $ \ol{a}_1,\dots, \ol{a}_g $ of the characteristic
polynomial 
$ \sum_{i=0}^{2g} \ol{a}_i T^i $ of $ M' $ to the unique integers~$ a_i $
satisfying
$ |a_i| \le {2g \choose i} q^{\frac i 2} $.

\item
Set $ a_0 = 1$ and compute $ a_{2g-i} $ as $ q^{g-i} a_i $ for $ i=0,\dots,g $.
\end{enumerate}

\medskip

\noindent
{\bf Output:}
The numerator $ P_1(T) $ of the zeta function of $ C_k $.

\end{algorithm}

As mentioned at the beginning of this section, we now discuss
how to modify the above if we use a~$ K $-basis of~$ \hdr^1(C_K/K) $.

\begin{remark} \label{ComputewithKbasis}
If we start with $\tom_1, \dots, \tom_{2g}$ that give a~$ K $-basis of $\hdr^1(C_K/K)$ instead of an~$R$-basis of $\hdr^1(C/R)$,
coming from~$ \Omega_{A_K/K}^1 $,
then we can still carry out the above algorithms by computing matrices~$\tM_1$,
with as entries the~$ [\tom_i] \cup [\tom_j] = \langle \tom_i , \tom_j \rangle_U $ ,
and $\tM_2$, with as entries the~$ \langle \tom_i , \phi(\tom_j) \rangle_U$. But
we may have to increase the precisions compared to using~$ M_1 $
and~$ M_2 $, for which we give estimates below.

After multiplying with suitable powers of $\pi$,
we may assume that~$\tom_1, \dots, \tom_{2g}$ come from~$ \Omega_{A/R}^1 $.
Then taking suitable~$ R $-linear combinations, we may assume
that they satisfy the condition in~\eqref{CompID} to be in~$ H^0(C, \I(D)) $
for a suitable~$ D $, and result in~$ 2g $ $ R $-linearly independent
elements of~$\hdr^1(C/R)$ by using Proposition~\ref{compdr}.
They do not necessarily form an $R$-basis,
so Remark~\ref{detrem} is not applicable here.
However, we still compute the entries of~$\tM_1$ and~$\tM_2$ as Lemma~\ref{estimates-lemma} still applies.

Suppose that $\tM_1$ has Smith normal form $ P \text{diag}(\pi^{m_1}, \dots, \pi^{m_{2g}}) Q$ 
for some~$P$ and~$ Q $ in~$GL_{2g}(R)$, and $0 \le m_1 \le \dots \le m_{2g}$. 
Then $\tM_1^{-1}$ has entries in~$\pi^{-m_{2g}} R$.
Because~$m_{2g}$ is the smallest integer~$m\ge 0$ such that the
image of multiplication by~$ \tM_1 $ on~$ ( R/\pi^{m+1}) ^{2g} $
has~$ p $-torsion of rank~$ 2g $, it can be found by 
computing~$\tM_1$ up to a high enough precision.
(Note that~$ \tM_1 $ is easy to compute.)

By Corollary~\ref{h-cor}, the entries of~$\tM_2$ are in~$\pi^{-h}R$. 
Therefore, $M = \tM_1^{-1} \tM_2$ has entries in $\pi^{-m_{2g}-h}R$.
Since $M' = \sigma^{l-1}(M) \times \dots \times \sigma(M)$, 
we see that the entries of $M'^j$ lie in $\pi^{-jm_{2g}-jh} R$.
Thus in order to determine the coefficients of the zeta function, one should know the entries of $M$ up to precision 
$N = \tilde N + g (l-1) (m_{2g} + h) e^{-1} $.
Then by Lemma~\ref{lemmaprecision} we need to know the entries of $\tM_1$ up to precision $N_1' = N + h e^{-1}$
and the entries of $\tM_2$ up to precision $N_2' = N + m_{2g} e^{-1}$.
\end{remark}

\section{Computing the first de Rham cohomology group} \label{de-rham}

In this section we explain how to get a basis of $\hdr^1(C/R)$
for~$ C $ and~$ R $ as in Section~\ref{sec:intro}. We consider
a proper, smooth curve~$f: C \to \Spec(R)$ over a discrete valuation ring $R$.
We let~$ K $ be the field of fractions~$  \Frac(R) $ of~$ R $,
and~$k$ its residue field.
We let~$ C_k $ and~$C_K$ be the geometrically irreducible special and generic fibres.
We do not need~$ R $ to be complete here, but we do assume its
characteristic is zero.

We shall use the following strategy for computing de Rham
cohomology.
Recall that it is defined as the hypercohomology of the complex $\O_C \xrightarrow{\dd} \Omega_{C/R}^1$. In order to compute $\hdr^1(C/R)$, we
use an exact sequence of complexes
{\smaller
\begin{equation*}
0 \to [\O_C \to \Omega_{C/R}^1] \to [\O_C(D) \to \Omega_{C/R}^1(D')] \to [\O_C(D)/\O_C \to \Omega_{C/R}^1(D')/\Omega_{C/R}^1] \to 0
\end{equation*}%
}%
for some suitably chosen divisors $D$ and~$D'$. We shall show below
that the induced map~$\dd : \O_C(D)/\O_C \to \Omega_{C/R}^1(D')/\Omega_{C/R}^1$ is injective. 
We then replace~$\Omega_{C/R}^1(D')$ with
a subsheaf $\I(D)$, containing~$ \Omega_{C/R}^1 $, for which~$\dd: \O_C(D)/\O_C \to \I(D)/\Omega_{C/R}^1$ is an isomorphism.
The hypercohomology of the complex $\O_C(D)/\O_C \to \I(D)/\Omega_{C/R}^1$ vanishes,
and, as a result, the hypercohomology of the complex $\O_C(D) \to \I(D)$ computes the de Rham cohomology of $C/R$
(see Proposition~\ref{compdr} for details.)
(The analogue on the generic fibre of~$ \I(D) $ consists of 1-forms
of the second kind; see Remark~\ref{secondkindrem}(2), Theorem~\ref{secondkindthm},
and Remark~\ref{refinedremark}.)

\medskip
We now make explicit the assumptions on the divisors $D$ and $D'$.

\begin{assumption} \label{D-D'-assumption}
$D$ and~$ D' $ are effective, finite over $R$, with~$\deg(D_k) \ge 2g-1$,
and~$ D' $ is such that~$ \dd $ maps $ \O_C(D)$ to $ \Omega_{C/R}^1(D')$.
\end{assumption}

\begin{prop} \label{H1Dtrivial}
Under the above assumptions, we have $H^1(C,\O_C(D))  = 0$. 
\end{prop}

\begin{proof}
The assumptions imply~$H^1(C_k, \O_{C_k}(D_k)) \cong H^0(C_k, \Omega_{C_k/k}^1(-D_k))^\vee = 0$.
By parts~(a) and~(b) of \cite[Theorem~5.3.20]{Liu} we then have $H^1(C, \O_C(D)) = 0$.
\end{proof}

\begin{prop} \label{dinject}
Under the assumptions as in Assumption~\ref{D-D'-assumption},
the induced map~$\dd: \O_C(D)/\O_C \to \Omega_{C/R}^1(D')/\Omega_{C/R}^1$ is injective.
\end{prop}

\begin{proof}
There is a commutative diagram
\begin{equation*}
\xymatrix{
  \O_C(D)/\O_C \ar[d] \ar[r]^\dd & \Omega_{C/R}^1(D')/\Omega_{C/R}^1 \ar[d] \\
  \O_{C_K}(D_K)/\O_{C_K} \ar[r]^\dd & \Omega_{C_K/K}^1(D_K')/\Omega_{C_K/K}^1  }
\end{equation*}
We first prove that the bottom horizontal map is injective. In
order to see this, note that~$\O_{C_K}(D_K)/\O_{C_K}$ and $\Omega_{C/K}^1(D_K')/\Omega_{C_K}^1$ are skyscaper sheaves,
so it suffices to prove the injectivity when~$D_K$ is a multiple of a point~$ P $.
We may extend~$ K $ and assume~$ P $ is~$ K $-rational.
With~$t_P$ a local parameter at $P$, we then have~$\dd: \O_{C_K}(nP)/\O_{C_K} \cong \oplus_{i =1}^n K t_P^i \to \Omega_{C_K}^1((n+1)P)/\Omega_{C_K}^1 \cong \oplus_{i = 1}^{n} K t_P^{i+1} \dd t$ is injective.

Next we shall prove that the left arrow of the above diagram is injective,
which can be checked at the stalks.
Because~$K = R[\frac{1}{\pi}]$, this amounts to showing that
multiplication by~$ \pi $ is injective at each stalk of~$ \O_C(D)/\O_C $.
This is clear at the stalks at points not in~$ C_k $
as~$ \pi $ is a unit there. And if~$ q $ is in~$ C_k $ then in
the local ring~$ \O_{C,q} $ we have that~$ D $ is defined by~$ \pi_1 \dots \pi_l $
for some prime elements in~$ \O_{C,q} $, a unique factorisation domain because~$ C $
is regular at~$ q $. By our assumptions on~$ D $, none of the~$ \pi_i $ are
associate to~$ \pi $. It follows that multiplication by~$ \pi $
on~$ \pi_1^{-1} \dots \pi_l^{-1} \O_{C,q} / \O_{C,q} $ is injective,
as required.
\end{proof}

\begin{definition} \label{ID}
Let~$\pr: \Omega_{C/R}^1(D') \to \Omega_{C/R}^1(D')/\Omega_{C/R}^1$
be the natural projection,
and let~$ \dd : \O_C(D)/\O_C \to \Omega_{C/R}^1(D')/\Omega_{C/R}^1$
be as in Proposition~\ref{dinject}. We then set~$\I(D) = \pr^{-1}(\im(\dd))$.
\end{definition}

Note that sections of~$ \I(D) $ are sections of~$ \Omega_{C/R}^1(D') $
such that their `polar parts' in~$ \Omega_{C/R}^1(D')/\Omega_{C/R}^1 $ can
be integrated to elements of~$ \O_C(D) / \O_C $. This depends
on~$ D $ but not on~$ D' $, justifying the notation.

We then have a short exact sequence of complexes,
$$
0 \to [\O_C \to \Omega_{C/R}^1] \to [\O_C(D) \to \I(D)] \to [\O_C(D)/\O_C \to \I(D)/\Omega_{C/R}^1] \to 0
\,,
$$
where the hypercohomology of the third complex vanishes because
the map in it is an isomorphism by Proposition~\ref{dinject}.
We then obtain the following.

\begin{prop} \label{compdr}
The natural map~$  \hdr^*(C/R) \to \H^{*}(\O_C(D) \to \I(D)) $ is an isomorpishm. Moreover, we
have an exact sequence
$$ 0 \to \hdr^0(C/R) \to H^0(C, \O_C(D)) \to H^0(C,\I(D)) \to \hdr^1(C/R) \to 0.$$
\end{prop}

\begin{proof}
Writing~$ H_{D}^*(C) $ for~$ \H^*(\O_C(D) \to \I(D)) $,
we have a commutive diagram
\begin{equation*}
\xymatrix@C=.8pc{
 0  \ar[r] & \hdr^0(C/R) \ar[r] \ar[d]^-{\cong} & H^0(C, \O_C) \ar[r] \ar[d] & H^0(C,\Omega_{C/R}^1) \ar[r] \ar[d] &  \hdr^1(C/R) \ar[r] \ar[d]^-{\cong} & H^1(C, \O_C) \ar[d]
\\
 0  \ar[r] & H_D^0(C) \ar[r] & H^0(C, \O_C(D)) \ar[r]^-{\dd} & H^0(C,\I(D)) \ar[r] &  H_D^1(C) \ar[r] & 0
\,,
}
\end{equation*}
from which the result is immediate.
\end{proof}

\begin{remark} \label{secondkindrem}
(1)
The kernel and cokernel of~$ \dd $ in the bottom row of the above
diagram are free~$ R $-modules because each~$ \hdr^*(C/R) $ is
free, but for the cokernel this is not obvious from the explicit
description that we shall obtain for~$ H^0(C, \I(D)) $.

(2)
The statements and diagram in Proposition~\ref{compdr} are compatible with tensoring
over~$ R $ with~$ K $. Then~$ \I(D) \otimes_R K $ consists of the
elements in~$ \Omega_{C_K/K}^1 $ that are of the second kind and have poles of order at most one higher than those of the
functions in~$ \O_D(D) $ (cf.~Theorem~\ref{secondkindthm}
and Remark~\ref{refinedremark}).
\end{remark}

\begin{remark}
We have a commutative diagram
\begin{equation*}
\xymatrix{
0 \ar[r] & \Omega_{C/R}^1 \ar[r] \ar@{=}[d] &  \I(D) \ar[r] \ar[d] & \im(\dd)  \ar[r] \ar[d] &  0
\\
0 \ar[r] & \Omega_{C/R}^1 \ar[r] &  \Omega_{C/R}(D') \ar[r]^-{\pr} & \Omega_{C/R}^1(D')/\Omega_{C/R}^1 \ar[r] & 0
}
\end{equation*}
of sheaves on~$ C $ with exact rows. It then follows from the
map between the long exact sequences of cohomology associated to the rows
that the commutative diagram
\begin{equation} \label{CompID}
\begin{split}
\xymatrix{
  H^0(C, \I(D)) \ar[d] \ar[r] & H^0(C, \Omega_{C/R}^1(D')) \ar[d] \\
  H^0(C, \im(\dd)) \ar[r] & H^0(C, \Omega_{C/R}^1(D')/\Omega_{C/R}^1)
}
\end{split}
\end{equation}
is a Cartesian square of $R$-modules.
Here the horizontal maps are injections, so we may view~$ H^0(C, \I(D)) $ as consisting of
the elements of~$ H^0(C, \Omega_{C/R}^1(D')) $ that satisfy a condition on their polar parts.
We obtain the same condition if we replace~$ H^0(C, \im(\dd)) $
with the isomorphic~$ H^0(C, \O_C(D)/\O_C) $, and the bottom
map with the one induced by~$ \dd $.
\end{remark}

\begin{remark}
Some special cases of the above results can be found in~\cite[\myS3]{bogaart}, where~$ D $
is a multiple of a section. In our results, the components of~$ D $
can be singular, and need not correspond to~$ K $-rational points on the generic fibre.
\end{remark}

\begin{example}
In~$ \Z_2[x,y] $, let~$h = x^2 - 4 x + 1$, and set~$f = y^2 + h y + x^5 $. Then~$ f $
defines a hyperelliptic curve of genus~2 over~$ \Z_2 $, with
a second affine patch defined by~$ \tf = \ty^2 + \th \ty + \tx $
for~$ \tx = 1/x $, $ \ty = y / x^3 $, and~$ \th = \tx^3 h(1/\tx) = \tx^3 - 4 \tx^2 + \tx $.
Let~$ \infty $ be the point of~$ C $ given by~$ \tx = \ty = 0 $
in the second affine patch. With~$ D = 3 (\infty) $ and~$ D' = 4 (\infty) $,
Assumption~\ref{D-D'-assumption} is satisified, and
\begin{alignat*}{1}
H^0(C, \O_C(D)) & = \langle 1, x \rangle_{\Z_2},
\\
H^0(C, \Omega_{C/R}^1(D')) & = \langle \o_1, \o_2, \o_3 , \o_4, \o_5 \rangle_{\Z_2},
\end{alignat*}
for~$ \o_1 = \frac1{2y+h} \dd x = - \frac1{5 x^4 + h'} \dd y$, $ \o_2 = x \o_1 $, $ \o_3 = x^2 \o_1 $,
$ \o_4 = y \o_1 $, and~$ \o_5 = x^3 \o_1 $, as one sees by computing on the closed and generic fibres, and applying
parts~(a) and~(b) of~\cite[Theorem~5.3.20]{Liu}.
Note that~$ \dd x = \o_1 - 4 \o_2 + \o_3 + 2 \o_4 $.

With~$ t = \ty $ a local parameter at~$ \infty $, the forms with poles
at~$ \infty $ have expansions
\begin{alignat*}{1}
\o_3(t) & = (- t^{-2} + 4 t - 8 t^2 + \dots) \dd t
\\
\o_4(t) & = (t^{-3} + t^{-2} - t^2 + \dots) \dd t
\\
\o_5(t) & = (t^{-4} + t^{-3} - t + \dots) \dd t
,
\end{alignat*}
so that~$ H^0(C, \I(D') ) = \langle \o_1, \o_2, \o_3, 2 \o_4, \o_4 + \o_5 \rangle_{\Z_2} $,
and~$ [\o_1] $, $ [\o_2] $, $ [\o_3] $ and~$ [\o_4 + \o_5] $ form a~$ \Z_2 $-basis of~$ \hdr^1(C/\Z_2) $.
The matrix $M_1$ with cup products for these classes is
\begin{equation*}
\renewcommand*{\arraystretch}{1.2}
\begin{pmatrix*}[r]
      0 &        0 &        0 & -\frac13 \\
      0 &        0 &       -1 &  \frac13 \\
      0 &        1 &        0 &  \frac43 \\
\frac13 & -\frac13 & -\frac43 &        0 
\end{pmatrix*}
\end{equation*}
Using only~$ x^i \o_1 $ for~$ i=0, \dots,3 $ and ignoring~$ \o_4 $
results in a lattice of index~2 with basis~$ [\o_1] $, $ [\o_2] $, $ [\o_3] $ and~$ [2 \o_5] $
(cf.~\cite[\myS6]{bogaart}).
\end{example}

In order to use Proposition~\ref{compdr} to compute the de Rham
cohomology of~$ C/R $ effectively, we need to describe~$H^0(C,\O_C(D))$ and $H^0(C, \I(D))$.
We do this for~$ C/R $ as in Notation~\ref{basicnot} and
Notation~\ref{globalnotation}, still assuming Assumption~\ref{D-D'-assumption}
holds.

In Algorithms~\ref{KtoRbasis} and~\ref{KtoRbasisOmega},
we describe how to obtain~$ R $-bases of~$ H^0(C, \O_C(D)) $
and~$ H^0(C, \Omega_{C/R}(D')) $ from~$ K $-bases of the corresponding~$ K $-vector
spaces on~$ C_K $.
(If~$ C $ is planar, then there are more direct
approaches; see Proposition~\ref{planarH0OCD} for~$ H^0(C, \O_C(D))$
and Proposition~\ref{P2C-prop} for~$ H^0(C, \Omega_{C/R}(D')) $.
Moreover, we can also use Proposition~\ref{R-basis-image-d} instead
of Proposition~\ref{newOCDprop}.)

\begin{algorithm} \label{KtoRbasis}
Suppose that~$ |D| \cap \Spec(A) = \emptyset $, and that~$a_1, ..., a_m$ in~$ K[\x] $ induce a~$ K $-basis of $H^0(C_K,\O_{C_K}(D_K)) \subseteq A_K = K[\x] / (f_2, \dots, f_n) $.
We can then obtain an $R$-basis of $H^0(C,\O_C(D))$ as follows.

(1)
Multiplying with powers of $\pi$ we can assume that~$a_1, ..., a_m$ are in $R[\x]$.
By our assumptions on~$ D $, and \cite[Theorem~5.3.20]{Liu}, it suffices to obtain such~$ a_i $
with~$ k $-linearly independent images~$ \ol{a_1}, \dots, \ol{a_m} $ in~$ \Abar = k[\x] / (\ol{f_2}, \dots, \ol{f_n}) $.

(2)
Suppose that~$ \ol{a_1}, \dots, \ol{a_{j-1}} $ are~$ k $-linearly independent
for some~$ j = 1, \dots, m $. If~$ \ol{a_j} $ is~$ k $-linearly
dependent of those, there exist~$ r_1, \dots, r_{j-1} $ in~$ R $ as well as~$ b_j $
and~$ c_2, \dots, c_n $ in~$ R[\x] $ such that~$ a_j - r_1 a_1 - \dots - r_{j-1} a_{j-1} = \pi b_j + \sum_{i=2}^n c_i f_i $
in~$ R[\x] $.
Here the~$ c_i $ can be found by computing that~$ \ol{a_j - r_1 a_1 - \dots - r_{j-1} a_{j-1}} $
is in~$ (\ol{f_2}, \dots, \ol{f_n}) $ in ~$ k[\x] $ using a Gr\"obner basis,
and lifting to~$ R[\x] $, which determines~$ b_j $.
Then~$ a_1,\dots,a_{j-1}, a_j $ and~$ a_1, \dots, a_{j-1}, b_j $
span the same~$ K $-subspace of~$ A $, so we may replace~$ a_j $
with~$ b_j $. Because~$  R a_1 + \dots + R a_j \subsetneq R a_1 + \dots + R a_{j-1} + R b_j $ in
the finitely generated free~$ R $-module~$ H^0(C,\O_C(D)) $,
this must stop after finitely many iterations, at which stage~$ \ol{a_1}, \dots, \ol{a_j} $
are~$ k $-linearly independent.

(3) Using~(2) for~$ j = 1, \dots, m-1 $ gives~$ k $-linearly independent~$ \ol{a_1}, \dots, \ol{a_m} $
in~$ \Abar $.
\end{algorithm}

We now consider how to get an $R$-basis of $H^0(C, \I(D))$ by
combining the Cartesian square in~\eqref{CompID} with the
isomorphism~$ H^0(C, \O_C(D)/\O_C) \to H^0(C, \im(\dd))$ induced by~$\dd $.
We start with describing~$ H^0(C, \O_C(D) / \O_C ) $, in such
a way that the map~$ \dd $ to~$ H^0(C, \Omega_{C/R}^1(D')/\Omega_{C/R}^1) $ is clear.
If~$ D = \sum_Q m_Q Q $ with the~$ Q $ disjoint $ R $-rational sections, and~$ t_Q $
is a local parameter on~$ C $ at~$ Q $, then~$ H^0(C, \O_C(D) / \O_C ) $ is just~$ \sum_Q \sum_{i=-m_Q}^{-1} R t_Q^i $.
In our rather more general situation, the irreducible components of~$ D $
may not be disjoint, and their generic points not~$ K $-rational.

Note that each irreducible component of~$ D $ is $ \Spec(R') $
for a subring~$ R' $ of the valuation ring~$ R_L $ in a finite
extension~$ L $ of~$ K $. Because `lying over' holds for the
integral extension~$ R_L / R' $, and~$ L $ is also complete, also~$ \Spec(R') $ has only two points.
Grouping the irreducible components of~$ D $ according to their
intersection with the special fibre, it is clear it suffices
to describe~$ H^0(C, \O_C(D)/ \O_C) $ under the assumption that
each irreducible component of~$ D $ meets~$ C_k $ in the same
point~$ q $, so that they are all contained in~$ \Spec(\O_{C,q} ) \subseteq C $.

The following proposition always applies because~$ C $
is regular, hence~$ \O_{C,q} $ is a unique factorisation domain.

\begin{proposition} \label{newOCDprop}
Suppose that~$ D $ in~$ \Spec(\O_{C,q}) $ is defined by~$ \pi_1^{m_1} \dots \pi_l^{m_l} $
where the~$ \pi_i $ are distinct prime elements of~$ \O_{C,q} $,
and all~$ m_i \ge 1 $.
If~$ \Sigma_i' \subseteq \O_{C,q} $ induces an~$ R $-basis
of~$ \O_{C,q} / (\pi_i) $, then~$ \Sigma_i = \Sigma_i' \cup \pi_i \Sigma_i' \cup \dots \cup \pi_i^{m_i-1} \Sigma_i' $
induces an~$ R $-basis of~$ \O_{C,q} / (\pi_i^{m_i}) $,
\begin{equation*}
\Sigma = \Sigma_1 \cup \pi_1^{m_1} \Sigma_2 \cup \pi_1^{m_1} \pi_2^{m_2} \Sigma_3 \cup \dots \cup \pi_1^{m_1} \pi_2^{m_2} \dots \pi_{l-1}^{m_{l-1}}\Sigma_l
\end{equation*}
induces an~$ R $-basis of~$ \O_{C,q}/ (\pi_1^{m_1} \dots \pi_l^{m_l}) $,
and~$ \pi_1^{-m_1} \cdots \pi_l^{-m_l} \Sigma $ induces an~$ R $-basis
of~$ H^0(C, \O_C(D)/ \O_C) $. 
The corresponding statements hold if we replace~$ \O_{C,q} $ by a ring~$ B $ with~$ \Spec(B) \subseteq C $
containing~$ q $ if the natural maps~$ B/(\pi_i) \to \O_{C,q}/ (\pi_i) $
are isomorphisms.
\end{proposition}

\begin{proof}
That~$ \Sigma_i $ induces an~$ R $-basis of~$ \O_{C,q} / (\pi_i^{m_i}) $
follows easily by induction on~$ m_i $ from the short exact sequence
$ 0 \to (\pi_i) / (\pi_i^{m_i}) \to \O_{C,q} / (\pi_i^{m_i}) \to \O_{C,q} / (\pi_i) \to 0 $
and the isomorphism~$ \O_{C,q} / (\pi_i^{m_i-1}) \to (\pi_i) / (\pi_i^{m_i}) $
given by multiplying by~$ \pi_i $.
That~$ \Sigma $ induces an~$ R $-basis of~$ \O_{C,q} / (\pi_1^{m_1} \dots \pi_l^{m_l}) $
then follows similarly based on induction on~$ l $, the short
exact sequences
\begin{equation*}
 0 \to (\pi_1^{m_1}) / (\pi_1^{m_1} \dots \pi_l^{m_l}) \to \O_{C,q} / (\pi_1^{m_1} \dots \pi_l^{m_l}) \to \O_{C,q} / (\pi_1^{m_1}) \to 0 
,
\end{equation*}
and the isomorphism~$ \O_{C,q} / (\pi_1^{m_1} \dots \pi_l^{m_l}) \to (\pi_1^{m_1}) / (\pi_1^{m_1} \dots \pi_l^{m_l})  $ given by multiplication by~$ \pi_1^{m_1} $.

We have an isomorphism~$ \O_{C,q} / (\pi_1^{m_1} \dots \pi_l^{m_l}) \to H^0(C, \O_C(D) / \O_C) $,
given by multiplication by~$ \pi_1^{-m_1} \dots \pi_l^{-m_l} $,
so we obtain an~$ R $-basis for the latter from~$ \Sigma $.
Finally, using the natural map~$ B \to \O_{C,q} $ we see that the short
exact sequences for~$ B $ corresponding to those above for~$ \O_{C,q} $ induce
isomorphisms~$ B / (\pi_i^{m_i}) \to \O_{C,q} / (\pi_i^{m_i}) $ and~$ B / (\pi_1^{m_1} \dots \pi_l^{m_l}) \to \O_{C,q} / (\pi_1^{m_1} \dots \pi_l^{m_l}) $.
\end{proof}

\begin{algorithm} \label{KtoRbasisOmega}
If~$ |D'| \cap \Spec(A) = \emptyset $, so that~$ H^0(C, \Omega_{C/R}^1(D')) \subseteq \Omega_{A/R}^1 $,
then one can get an~$R$-basis of $H^0(C, \Omega_{C/R}^1(D'))$
from a~$ K $-basis of~$H^0(C, \Omega_{C_K/K}^1(D_K'))$,
using the following variation of Algorithm~\ref{KtoRbasis}.
We may identify~$ \Omega_{R[\x]/R}^1 $ with~$ R[\x]^n $ by
letting~$ \sum_{i=1}^n g_i \dd x_i $ correspond to~$ (g_1,\dots, g_n) $.
Using that~$ A= R[\x] / (f_2, \dots, f_n) $ as in  Assumption~\ref{Aassumption}, 
this way~$ \Omega_{A/R}^1 $ identifies with the quotient~$ M $ of~$ A^n $
by its~$ A $-submodule generated by~$ (f_{i,x_1}, \dots, f_{i, x_n})  $
for~$ i = 2, \dots, n $. We view~$ M $ as the quotient of~$ R[\x]^n $
by the~$ R[\x] $-submodule generated by some explicit~$ F_1, \dots, F_{ n^2-1} $.
We have compatible identifications
of~$ \Omega_{A_K/K}^1 $ with~$ M_K = M \otimes_R K $
and of~$ \Omega_{\Abar/k}^1 $ with~$ \ol{M} = M \otimes_R k $.
We view~$ M_K $ as the quotient of~$ K[\x]^n $ by the~$ K[\x] $-submodule
generated by~$ F_1, \dots, F_{n^2-1} $, and~$ \ol{M} $ as the
quotient of~$ k[\x]^n $ by the~$ k[\x] $-submodule generated by
the images of~$ F_1, \dots, F_l $.
Viewing a~$ K $-basis of~$H^0(C, \Omega_{C_K/K}^1(D_K'))$ as
being in~$ M $, and lifting it to~$ v_1, \dots, v_l $ in~$  K[\x]^n $,
multiplying by powers of~$ \pi $ we can assume they are in~$ R[\x]^n $,
and want their images~$ \ol{v_1}, \dots, \ol{v_l} $ in~$ \ol{M} $
to be~$ k $-linearly independent. In the equivalent of Algorithm~\ref{KtoRbasis}(2),
we lift a~$ k $-linear dependency of~$ \ol{v_j} $ on~$ \ol{v_1}, \dots, \ol{v_{j-1}} $
to~$ v_j - r_1 v_1 - \dots - r_{j-1} v_{j-1} = \pi w_j + \sum_{i=1}^{n^2-1} c_i F_i $
in~$ R[\x]^n $, with the~$ r_i $ in~$ R $, $ w_j $ in~$ R[\x]^n $, and the~$ c_i $
in~$ R[\x] $. This can be done by expressing~$ \ol{v_j - r_1 v_1 - \dots - r_{j-1} v_{j-1}} $
in~$ k[\x]^n $ as a member of the~$ k[\x] $-submodule generated
by the images of~$ F_1, \dots, F_{n^2-1} $,
which can be done using Gr\"obner basis.
Replacing~$ v_j $ with~$ w_j $, we can iterate until~$ \ol{v_j} $
is~$ k $-linearly independent of~$ \ol{v_1}, \dots, \ol{v_{j-1}} $.
Doing this for~$ j = 1, \dots, l $ results in~$ v_1, \dots, v_l $
in~$ R[\x]^n $ such that their images
in~$ M $ correspond to an~$ R $-basis of~$ H^0(C, \Omega_{C/R}^1(D')) \subseteq \Omega_{A/R}^1 $.
\end{algorithm}

Algorithms~\ref{KtoRbasis} and~\ref{KtoRbasisOmega} allow us
to write~$ H^0(C, \O_C(D) / \O_C) \rightiso  H^0(C, \im(\dd)) $
and~$ H^0(C, \Omega_{C/R}^1(D')) $ as finitely generated free~$ R $-modules.
In order to make the maps to~$ H^0(C, \Omega_{C/R}^1(D')/\Omega_{C/R}^1) $
in~\eqref{CompID} explicit,
we can use the injection of this last group into~$ H^0(C_K, \Omega_{C_K/K}^1(D_K') / \Omega_{C_K/K}^1) $,
followed by another injection obtained by
extending the coefficients to a finite extension~$ L $ of~$K $ such that~$ |D_L| $ consists of~$ L $-rational points.
The resulting~$ L $-vector space~$ H^0(C_L, \Omega_{C_L/L}^1(D_L') / \Omega_{C_L/L}^1) $ is
easily described using local parameters, as are the resulting composed 
maps to it from~$ H^0(C, \O_C(D) / \O_C) $ and~$ H^0(C, \Omega_{C/R}^1(D')/\Omega_{C/R}^1) $.
This way, using~\eqref{CompID}, we can compute~$ H^0(C, \I(D)) $ explicitly as an~$ R $-submodule
of~$ H^0(C, \Omega_{C/R}^1(D') ) $.

\section{The case of smooth plane curves} \label{smooth-planar}

Let~$ R $, $ \pi $ and~$ k $ be as in Notation~\ref{basicnot}.
Any curve in~$ \P_k^2 $ that is smooth over~$ k $ is obtained
as the special fibre $ C_k $ of a curve~$ C $ in $ \P_R^2 $,
smooth over~$ R $, by simply lifting the defining homogeneous
equation over $ k $ to some homogeneous~$ F[X,Y,Z] $ in~$ R[X,Y,Z] $ of the same degree.
In this section we discuss how to make~\eqref{congeq2} explicit,
i.e., if we let~$ f(x,y) = F(x,y,1) $, then we describe how to calculate $ P_1 $, $ P_2 $ in~$ R[x,y] $ such that ~$ \ol{P_1} \, \ol{f_x} + \ol{P_2} \, \ol{f_y} = 1 + \ol{f} \, \ol{\Delta} $
in~$ k[x,y] $ for some~$ \Delta $ in $ R[x,y] $.
Clearly, we only need~$ \ol{P_1} $ and~$ \ol{P_2} $, so can work
directly over~$ k $, and simplify notation accordingly. Also, because~$ \Delta $ plays no role once
we go from~\eqref{globalG} to the system~$ H(\S) $ in Section~\ref{local-frobenius}
(see Remark~\ref{locnewtonrem}), we work directly on~$ C_k $.

We also compute the~$ R $-modules required in Proposition~\ref{compdr}
and~\eqref{CompID} directly, i.e., without using Algorithms~\ref{KtoRbasis}
and~\ref{KtoRbasisOmega}, and provide an alternative to Proposition~\ref{newOCDprop}.

In Section~\ref{complexity}, we shall discuss the complexity of the algorithm in this case.

\begin{lemma} \label{Olemma}
Let~$ k $ be a field, and suppose~$ F(X,Y,Z) $
is homogeneous of degree~$ d \ge 1 $ in~$ k[X,Y,Z] $, not divisible
by~$ Z $, 
and defines a smooth projective curve~$ C $ in~$ \P_k^2 $. 
If we write~$ x = X/Z $ and~$ y = Y/Z $, and~$ C_\infty = C \cdot \{ Z = 0 \} $, then for~$ m \ge d - 2 $
the image of~$ k[x,y]_{\le m} $ in~$ k(C) $ is equal to~$ L(m C_\infty) $.
\end{lemma}

\begin{proof}
Clearly the image is contained in~$ L(m C_\infty) $.
Let~$ g = \binom{d-1}2 $ be the genus of~$ C $.
If~$ m = d-2 $ then~$ m C_\infty $ has degree~$ d (d-2) > 2 g - 2 = d (d-3) $,
so that~$ L(m C_\infty) $ has dimension~$ 1 - g + d (d-2) = \frac12 ( d^2 -  d) $,
whereas $ k[x,y]_{\le d-2} $ injects into $ k(C) $ and has dimension~$ \binom d 2 $.
Its image therefore coincides with ~$ L((d-2) C_\infty) $.
For~$ m \ge d-2 $, if we increase~$ m $ by~1 then the dimension
of~$ L(m C_\infty) $ increases by~$ \deg(C_\infty) = d $, whereas
the dimension of the image of~$ k[x,y]_{\le m} $ in~$ k(C) $
increases by~$ d $ as well.
\end{proof}

\begin{remark} \label{coordringbasis}
In order to see the increase in dimension of the image of~$ k[x,y]_{\le m} $,
let~$ a = 0, \dots, d $ the smallest
values such that the monomial~$ x^a y^{d-a} $ occurs in~$ f $.
Then~$ k[x,y] / (f) $ has as a~$ k $-basis the classes
of~$ x^i y^j $ with~$ i+j < d $ and those~$ x^i y^j $ with~$ i + j \ge d $
such that~$ i < a $ or~$ j < d - a $, and the image
of~$ k[x,y]_{\le m} $ in~$ k[x,y] / (f) $ has as~$ k $-basis
those classes for which~$ i + j \le m $.
In order to see this, note that for any~$ h $ in~$ k[x,y] $ of
degree~$ e \ge d $, we can use this monomial in~$ f $ to successively eliminate~$ x^a y^{e-a}, x^{a+1} y^{e-a-1}, \dots, x^{e-d+a} y^{d-a} $,
while introducing only terms of lower degree or terms of degree~$ e $
with higher degree in~$ x $.
\end{remark}

For~$ F(X,Y,Z) $ as in Lemma~\ref{Olemma}, it can happen that
one of~$ F_X $, $ F_Y $ and~$ F_Z $ is identically~0 (e.g.,
if $ F(X,Y,Z) = X^2 + Y Z $ and~$ k $ has characteristic~$ 2 $).
We can avoid such special cases in the next lemma, as follows.
Because~$ C $ is smooth, we cannot have~$ F_X $,  $ F_Y $ and~$ F_Z $
identically~0. Also, if~$ d \ge 2 $ then we cannot have two of
those identically~0 because~$ d F =  X F_X + Y F_Y + Z F_Z $
and our curve must be irreducible. So for~$ d \ge 2 $, by
permuting the coordinates if necessary, we may assume
that $ F_X \not\equiv 0 $ and~$ F_Y \not\equiv 0 $,
and for~$ d = 1 $ this can be achieved by a coordinate change.
If~$ f(x,y) = F(x,y,1) $, then~$ f_x(x,y) = F_X(x,y,1) $ and~$ f_y(x,y) = F_Y(x,y,1) $
are not identically~0.

\begin{lemma} \label{Omegalemma}
Let $ F $ and $ C $ be as in Lemma~\ref{Olemma}, and set~$ f(x,y) = F(x,y,1) $.
Assume that~$ f_x $ and~$ f_y $ are not identically~0.
Then~$ \omega = f_y^{-1} \dd x = - f_x^{-1} \dd y $
defines an element of~$ \G(C, \Omega_{C/k}^1(-K)) $, with divisor~$ K= (d-3) C_\infty $.
Moreover, if~$ m \ge 0 $ then
\begin{equation*}
\G(C, \Omega_{C/k}^1(m C_\infty))
=
\{h \o \text{ with } h \text{ in the image of } k[x,y]_{\le d+m-3} \}
\,.
\end{equation*}
\end{lemma}

\begin{proof}
On the part where~$ f(x,y) = F(x,y,1) $,
locally either~$ f_x \ne 0 $ and~$ \dd y $ generates the 1-forms,
or~$ f_y \ne 0 $ and~$ \dd x $ generates the 1-forms. Then~$ \o $
gives a section of~$ \Omega_{C/k}^1 $ on this affine part because~$ 0 = \dd f = f_x \dd x + f_y \dd y $.
Similarly, on the part where~$ Y \ne 0 $ we let~$ x_1 = X/Y= x/y $
and~$ z_1 = Z/Y = 1/y $, so the equation
becomes~$ f_1(x_1, z_1) = F(x_1, 1, z_1) = z_1^d f(x_1/z_1,1/z_1) $.
Here~$ f_{1, x_1}^{-1} \dd z_1 $ and~$ - f_{1 , z_1}^{-1} \dd x_1 $
define a section~$ \o_1 $ of~$ \Omega_{C/k}^1 $.
Finally, where~$ X \ne 0 $ we let~$ y_2 = Y/X= y/x $
and~$ z_2 = Z/X = 1/x $, so the equation
becomes~$ f_2(y_2, z_2) = F(1, y_2, z_2) = z_2^d f(1/z_2,y_2/z_2) $.
Here~$ f_{2, y_2}^{-1} \dd z_2 $ and~$ - f_{2 , z_2}^{-1} \dd y_2 $
define a section~$ \o_2 $ of~$ \Omega_{C/k}^1 $.

In order to see that~$ \o $, $ z_1^{d-3} \o_1 $ and~$ - z_2^{d-3} \o_2 $
glue to an element of~$ \G(C, \Omega_{C/k}^1) $,
we note that~$ \Omega_{C/k}^1 $ has no torsion,
so that it suffices to check they agree at the generic point.
For this, using~$ z_1^{-d} f_1(x_1,z_1) = f(x, y) $
with~$ x = x_1 z_1^{-1} $ and~$ y = z_1^{-1} $
gives us~$ z_1^{-d} f_{1, x_1}(x_1, z_1) = f_x(x,y) z_1^{-1} $,
so that~$ - f_x(x, y)^{-1} \dd y = z_1^{d-3} f_{1, x_1}(x_1,z_1)^{-1} \dd z_1 $
because~$ \dd y = - z_1^{-2} \dd z_1 $.
Similarly, using~$  z_2^{-d} f_2(y_2, z_2) = f(x,y) $
with~$ y = y_2 z_2^{-1} $ and~$ x = z_2^{-1} $ gives
us~$ z_2^{-d} f_{2, y_2}(y_2, z_2) = f_y(x,y) z_2^{-1} $,
which, with~$ \dd x = - z_2^{-2} \dd z_2 $,
implies that~$ f_y(x,y)^{-1} \dd x = - z_2^{d-3} f_{2 , y_2}(y_2, z_2)^{-1} \dd z_2 $.
Therefore they glue, and we see that~$ \o $ when viewed as this
global section has divisor~$ (d-3) C_\infty $.

It now follows that~$ h \o $ with~$ h $ in~$ k[x,y]_{\le d-3} $ is in~$ \G(C, \Omega_{C/k}^1 ) $.
Because~$ k[x,y]_{\le d-3} $ has dimension~$ \binom{d-1}2 = g $
and injects into~$ k(C) $, we get all of~$ \G(C, \Omega_{C/k}^1) $ this way.
If~$ m \ge 1 $ then~$ \Omega_{C/k}^1(m C_\infty) \simeq \O_C( (d + m - 3) C_\infty) $,
and so~$ \G(C, \Omega_{C/k}^1 (m C_\infty)) $ has dimension
$ 1 - g + (d + m - 3) d = \frac d2 (d + 2m - 3 ) $.
On the other hand, the image of~$ k[x,y]_{\le d+m-3} $ in~$ k(C) $
for~$ m \ge 1 $ has dimension
$ \binom {d+1}d2 + d (m-2) = \frac d2 (d + 2 m - 3 ) $
as well.
Because~$ h \o $ is in~$ \G(C, \Omega_{C/k}^1(m C_\infty) ) $
if~$ h $ is in the image of~$ k[x,y]_{\le d + m -3} $, this implies
the last statement of the lemma.
\end{proof}

For the map~$ \dd : \G(C, \O_C( m C_\infty ) ) \to \G (C, \Omega_{C/k}^1( (m+1) C_\infty) ) $
with~$ m \ge d-2 $,
using the description of Lemma~\ref{Olemma} for the domain leads to
a description of the image of~$ \dd $ that is not directly compatible with the one used in Lemma~\ref{Omegalemma}.
But for~$ h $ in $ k[x,y]_{\le m} $, one can simply rewrite~$ \dd h = h_x \dd x + h_y \dd y $ as $ (h_x f_y - h_y f_x) \o $,
with~$ h_x f_y - h_y f_x $ in~$ k[x,y]_{\le d+m-2} $.

In the other direction, it would be sufficient to write~$ \o = a \dd x + b \dd y $ for~$ a $
and~$ b $ in $ k[x,y] / (f) $, which can always be done on this
affine part of the curve.
This is equivalent to~$ a f_y \dd x + b f_y \dd y = \dd x $,
hence to~$ (a f_y - b f_x) \dd x = \dd x $, and therefore to~$ - b f_x  + a f_y = 1 $
in~$ k[x,y] / (f) $.
This identity is also what is needed in~\eqref{congeq2} (see
Section~\ref{planecurvesection}).
The next proposition and corollary through Lemma~\ref{Olemma}
bound the degrees of elements of~$ k[x,y] $ that one has to consider
to find such~$ a $ and~$ b $. Of course, if~$ d = 1 $ we could
use elements of~$ k $. By assuming~$ d \ge 2 $, we also ensure
that~$ f_x $ and~$ f_y $ are non-zero.

\begin{proposition} \label{P1P2bound}
Let~$ C $ and~$ f $ be as in Lemma~\ref{Omegalemma}, and assume
that~$ d \ge 2 $.
Write $ (f_x)_0 = (f_x)_0^a + (f_x)_0^\infty $
with~$ (f_x)_0^a $ supported in the affine part of the curve,
and~$ (f_x)_0^\infty $ supported in~$ |C_\infty| $, and
write~$ (f_y)_0 = (f_y)_0^a + (f_y)_0^\infty $ similarly.
Let~$ D' = (f_x)_0^a + (f_y)_0^a $,
and fix~$ D \ge 0 $, supported in~$ |C_\infty| $ and with~$ \deg(D) - \deg(D') > d (d-3) $.
Then there exist~$ P_1 $ in~$ L(D+(f_x)_0^\infty - (f_x)_\infty) $
and~$ P_2 $ in~$ L(D+(f_y)_0^\infty - (f_y)_\infty) $
such that~$ P_1 f_x + P_2 f_y = 1 $ in~$ k[x,y] / (f) $.
\end{proposition}

\begin{proof}
Because~$ C $ is smooth, there exist~$ \tilde a$ and~$ \tilde b $ in~$ k[x,y] / (f) $
with~$ \tilde a f_x + \tilde b f_y = 1 $.
We shall rewrite this into~$ P_1 f_x + P_2 f_y = 1 $
with~$ P_1 $ and~$ P_2 $ as in the proposition.

Because~$ d (d-3) = 2 g - 2 $, for any divisor~$ E \ge 0 $ supported at infinity we have a commutative diagram
{\smaller
\begin{equation*}
\xymatrix{
0 \ar[r] & L( D - D' ) \ar[r]\ar[d] & L( D + E - D' ) \ar[r]\ar[d] & \G(C, \O_C(D+E)/\O_C(D)) \ar@{=}[d] \ar[r] & 0
\\
0 \ar[r] & L( D ) \ar[r] & L( D + E ) \ar[r] & \G(C, \O_C(D+E)/\O_C(D)) \ar[r] & 0
\,.
}
\end{equation*}
}
For the map on the right, note that the support of~$ D' $ is
disjoint from that of~$ D $ and~$ D+E $, so that~$  \O_C(D+E-D')/\O_C(D-D')$ and $ \O_C(D+E)/\O_C(D) $
identify naturally.
We note that for a fixed~$ h \ne 0 $ in $ L(D + E - D' ) $ we have
\begin{equation*}
(h/f_x)
=
(h) - (f_x)_0^a - (f_x)_0^\infty  + (f_x)_\infty
\ge
- D - E - (f_x)_0^\infty  + (f_x)_\infty
\end{equation*}
because~$ (h) \ge - D - E + D' $ and~$ - (f_x)_0^a \ge - D' $.
Therefore~$ h = a f_x $ with~$ a $ in~$ L(D + E + (f_x)_0^\infty  - (f_x)_\infty) $.
This also holds if~$ h = 0 $, and we can similarly write~$ h = - b f_y $ with~$ b $
in~$ L(D + E + (f_y)_0^\infty  - (f_y)_\infty) $.

We take such an~$ E $ for which~$ \tilde a f_x $
and~$ - \tilde b f_y $ are in~$ L(D+E) $.
They have the same image in~$ \G(C, \O_C(D+E)/\O_C(D)) $ because~$ D \ge 0 $,
so that~$ 1 $ is in~$ L(D) $.
We now fix~$ h $ in $ L( D+E - D') $ with that
image in the diagram.
Writing~$ h = a f_x = - b f_y $ as before, we
have~$ (\tilde a - a) f_x + (\tilde b - b) f_y = 1 $, where~$ ( \tilde a - a ) f_x $
and~$ ( \tilde b - b ) f_y $ are in~$ L(D) $ by the commutativity
of the diagram.
We take~$ P_1 = \tilde a - a $, which is in~$ L(D + (f_x)) $.
But~$ \tilde a $ and~$ a $ have no poles in the affine part
of the curve, so~$ P_1 $ is in~$ L(D+(f_x)_0^\infty - (f_x)_\infty) $.
Similarly~$ P_2 = \tilde b - b $ is in~$ L(D+(f_y)_0^\infty - (f_y)_\infty) $.
\end{proof}

\begin{corollary} \label{smooth-planar-cor}
In the situation of Proposition~\ref{P1P2bound}, the following
hold.

(1)
We have~$ P_1 f_x + P_2 f_y = 1 $ with~$ P_1 $ and~$ P_2 $ obtained
from~$ k[x,y]_{\le d^2 + d - 3} $.

(2)
If the curves defined by $ f_x $ and~$ f_y $ do not meet~$ C $
at infinity then we can obtain such~$ P_1 $ and~$ P_2 $
from~$ k[x,y]_{\le 2d-3} $.
\end{corollary}

\begin{proof}
(1)
We have~$ d(d-1) C_\infty > (f_x)_0^\infty $
because~$ (f_x) $ is the difference of the intersection divisor
of~$ C $ with a curve of degree~$ d-1 $, and~$ (d-1) C_\infty $.
So if we take~$ D = (3d-4) C_\infty $, then~$ \deg(D) - \deg(D') >  d (3d-4) - 2 d (d-1) = d (d-2) \ge 0 $.
We also have~$ L(D+(f_x)_0^\infty - (f_x)_\infty) \subseteq L(D+ d(d-1) C_\infty - (d-1) C_\infty) $.
The latter is~$ L( (d^2+d-3) C_\infty ) $, which is the 
image of~$ k[x,y]_{\le d^2 + d - 3} $ by Lemma~\ref{Olemma}.
A similar argument works for~$ P_2 $.

(2)
Here~$ (f_x)_0^\infty = (f_y)_0^\infty  = 0 $
and~$ D' $ has degree~$ 2 d (d-1) $.
We let~$ D = (3d-4) C_\infty $, which has degree~$ d (3d-4) $.
Because~$ (f_x)_\infty = (f_y)_\infty = (d-1) C_\infty $,
the proposition gives that~$ P_1 $ and~$ P_2 $ can be
taken to be in~$ L( (2d-3) C_\infty) $, 
which by Lemma~\ref{Olemma} is the image of~$ k[x,y]_{\le 2d-3} $.
\end{proof}

\begin{remark} \label{rare-rem}
Note that $ (f_x)_0^\infty \ne 0 $ and~$ (f_y)_\infty^0 \ne 0 $
in Proposition~\ref{P1P2bound}
can occur. For example if~$ F(X,Y,Z) = XY-Z^2 $, where both have
degree~1 (and in Corollary~\ref{smooth-planar-cor} the curves defined by~$ f_x $ and~$ f_y $ do meet~$ C $
at infinity).

However, if the field~$ k $ is large enough compared to~$ d $,
then we can often change coordinates in such a way that
Corollary~\ref{smooth-planar-cor}(2) applies.
In order to see, this, first
suppose that at every point~$ Q $ in~$ |C_\infty| $ we have~$ F_X(Q) \ne 0 $
or~$ F_Y(Q) \ne 0 $. Let~$ G(X,Y,Z) = F(a_1 X + a_2 Y, b_1 X + b_2 Y, Z) $
for~$ a_1 $, $ a_2 $, $ b_1 $ and~$ b_2 $ in~$ k $ with~$ a_1 b_2 - a_2 b_1 \ne 0 $.
If~$ Q' $ on~$ C_G $ corresponds to~$ Q $ on~$ C_F $, then~$ G_X(Q') = a_1 F_X(P) + a_2 F_Y(P) $
and~$ G_Y(Q') = b_1 F_X(P) + b_2 F_Y(P) $. If~$ k $ is large enough,
then we can choose~$ a_1 $, $ a_2 $, $ b_1 $ and~$ b_2 $ such that~$ G_X(Q') \ne 0 $
and~$ G_Y(Q') \ne 0 $ at all~$ Q' $, and Corollary~\ref{smooth-planar-cor}(2)
applies.
Next, suppose that~$ F_X(Q) = F_Y(Q) = 0 $ for some~$ Q = [\a, \b, 0] $ in~$ |C_\infty| $,
where~$ \a $ and~$ \b $ may be in a finite extension of~$ k $.
Write
\begin{equation*}
F(X,Y,Z) = F_d(X,Y) + F_{d-1}(X,Y) Z + \dots + F_0(X,Y) Z^d
,
\end{equation*}
with~$ F_i(X,Y) $ homogeneous of degree~$ i $, and~$ F_d(X,Y) \ne 0 $.
Because~$ Q $ is on the curve we have~$ F_d(X,Y) = (\b X  - \a Y) B(X,Y) $
for some~$ B(X,Y) $. But~$ F_X(Q) = F_Y(Q) = 0 $
is equivalent to~$ \b B(\a, \b) = \a B(\a, \b) = 0 $, so 
that~$ B(\a, \b) = 0 $. Therefore~$ F_d(X,Y) $ contains
the square of~$ \b X - \a Y $, and for most~$ F(X,Y) $ one has~$ F_d(X,Y) $
not divisible by a square.
\end{remark}

Using Corollary~\ref{smooth-planar-cor} allows us to write down the system~$ G(\S) $
in~\eqref{globalG} with some control on the degrees of the polynomials~$ P_1 $ and~$ P_2 $.
Using local expansions, gives us
the system~$ H(\S) $ in Theorem~\ref{locallift0}.

\begin{example} \label{gen-planar-ex}
As before, let~$ f $ of degree~$ d \ge 2 $ in~$ R[x,y] $ define the
affine part of a smooth curve in~$ \P_R^2 $, and put~$ A = R[x,y] / (f) $.
Assume that we have~$ P_1 $ and~$ P_2 $ in~$ R[x,y] $ of degree at most~$ D $, and such that~$ P_1 f_x + P_2 f_y \equiv 1 $
modulo~$  \pi A $.
Let us assume that we are in the generic situation, where~$ C_\infty = C \cdot \{Z = 0 \} $
consists of~$ d $ points (after extending~$ R $ if necessary),
and that they all reduce to different point on~$ k $.
We then have~$ d $ ends, and at the corresponding residue disc
we take a parameter that is centred at the unique point~$ Q $ of~$ C_\infty $
in that disc. Then~$ x $ and~$ y $ will have poles at each~$ Q $
of order at most (and in general equal to)~$ -1 $.
Then~\eqref{globalG}, as made explicit at the beginning
of Section~\ref{planecurvesection}, results in
\begin{equation} \label{H-planar} 
H(S) = \sum_{i=0}^d H_i S^i
\end{equation}
in~$ R[[t]] (S) $, as in Theorem~\ref{locallift0}.
With notation as in Notation~\ref{Ralphabetanotation}, we have~$ H_i $ in~$ t^{- \pa ( i (D-1) + d )} R[[t]] $,
$ H_0 $ in~$ t^{-\pa d} \cap V_{1/e} $,
and~$ H_1 $ is in~$ \Rtab {\nu} 0 $ for~$ \mu = e \pa (d+D-1) $.
Then~$ H_i $ for~$ i = 2, \dots, d $, is in~$ \Rtab {\mu} {\b_i} $
with~$ \b_i = - \pa ( i (D-1) + d) $.
We apply Proposition~\ref{hensel2} with~$ z = 0 $, as Newton
iteration determines the solution~$ \ts $ in~$ \pi \Rtstar $.
We can take $ b = c = 1/e $ and~$ \rho = \nu = \pa (D-1) $ in
conditions~(1) through~(4) in the proposition.
Then
\begin{alignat*}{1}
\mu' & = \mu + \maxp_{2 \le i \le d, 0 < k \le i - 1} \left\{ -\frac{\b_i + (i-k-1) \nu + k \rho}{(i-k-1)b + k c}\right\}
\\
& =
\mu + \maxp_{2 \le i \le d} \left\{ -\frac{- \pa ( i (D-1) + d)  + \pa (i-1) (D-1) }{(i-1)/e }\right\}
\\
& =
\mu + e \pa \maxp_{2 \le i \le d} \left\{ \frac{  i (D-1) + d  - (i-1) (D-1) }{(i-1) }\right\}
\\
& =
\mu + e \pa (d+ D-1)
\\
& =
2 \mu 
,
\end{alignat*}
$ b' = b = 1/e $, and~$ \nu' = \nu + \pa (d + D -1 )= \pa (d + 2 D - 2 ) $,
and by the proposition all Newton iterates, and~$ \ts $ itself, are in~$ \Rtab {\mu'} {\nu'} $.
It follows from~\eqref{local-frob-expansion} that~$ \ex(\phidagA(x)) = \ex(x)^{\pa} + \ex( \psi(P_1)) \ts $
and~$ \ex(\phidagA(y)) = \ex(y)^{\pa} + \ex( \psi(P_2)) \ts $
are in~$ \Rtab {2\mu} { \nu ''} $, with~$ \nu'' = \nu' +  \pa (D-1) = \pa (d + 3 D - 3 ) $.
Finally, we note that Remark~\ref{rare-rem} implies that in the
general situation we are in the situation of Corollary~\ref{smooth-planar-cor}(2), 
and can assume~$ D \le 2 d - 3 $.
(By contrast, in Corollary~\ref{locesti} we have
\begin{equation*}
M
= \maxp_{2 \le i \le d  }\{- (\b_i + (i-1)\nu) \}
= \pa (d + D - 1)
= e^{-1} \mu
,
\end{equation*}
and its estimate gives that~$ \ts $ is in $\Rtab {3 \mu} {\b} $, where~$ \b = \min_{2 \le i \le d} \{\b_i\} = - \pa d D $.)
\end{example}

\medskip

In the remainder of the section we explain how to obtain an $R$-basis of $\hdr^1(C/R)$ by
means of Proposition~\ref{compdr} and~\eqref{CompID}, by writing down more directly~$ R $-bases
of~$H^0(C, \O_C(D))$, $ H^0(C, \O_C(D)/\O_C)$ and~$H^0(C, \Omega_{C/R}^1(D'))$
than is done by the more general Algorithms~\ref{KtoRbasis} and~\ref{KtoRbasisOmega}
and Proposition~\ref{newOCDprop}.

\begin{notation} \label{notation-smooth-planar}
Let $H$ be the hyperplane in $\P_R^2$ defined by $Z = 0$,
$D = n C_\infty$ for~$ C_\infty = C \cdot \{Z = 0\}$,
and $D' = (n+1)C_\infty$, where $n \ge 1$.
Let $Y = \P_R^2 \setminus H$ be the affine plane in $\P_R^2$ with coordinates $x, y$,
and $X = C \cap Y$ the affine part of $C$ with coordinate ring $A = R[x, y]/(f)$.
\end{notation}

\begin{proposition} \label{R-basis-image-d}
Let $U'$ be any affine chart of $\P_R^2$ with coordinates $u, v$ such that $U'$ contains the support of the divisor $D$,
and let $U = C \cap U'$ be the new affine part of $C$.
Then $H^0(C, \O_C(D)/\O_C)$ can be identified with the cokernel of the map
$H^0(U, \O_C) \to H^0(U, \O_C(D))$.
Moreover, if we assume that $g(u, v)$ is the defining equation of $U$ of degree $d$,
$h(u, v)$ is the defining equation of~$ \{ Z = 0 \} $,
and that $v^d$ has non-zero coefficient in $g(u, v)$,
then $H^0(C, \O_C(D)/\O_C)$ can be identified with the cokernel of multiplication by $h^n$ from
$\oplus_{0 \le j \le d-1} R[u]v^j$ to $\oplus_{0 \le j \le d-1} R[u]v^j$.
\end{proposition}

\begin{proof}
Since the support of $\O_C(D)/ \O_C$ is the same as the support of $D$,
we have
\begin{equation*}
H^0(C, \O_C(D)/\O_C) = H^0(U, \O_C(D)/\O_C) = \text{Coker}(H^0(U, \O_C) \to H^0(U, \O_C(D))).
\end{equation*}
If $v^d$ has non-zero coefficient in $g(u, v)$, then we have
$H^0(U, \O_C) = \oplus_{0\le j \le d - 1} R[u]v^j$
and $H^0(U, \O_C(D)) = \oplus_{0 \le j \le d-1} h^{-n} R[u]v^j$,
thus proving the statement.
\end{proof}

The following algorithm explains how to explicitly compute an $R$-basis of the cokernel
in Proposition~\ref{R-basis-image-d}.

\begin{algorithm} \label{Rx-algo}
Let $ R $ be a discrete valuation ring, $ K $ its
field of fractions, and~$ x $ a variable.
Suppose that~$ N $ is an~$ R[x] $-submodule of~$ R[x]^n $, given
by a finite set of generators, and let~$ M = R[x]^n/N $.
If $ M $ as~$ R $-module is finitely generated and torsion free,
then one can compute an $ R $-basis of it as follows.

(1)
First we compute the structure of $M \otimes_R K$ as a finitely generated $K[x]$-module. 
This amounts to finding an automorphism of $K[x]^n$ as $K[x]$-module that induces an isomorphism 
 $ \phi : M \otimes_R K \rightiso K[x]/(d_1(x)) \times \dots \times K[x]/(d_n(x)) $ 
for (unique) monic~$ d_i(x) $ in $ K[x] $ with $ d_1(x) | d_2(x) | \dots | d_n(x) $.

(2)
Then each $d_i(x)$ is in~$R[x]$.  In order to see this, let $M' = \phi(M)$
be the image of~$ M $, which is an $R[x]$-submodule as well as a finitely generated $R$-module.
If~$ \a $ in a finite extension of~$ K $ is a root of~$ d_i(x) $,
then~$ M' $ cannot map to~0 under
the homomorphism obtained by projection to the $ i $th position
followed by the surjection~$ K[x]/(d_i(x)) \to K(\a) $ that sends~$ x $
to~$ \a $, because $ M' $ generates the target of~$ \phi $ as $ K $-vector space.
Then the non-zero image of~$ M' $ in~$ K(\a) $ is closed under
multiplication by~$ \a $, so it is a faithful~$ R[\a] $-module
that is finitely generated as~$ R $-module. Thus~$ \a $ is integral
over~$ R $ by \cite[p.334]{lang93}, hence~$ d_i(x) $ is in~$ R[x] $.

(3)
In particular, $ M \subseteq M \otimes_R K $ is annihilated
by~$ d_n(x) $.
Let~$ m = \deg(d_n(x)) $, and
let~$ v_1,\dots, v_n $ be the~$ n $ elements~$ (0,\dots, 0, 1, 0, \dots, 0) $ in $ R[x]^n $.
Then the classes in~$ M = R[x]^n/N $ of the~$ x^i v_j $ for $ i=0,\dots, m-1 $ and~$ j=1,\dots,n $
generate~$ M $ as~$ R $-module, and their images under~$ \phi $ generate~$ M' $.
Identifying each~$ K[x]/(d_i(x)) $ with~$ K^{\deg(d_i(x))} $,
it is easy to compute an~$ R $-basis of $ M' $, consisting of~$ R $-linear expressions in these images,
because~$ R $ is a discrete valuation ring.
The corresponding~$ R $-linear expressions
of the classes of the~$ x^i v_j $ then form an~$ R $-basis of~$ M $.
\end{algorithm}

Next, an $R$-basis of $H^0(C, \O_C(D))$ can be obtained as follows.

\begin{proposition} \label{planarH0OCD}
Let the notation be as in Notation~\ref{notation-smooth-planar},
with~$ n \ge d-2 $. Then $H^0(C, \O_C(D)) =  \oplus_{i+j \le n} R x^iy^j /\oplus_{i+j \le n - d} R fx^iy^j $.
\end{proposition}

\begin{proof}
The short exact sequence
$$ 0 \to \O_{\P_R^2} (-C + nH) \to \O_{\P_R^2}(nH) \to \O_C(D) \to 0$$
of sheaves results in a short exact sequence of the corresponding
global sections as~$H^1(\P_R^2, \O_{\P_R^2} (-C + nH))$ is trivial.
The result now follows from
$H^0(\P_R^2, \O_{\P_R^2} (-C + nH)) \cong \oplus_{i+j \le n - d} R fx^iy^j$
and
$H^0(\P_R^2, \O_C(nH)) \cong \oplus_{i+j \le n} R x^iy^j$.
\end{proof}

The next result describes how to get an $R$-basis of $H^0(C, \Omega_{C/R}^1(D'))$.

\begin{proposition} \label{P2C-prop}
Let~$ A = R[x,y]/(f) $ with~$ f $ as before.
For~$n \ge \max\{2d - 3, d\}$ consider the commutative diagram
\begin{equation*}
\xymatrix{
0 \ar[r] & M_1 \ar[r]^-{\cc} \ar[d] & M_2  \ar[r] \ar[d] & M_3 \ar[r] \ar[d] & 0
\\
0 \ar[r] & R[x,y] \ar[r]^-{\cdot \dd f } & 
A \dd x \oplus A \dd y \ar[r]^-{\d} &
\Omega^1_{A/R} \ar[r] & 0
}
\end{equation*}
of~$ R $-modules,  with exact rows, 
where all vertical maps are inclusions,
$ M_1 $ is generated by all $x^iy^j$ with $i + j \le n + 1 - d$,
$ M_2 $ by all~$x^iy^j \dd x$ and $ $  $x^iy^j \dd y$ for $i + j \le n-1$,
together with all~$x^iy^j (y\dd x - x\dd y)$ for $i + j = n -1$,
$ \cc $ is multiplication by~$ \dd f = f_x \dd x + f_y \dd y $,
and~$ \d $ is the quotient map to~$ \Omega_{A/R}^1 = (A\dd x \oplus A \dd y )/A \dd f $.

Then~$M_3 = \d(M_2) $ is the image of $H^0(C, \Omega_{C/R}^1(D'))$ in~$ \Omega_{A/R}^1 $
under localisation for~$ D' = (n+1) C_\infty $, where~$ C_\infty = C \cdot \{ Z = 0 \} $.
Moreover in~$ A \dd x + A \dd y $ one can effectively compute
an~$ R $-basis of~$ M_2 / \cc(M_1)) $ by Remark~\ref{coordringbasis}.
\end{proposition}

\begin{proof}
The short exact sequence of sheaves
\begin{equation} \label{sos}
0 \to \I/\I^2\otimes \O_{\P_R^2}(n+1)
\xrightarrow{\dd \otimes \text{id}}
\Omega^1_{\P_R^2/R} \otimes \O_C \otimes \O_{\P_R^2}(n+1)
\to
\Omega^1_{C/R}(D')
\to 0
,
\end{equation}
where~$\I = \O_{\P_R^2}(-C)$ is the ideal sheaf of~$ C $,
is the second exact sequence of differentials, twisted by~$  \O_{\P_R^2}( (n+1) H) $
for~$ H $ the line at infinity, so~$ D' = (n+1) C_\infty $.
Then~$ H^1( \P_R^2 , \I/\I^2\otimes \O_{\P_R^2}(n+1) ) $ is trivial.
This can be seen by forming the long exact sequence of cohomology
groups for the short exact sequence 
\begin{equation*}
0 \to \O_{\P_R^2}(-2C)(n+1) \to \O_{\P_R^2}(-C)(n+1) \to \I/\I^2\otimes \O_{\P_R^2}(n+1) \to 0
\end{equation*}
of sheaves, and using that $H^1(\P_R^2, \O(k)) = 0$ for all $k$,
and $H^2(\P_R^2, \O_{\P_R^2}(-2C + (n+1)H)) = H^2(\P_R^2, \O(-2d + n + 1)) = 0$ for $n + 1 - 2d > -3$
(cf.~\cite[p. 225]{hart77}).
Moreover, the triviality of~$ H^1( \P_R^2 , \O_{\P_R^2}(-2C)(n+1) )  $
implies that~$ H^0( \P_R^2 , \I/\I^2\otimes \O_{\P_R^2}(n+1) ) $ is generated
by the image of~$ H^0( \P_R^2 , \O_{\P_R^2}(-C)(n+1) ) $.
As an~$ R $-module, this is generated by all $fx^iy^j$ for $i + j \le n + 1 - d$
under the natural projection map (cf.~Lemma~\ref{Olemma}).
This can be easily checked by dimension counting
as $\dim H^0( \P_R^2 , \O_{\P_R^2}(-C)(n+1) ) = \binom{n+3 - d}{2}$
coincides with the number of elements $\{f x^iy^j | i + j \le n + 1 - d\}$.
Therefore the image of~$ H^0( \P_R^2 , \I/\I^2\otimes \O_{\P_R^2}(n+1) ) $ under
localisation to~$ \A_R^2 $ gives~$ M_1 $ as in
the proposition.

Similarly, we have a short exact sequence of sheaves
\begin{equation*}
0 \to \Omega^1_{\P_R^2/R}(-C)(n+1) \to
 \Omega^1_{\P_R^2/R}(n+1) \to \Omega^1_{\P_R^2/R} \otimes \O_{\P_R^2}(n+1) \otimes \O_C \to
0
.
\end{equation*}
We want to see that~$ H^1 ( \P_R^2 , \Omega^1_{\P_R^2/R}(-C)(n+1) ) $ is trivial.
For this, consider the twisted Euler sequence 
\begin{equation*}
0 \to \Omega_{\P_R^2}^1(n + 1 - d) \to \O_{\P_R^2}(n - d)^{\oplus 3} \to \O_{\P_R^2}(n + 1 - d) \to 0
.
\end{equation*}
Then~$H^1(\P_R^2, \P(k))$ is zero for all $k$.
Also, the corresponding sequence on global sections is short
exact for $n - d \ge 0$, as one sees by explicit calculation (see below).
Therefore~$ H^1 ( \P_R^2 , \Omega^1_{\P_R^2/R}(-C)(n+1) ) $ is trivial.

On the other hand, in order to get generators of $ H^0 ( \P_R^2 ,  \Omega^1_{\P_R^2/R}(n+1) ) $,
consider the Euler sequence
\begin{equation*}
0 \to \Omega_{\P_R^2/R}^1(n + 1) \to \O_{\P_R^2}(n)^{\oplus 3} \to \O_{\P_R^2}(n+1) \to 0,
\end{equation*}
where the first map sends $a \dd x + b \dd y$ to $(a, b, -ax - by)$, and the second sends $(a, b, c)$ to $ax + by +c$
for $a, b, c$ in $\O_{\P_R^2}(n)$.
This implies that a $1$-form $a\dd x + b\dd y$ is in $H^0(\P_R^2, \Omega_{\P_R^2}^1(n+1))$ 
if and only if~$a$ and $b$ are in $H^0(\P_R^2, \O_{\P_R^2}(n))$ and there exists some $c$ in $ H^0(\P_R^2, \O_{\P_R^2}(n))$
such that $c = -ax - by$.

For $a, b$ in $H^0(\P_R^2, \O_{\P_R^2}(n-1))$ such a $c$ always exists,
i.e., all $x^iy^j\dd x$ and $x^iy^j\dd y$ with $i + j \le n - 1$ are in $H^0(\P_R^2, \Omega_{\P_R^2}^1(n+1))$.
If $a = x^iy^{n - i}$ has a pole of order exactly~$n$ at infinity, then $b$ can be uniquely determined by setting $ax + by = 0$
to make sure that $c$ does not have a pole of order $n + 1$ at infinity.
Thus all $x^iy^{n - 1 - i}(y \dd x - x\dd y)$ for~$ i = 0, \dots, n-1 $
are also in $H^0(\P_R^2, \Omega_{\P_R^2}^1(n+1))$.
These give a basis of $H^0(\P_R^2, \Omega_{\P_R^2/R}^1(n+1))$ as 
the number of these elements coincides with  the rank of
$H^0(\P_R^2, \Omega_{\P_R^2/R}^1(n+1))$, which is
\begin{equation*}
3\rk H^0(\P_R^2, \O_{\P_R^2}(n)) - \rk H^0(\P_R^2, \O_{\P_R^2}(n+ 1))
= 3\binom{n+2}{2} - \binom{n+3}{2} = n(n+2)
.
\end{equation*}
Therefore~$ M_2 $ is generated by the elements given in the proposition,
under the localisation from~$ \P_R^2 $ to~$ \A_R^2 $ followed
by the pullback to~$ C \setminus C_\infty $.

Finally, $ M_3 $ is the image of~$ H^0( C , \Omega^1_{C/R}(D') ) $
under the localisation to~$ C \setminus C_\infty = \Spec(A) $.
Therefore the bottom row in the diagram in the proposition corresponds
to the global sections of the sheaves in~\eqref{sos}
under the localisation to~$ \A_R^2 $, $ \Spec(A) $ and~$ \Spec(A) $,
respectively, and the upper row identifies with the short exact
sequence of the global sections of the sheaves in~\eqref{sos}.
\end{proof}

Therefore~\eqref{CompID} for~$ D = n C_\infty $ and~$ D' = (n+1) C_\infty $
is
\begin{equation} \label{planarCompID}
\begin{split}
\xymatrix{
  H^0(C, \I(D)) \ar[r] \ar[d] & H^0(C, \Omega_{C/R}^1((n+1)C_\infty)) \ar[d] \\
  H^0(C, \im(\dd)) \ar[r] & H^0(C, \Omega_{C/R}^1((n+1)C_\infty) / \Omega_{C/R}^1)
  .
  }
\end{split}  
\end{equation}
Algorithm~\ref{Rx-algo} enables us to compute~$ H^0(C, \im(\dd)) $,
and~$ H^0(C, \Omega_{C/R}^1((n+1)C_\infty)) $ can now be calculated
using Proposition~\ref{P2C-prop}.
We can now compute~$ H^0(C, \I(D)) $ inside~$ H^0(C, \Omega_{C/R}^1((n+1)C_\infty)) $,
treating~$ H^0(C, \Omega_{C/R}^1((n+1)C_\infty) / \Omega_{C/R}^1) $
just as we did at the very end of Section~\ref{de-rham}.

For~$ n \ge \max\{2d-3,d\} $ all hypotheses are satisfied, as
is Assumption~\ref{D-D'-assumption}, that is, $n \ge d - 2$.
Then~$ \hdr^i(C/R) $ for~$ i = 0 $ and~1 coincide with the kernel and cokernel of~$ H^0(C, \O_C(D)) \xrightarrow{\dd} H^0(C, \I(D)) $.

\section{The complexity of the algorithm for smooth planar curves} \label{complexity}

In this section we give a rough estimate for the complexity of
Algorithm~\ref{est-algorithm} for smooth planar curves,
based on our earlier discussion in Example~\ref{gen-planar-ex}.

We have made various simplifying assumptions. We
shall be using the soft $O$ notation $\ot$, meaning that logarithmic factors
are neglected compared with polynomial ones, so that for example
$O(l\log^9(l)) = \ot(l)$.

To simplify matters, we assume~$K$ is unramified over $\Qp$, that $p' = p $
and that~$C\setminus  \Spec(A)$
consists of a finite number of disjoint $R$-sections. Not assuming
this probably does not change the complexity much because one is
typically working over a larger extension, but at the same time the
results of the computation, being Galois conjugates of one another,
can be computed once for a bunch of points.

We let $q=p^l$ be the size of the residue field $ k $. We
assume that the plane model is given by a polynomial~$ f $ of degree $d $,
so $d= O(\sqrt{g}) $ and the number of annuli we need to compute
residues at is also $O(d) $ as they correspond to the points
at infinity of~$ C = C_f $.
\begin{prop}\label{complprop}
The asymptotic complexity of this algorithm is $\ot(p g^5 l^3)$.
\end{prop}
\begin{proof}
By the discussion at the beginning of Section~\ref{algorithm-estimates}
and the estimates in Algorithm~\ref{est-algorithm}, we need to compute with $p $-adic precision  $N=\ot(g l)$.
(Note that we are assuming that~$ e = 1 $, so~$ h = 0 $ in Proposition~\ref{ep-prop}.)

The computation of the matrices $M_1$ and $M_2$ involve computation
of cup products, which in turn decomposes into certain residue
computations as described in Section~\ref{cup-product-estimates}. We consider the
computations of $M_2$, as these are clearly more time consuming.
Since we assumed $K $ was unramified, the computation is done
integrally, i.e., without denominators, by Corollary~\ref{h-cor}.

By Section~\ref{cup-product-estimates} the complexity of the residue
computation is controlled by the
parameter $\a$ of overconvergence. Indeed, the images under~$ \phidagA $
or our 1-forms are in some~$ \Rtab {\a} {0} \, \dd t $,
as defined in Notation~\ref{RtabNdef} (it is clear that from the point of view of the asymptotic
complexity the parameter $\b$ may be neglected). Since we are
considering the coefficients up to precision $N $, we need to take
negative powers of $t $ in  $  \Rtab {\a} {0} $, up to $-N \a $, as after that the coefficients become divisible
by $p^N $. The number of positive powers of $t $ is roughly $N \a $ times the number
of multiplications of elements of $ \Rtab {\a}
{0} $ we will do, as each multiplication reduces the highest power of
$t $ (see Proposition~\ref{Prodprec}).

As the convergence of the solution is with
respect to the $ p $-adic topology, the number of iterations is
proportional to the log of the $p$-adic precision, which is~$N$,
hence ultimately to $\log(l)$. This may be swallowed in the soft $O $
notation. Each iteration will involve the evaluation of the polynomial~$ H(S) $
in~\eqref{H-planar} and its derivative, both of degree $\sim d $.
  Thus, the number of powers of $t $ we need to consider
is $O(N \a d) $ and the complexity of working with this ring is $\ot(
N \a d) $, which is  $\ot(
N^2 \a d) $ operations in the residue field $k$. As
this has size $p^l$, operations take $\ot(\log(p^l))= \ot(l)$.

As for the actual work, the dominant term would be the substitution of
the local Frobenius into our basis of forms.
This has to be done for $2 g $
forms and $d $ residue discs and involves $d $ operations in  $ \Rtab {\a} {0} $. The overall complexity is thus
\begin{equation}
  \ot \left( g d^3 N^2 \a l \right) = \ot \left( (gdl)^3 \a \right) =
  \ot \left( g^{9 /2} l^3 \a \right)
.
\end{equation}
From Example~\ref{gen-planar-ex} we get $\a = O(p d) $, proving the result.
\end{proof}

\appendix

\section{Cup products in algebraic de Rham cohomology} \label{algebraic-cup-products}

In this appendix we review the description of $\hdr^1(C_K/K)$ in terms of
differential forms of the second kind, as well as a formula, due to Serre,
describing the cup product pairing~\eqref{eq:cppair}  in terms of residues and integrals.
This material is well-known, but we could not find a suitable reference.
The pairing is compatible with field extensions,
and, for convenience, we set up the discussion for a smooth complete curve $X$
over an algebraically closed field $K$ of characteristic~0.

Recall that a meromorphic form on $X$ is said to be of the second kind if
all of its residues are~0. For a function~$f$ in~$ K(X)$, the form~$\dd f$ is of the second kind.

\begin{theorem} \label{secondkindthm}
The space 
\begin{equation*}
  \frac{\{ \textup{forms $\omega$ of the second kind on $X$} \}}{\{ \textup{forms $\dd f$ for $f  $ in~$K(X)$} \}}
\end{equation*}
is isomorphic to $\hdr^1(X / K)$.
\end{theorem}
\begin{proof}
We can write a flasque resolution of~$\O_X$
by quasi-coherent $ \O_X $-modules as
\begin{equation*}
  K(X) \to \bigoplus_{x\in X} i_{x,\ast} K(X)/ \O_{X,x}
,
\end{equation*} 
where the sum is over all closed points of~$X$, and similarly for 1-forms.
The differentials are compatible, so we can form the double complex
\begin{equation} \label{doublecomplex}
\begin{split}
    \xymatrix{
	K(X) \ar[r] \ar[d]^{\dd}  & \bigoplus_{x\in X} \frac{ i_{x,\ast} K(X)}{ \O_{X,x} }
	\ar[d]^{\dd}\\
   \Omega_{K(X)}^1 \ar[r]   & \bigoplus_{x\in X} \frac{ i_{x,\ast} \Omega_{K(X)}^1  }{ \Omega_{X,x}^1 }
}
\end{split} 
\end{equation}
whose total complex we denote by $T^{\bullet}$. Recall that for a cohomological double complex with $d_1: A^{pq}\to A^{p+1,q} $ and $d_2: A^{pq} \to A^{p,q+1} $ then in the total complex terms the differential is $d_1+(-1)^{p} d_2 $. Thus,
explicitly,
\begin{alignat*}{1}
T^0&=K(X),
\\
T^1&=
\left(\bigoplus_{x\in X} i_{x,\ast} K(X)/\O_{X,x}\right)
\oplus\Omega^1_{K(X)}
\\
\intertext{and}
T^2&=
\bigoplus_{x\in X}  i_{x,\ast} 
\Omega^1_{K(X)}/\Omega^1_{X,x}
,
\end{alignat*}
with differentials
\begin{alignat}{1}
\nonumber
f&\longmapsto ((f_x)_x,\dd f)
\\
\intertext{and}
\label{deg2}
((g_x)_x,\omega) &\longmapsto
(\dd g_x-\omega_x)_x
.
\end{alignat}
There is a natural map~$ \Omega_X^\bullet \to T^\bullet $ of
complexes, and
the complex $\Gamma(X,T^{\bullet} )$ of global sections of~$T^{\bullet}$
computes the algebraic de Rham
cohomology of $X$. The closed terms in degree~1 of this complex consist of pairs
$((f_x)_{x \in X},\omega)$, where~$\omega$ is a meromorphic form on $X$,
while each~$f_x$ is a germ of a meromorphic function at~$x$, modulo regular
functions, such that $\omega \equiv \dd f_x $ modulo a regular differential at~$x$.
In characteristic~0 this is equivalent to~$\omega$ being
a form of the second kind on~$X$, and
the~$f_x$ are uniquely determined by~$\omega$. The exact
terms in degree~1 are pairs~$( (f|_x)_{x \in X},\dd f)$ with~$f$ in~$K(X)$.
So mapping~$ ((f_x)_{x \in X},\omega)$ to~$ \omega $
gives the result.
\end{proof}

\begin{remark} \label{refinedremark}
Proposition~\ref{compdr} is a more refined version over our~$ R $, but by localising it
to~$ K $ we obtained a version over an arbitrary field of characteristic~0,
with more control on the order of the poles involved.
\end{remark}

We now interpret the cup product pairing~\eqref{eq:cppair} using
forms of the second kind.

\begin{theorem}\label{seckindcup}
    Let $\omega$, $\eta$ be forms of the second kind on $X$, and let
    $[\omega]$, $[\eta]$ be their respective cohomology classes in
    $\hdr^1(X /K)$. Then 
    \begin{equation*}
    \tr ([\omega]\cup [\eta]) =  \sum_x \res_x \left(( \tint \omega)\cdot  \eta \right)
    .
  \end{equation*}
Here $\int \omega  $ is the Laurent series obtained by termwise
integration of a local expansion of $\omega  $ at~$ x $. This determines
$\int \omega  $ up to a constant, but this constant does not change the
expression because~$\res_x (\eta)  = 0 $.
\end{theorem}

We need some preparation before we can prove the theorem. We first give an explicit description of the trace map in algebraic de Rham cohomology,
\begin{equation*}
  \tr: \hdr^2(X /K)\to K\;.
\end{equation*} 
We have an isomorphism
\begin{equation*}
  \hdr^2(X /K)\cong H^2(\Gamma(X, T^\bullet))\;.
\end{equation*}
Elements of the right hand side are represented by  collections $(\omega_x)_{x\in X}$, where $\omega_x \in \Omega^1_{K(X)}/\Omega^1_{X,x} $. An element of $\Omega^1_{K(X)}/\Omega^1_{X,x} $ has a well-defined
residue at $x$. This residue vanishes on the restriction of~\eqref{deg2} to the first component, as residues of exact differentials are $0 $, and the sum of all the residues vanishes also on the restriction to the second component by  the residue theorem. It therefore defines a map,
\begin{equation}\label{eq:trdef}
  \tr:  H^2(\Gamma(X, T^\bullet))\to K, \quad \tr((\omega_x)_x) = -\sum_x \res_x (\omega_x)\;.
\end{equation} 
\begin{proposition}
  The composed map $ \hdr^2(X /K)\cong H^2(\Gamma(X, T^\bullet)) \xrightarrow{\tr} K $  is the trace map in algebraic de Rham cohomology.
\end{proposition}
\begin{proof}
The trace map is obtained by composing the inverse of the isomorphism
\begin{equation*}
  H^1(X,\Omega_X^1) \xrightarrow{\sim} \hdr^2(X/K)
\end{equation*}
with the trace map for Serre duality, defined, for example,
in~\cite[p.248]{hart77}. The isomorphism above, coming from the
hypercohomology spectral sequence, is induced from the map
\begin{equation*}
  \Omega_X^1[1] \to \Omega_X^\bullet
,
\end{equation*}
and therefore also by the  inclusion
of the second line of~\eqref{doublecomplex}
into the total complex~$ T^\bullet $, but the shift introduces a sign.
The trace map of~\cite{hart77} is defined on this second line by~$-\tr$, whence the
result.
\end{proof}

We now describe de Rham cohomology in terms of \v Cech cocycles.  In
particular, we shall show how to associate such a cocycle to a form of the second
kind, and also show how to compute the trace map in term of \v Cech cocycles.

We fix a finite affine covering $U_1,\dots,U_n$ of $X$. In terms of this covering, de
Rham cohomology is the cohomology of the total complex associated to the
double complex $C^{\bullet}(\Omega_X^{\bullet})$ of \v Cech cochains in the
sheaves of differential forms. This follows since  for quasi-coherent~$ \O_X $-modules  \v Cech cohomology
coindices with cohomology using injective resolutions~\cite[Theorem~III.4.5]{hart77}. In all double complexes involving \v Cech cocycles we will let the \v Cech differential be the first differential and the differential induced from the complex of sheaves be the second differential. Concretely, an element of $\hdr^1(X /K)$ is
given by the data $(\omega_i, f_{ij})$, where~$\omega_i$ is in~$ \Omega_X^1(U_i)$,
$f_{ij} $ is in~$ \O_X(U_{ij})$, one has~$ \omega_j-\omega_i = \dd f_{ij} \text{ on } U_{ij} $,
and the $f_{ij}$ satisfy the usual cocycle condition.
These are taken modulo those of the form $(\dd f_i,f_j-f_i)$ with~$f_i$
in~$\O_X(U_i)$. Similarly, elements of $\hdr^2(X /K)$ are given by data
$(\omega_{ij},f_{ijk})$ where the $f_{ijk}$ satisfy the cocycle condition
and ~$ \omega_{jk}-\omega_{ik}+\omega_{ij}= -\dd f_{ijk} $.
These are taken modulo the obvious relations.

To pass from a representation of de Rham cohomology in terms of the
complex~$\Gamma(X,T^{\bullet})$
and a representation in terms of \v Cech cocycles, we
consider the \v Cech complex~$ C^\bullet(T^\bullet) $ corresponding to the complex $T^{\bullet}$. This gives a double
complex whose total complex computes de Rham cohomology as well, and~$ C^\bullet(\Omega_X^\bullet) $
as well as the complex~$ \Gamma(X,T^{\bullet})$ map to it.
The result in low degrees is the following diagram, where
the top row corresponds to the total complex associated to~$ C^\bullet(\Omega_X^\bullet) $,
the middle row to that of~$ C^\bullet(T^\bullet) $,
and the bottom row to~$ \Gamma(X, T^\bullet) $.
\begin{equation} \label{difhtwo}
\begin{split}
\resizebox{300pt}{!}{%
\xymatrix{
C^0(\O_X) \ar[r] \ar[d] & C^0(\Omega_X^1)\oplus C^1(\O_X) \ar[r] \ar[d] & C^1(\Omega_X^1)\oplus C^2(\O_X) \ar[r] \ar[d] & \dots
\\
C^0(T^0) \ar[r] & C^0(T^1)\oplus C^1(T^0) \ar[r] & C^0(T^2)\oplus C^1(T^1) \oplus C^2(T^0) \ar[r] & \dots
\\
\Gamma(X, T^0) \ar[r] \ar[u] & \Gamma(X, T^1) \ar[r] \ar[u] & \Gamma(X, T^2) \ar[r] \ar[u] & 0
}%
}
\end{split}
\end{equation}
Here the maps from the first to the second row are obtained from the
map~$ \Omega_X^\bullet \to T^\bullet $ of complexes, and the
maps from the last to the second row from the inclusions
of~$ \Gamma(X, T^i) $ into~$ C^0(T^i) $.
In the second row, we more explicitly have
\begin{alignat*}{1}
C^0(T^0) & =
\oplus_i K(X)
\\
C^0(T^1) \oplus C^1(T^0) & =
\oplus_{i; x\in U_i} \frac{ i_{x,\ast} K(X) }{ \O_{X,x}} \bigoplus \oplus_i \Omega^1(K(X)) \bigoplus \oplus_{i,j} K(X) 
\end{alignat*}
and~$  C^0(T^2)\oplus C^1(T^1) \oplus C^2(T^0) $ equal to
\begin{equation*}
\oplus_{i; x\in U_i} \frac{ i_{x,\ast} \Omega^1(K(X)) }{ \Omega_{X,x}^1 }
\bigoplus 
\oplus_{i,j; x\in U_{ij}} \frac{ i_{x,\ast} K(X) }{ \O_{X,x} }
\bigoplus
\oplus_{i,j} \Omega^1(K(X))
\bigoplus
\oplus_{i,j,k}K(X)
,
\end{equation*}
with the map in degree~0 mapping~$ (f_i) $ to
\begin{equation} \label{second-row-diff-0}
 (f_{i,x} ,\dd f_i, f_j-f_i) 
\end{equation}
and the one in degree~1 mapping~$   ( g_{i,x} , \omega_i,f_{ij}) $ to
\begin{equation} \label{second-row-diff-1}
( \dd g_{i,x} - (\omega_i)_x, g_{j,x}-g_{i,x}- f_{ij,x} , \omega_j-\omega_i - \dd f_{ij}, f_{jk}-f_{ik}+f_{ij})
.
\end{equation}

Diagram~\eqref{difhtwo} allows us to pass from cohomology classes in the
top row, represented by \v Cech cocycles, to cohomology classes on the bottom
row, represented by forms of the second kind for $H^1$, or by elements in $
H^2(\Gamma(X,T^{\bullet}))$
from which we can compute the trace map.

\begin{proposition}\label{prop:ctof}
The \v Cech cocycle $(\omega_i,f_{ij})$ and the meromorphic form $\omega$ of the second kind represent the same class in $\hdr^1(X /K)$ if and only if there are 
meromorphic functions $f_i$ with~$f_{ij}=f_j-f_i$,
and $\omega|_{U_i}=\omega_i- \dd f_i$.
\end{proposition}

\begin{proof}
Recall that the form of the second kind $\omega$ corresponds to an element 
$(g_x,\omega)$ in the degree $1$ term of the bottom row of~\eqref{difhtwo}
such that $\omega \equiv \dd g_x $ modulo a regular differential at~$x$.
Suppose this has the same image in the cohomology of the middle row as the
cocycle
$(\omega_i,f_{ij})$. This means that they are the same up to a coboundary.
In other words,  we have~$f_i$ in $ \oplus_i K(X)$ such that
\begin{equation*}
  (0,\omega_i,f_{ij}) = (g_x, \omega|_{U_i},0)+
  ( (f_i)_x, \dd f_i, f_j-f_i)\;.
\end{equation*} 
In particular, $f_{ij}= f_j-f_i$ and $\omega|_{U_i}= \omega_i-\dd f_i$.
  Conversely, suppose the $f_i$ satisfy the above condition. Then,
using~\eqref{second-row-diff-0} we have
\begin{equation*}
  (0,\omega_i,f_{ij}) - (g_x, \omega|_{U_i},0)-
  ((f_i)_x,\dd f_i, f_j-f_i) = (h_{i,x},0,0) 
\end{equation*}
for some germs of rational functions $h_{i,x} = - g_x - f_{i,x}$ at $x$, taken module germs
  of holomorphic functions. This last element is closed, implying in
  particular that $\dd h_{i,x}$ is holomorphic, hence $h_{i,x}$ itself is
  holomorphic, hence $0$; so we indeed get the same cohomology class. 
\end{proof}
\begin{remark}\label{ratflsq}
   The existence of rational functions $f_i$ such that $f_{ij}= f_j-f_i$ followed by cohomological arguments in the above proof,  but it is immediate for any rational $f_{ij} $ as flasqueness
  implies the vanishing of the higher \v Cech cohomology groups
with coefficients in~$ K(X) $~\cite[Proposition~III.4.3]{hart77}.)
\end{remark}

Proposition~\ref{prop:ctof} does not tell us how to pass from one representation to
the other. It is straightforward though to make this more concrete,
yielding the following.

\begin{proposition}\label{prop:ctof1}
  To get from a cocycle $(\omega_i,f_{ij})$ to a form~$ \omega $ of the second kind
  that represents the same cohomology class,
pick, using Remark~\ref{ratflsq}, any
  rational functions $f_i$ such that $f_{ij}= f_j-f_i$.
On~$ U_{ij} $ we then have
\begin{equation*}
(\omega_j - \dd f_j) - (\omega_i- \dd f_i) = (\omega_j-\omega_i)-\dd f_{ij}=0
\,,
\end{equation*} 
   so that the~$\omega_i - \dd f_i $  glue to the required meromorphic form $\omega$, which is clearly 
  of the second kind.
  In the reverse
  direction, starting with~$ \omega $, a form of the second kind, we may pick
  any rational $f_i$ such that $\omega_i=\omega+\dd f_i$ is
  holomorphic on $U_i$ (their existence is again implied by
  Proposition~\ref{prop:ctof}). Then letting $f_{ij}=f_j - f_i$ we see that
  $\dd f_{ij}=\omega_i-\omega_i$ so $f_{ij}$ is holomorphic giving the
  required cocycle $(\omega_i,f_{ij})$.
\end{proposition}

As a useful corollary we have the following.

\begin{corollary}\label{rest}
    Let $\omega$ be a form of the second kind on $X$ and let $U\subset X$
    be an open affine such that $\omega$ is holomorphic on $U$. Then the
    restriction of $[\omega]$ to $U$ is represented by $\omega|_U$
    in~$ \hdr^1(U/K) $.
\end{corollary}

\begin{remark}
  It follows from Proposition~\ref{prop:ctof} that the choice of $f_i$ in both
  directions of Proposition~\ref{prop:ctof1} can be arbitrary, as long as the conditions
  are satisfied. This can be seen directly. Going from cocycles to forms,
  other choices of the~$f_i$ differ by the addition of
  a rational function $f$, and $\omega$ changes with~$ \dd f$, which
  does not change the cohomology class. In the other direction, other choices
  of the~$f_i$ such that $\omega+ \dd f_i$ is holomorphic on $U_i$ add
  to each~$f_i$ a holomorphic function on~$U_i$, resulting in
  cohomologous \v Cech cocycles.
\end{remark}
\begin{prop}\label{seckindincech}
The trace of a class in $\hdr^2(X /K)$ represented by a \v Cech cocycle $(\omega_{ij},f_{ijk})$
can be computed as follows. There exists a choice of rational 1-forms $\theta_i$ and rational
functions $f_{ij}$ such that
\begin{equation}\label{tracecond}
  \omega_{ij}=\theta_j-\theta_i - \dd f_{ij}\;.
\end{equation}
For any such choice,
\begin{equation}\label{traceformula}
  \tr (\omega_{ij},f_{ijk})= -\sum_x \res_x(\theta_i)
\;,
\end{equation}
where, for each $x$, one takes
  any $i$ such that $x$ is in~$U_i$.
\end{prop}
\begin{proof}
We first show the existence of such a choice for which the formula holds. 
We look at the groups in degree~2 in~\eqref{difhtwo}. The \v Cech cocycle $(\omega_{ij},f_{ijk})$
gives the same cohomology class in the middle row of~\eqref{difhtwo} as an element comprised of germs $\eta_x$
of meromorphic 1-forms. This means that
\begin{equation} \label{desired-shape}
  (- \eta_{i,x} ,0,\omega_{ij},f_{ijk})
\end{equation}
is the differential of an element $   ( g_{i,x} , \theta_i,f_{ij})$ as in~\eqref{second-row-diff-1}. In particular, we see that $\omega_{ij}=\theta_j-\theta_i - \dd f_{ij}$ and,
for each~$x\in U_i$, that~$\eta_{i,x} = \theta_{i,x}$
up to a differential of a function in~$ \O_{X,x} $.
Hence $\res_x (\eta_x)
  = \res_x (\theta_i)$, and~\eqref{eq:trdef} shows that~\eqref{traceformula} holds for this
  choice of $\theta_i$ and $f_{ij}$.

It remains to show that the trace is the same for any choice of rational forms $\theta_i $ and rational functions $f_{ij} $ satisfying~\eqref{tracecond}. By remark~\ref{ratflsq} we can find rational functions~$f_i$ such that $f_{ij}=f_j - f_i$. Taking  $\theta_i^\prime = \theta_i - \dd f_i$, we obtain a new choice satisfying~\eqref{tracecond}, with $f_{ij}^\prime = 0 $. This has the same trace since exact forms do not have residues. It is therefore sufficient to consider choices satisfying~\eqref{tracecond} with $f_{ij}= 0 $. But such choices clearly differ by the addition of a global meromorphic form and the trace is unchanged
  by the residue theorem.
\end{proof}

\begin{remark}\label{modrem}
    It follows from the above that the trace of a cocycle can be read off from its~$\omega_{ij}$-part.
    The trace is unchanged if we add $\dd h_{ij} +\alpha_j - \alpha_i $ to
    $\omega_{ij}$, with $h_{ij}$ meromorphic and $\alpha_i$ holomorphic on
    $U_i$.
\end{remark}

\begin{proof}[Proof of Theorem~\ref{seckindcup}]
Using Proposition~\ref{prop:ctof} we get, from the forms $\omega$ and $\eta$, rational functions
$f_i$ and $g_i$ such that $\omega_i=\omega+ \dd f_i$ and $\eta_i=\eta+\dd g_i$ are holomorphic on~$U_i$,
and the classes $[\omega]$ and $[\eta]$ are represented
by cocycles  $(\omega_i, f_j-f_i)$ and $(\eta_i,g_j-g_i)$.
We now use the explicit formula for the cup product, as explained
in~\href{https://stacks.math.columbia.edu/tag/01FP}{Tag 01FP}),
based on the map of complexes in~\href{https://stacks.math.columbia.edu/tag/07MB}{Tag 07MB}).
The
$C^1(\Omega_X^1)$-component of the cup product is
$$
-\omega_i (g_j-g_i)+(f_j-f_i)\eta_j.
$$
Changing the~$ g_i $ if necessary,
we may assume that $\eta_i$ vanishes to such a large order at a finite
number of prescribed points that the following two
conditions are satisfied.
\begin{enumerate}
\item at every $x$ in $ U_i$ we have
\begin{equation}\label{setcond}
  \res_x (-g_i \omega)  =  \res_x (\omega \tint \eta)  
\end{equation}
\item the forms $ f_i \eta_i$ are holomorphic on $U_i$.
\end{enumerate}
This is because, by Riemann-Roch, one has a
function with prescribed Laurent expansion modulo a given power of the
uniformiser at each of a finite number of points, provided this function is allowed
to have poles of high enough order at prescribed other points.
  To compute $\tr ([\omega]\cup [\eta])$ we are allowed, by Remark~\ref{modrem}
to modify the 1-forms on the~$U_{ij}$ by
either $\alpha_j-\alpha_i$, with ~$\alpha_i$ a holomorphic form on~$U_i$, or
by $\dd h_{ij}$, with $h_{ij}$ a meromorphic function on~$ U_{ij} $.
Taking~$\alpha_i = f_i \eta_i $ we may replace $(f_j-f_i)\eta_j$ by
$f_i(\eta_i-\eta_j)= f_i (\dd g_i-\dd g_j)$. Subtracting $\dd (f_i(g_i-g_j))$ we may
replace this with $(g_j-g_i) \dd f_i$, so our 1-form on~$ U_{ij} $ is transformed into 
\begin{equation*}
  -\omega_i (g_j-g_i)+ (g_j-g_i) \dd f_i = -\omega(g_j-g_i) = -\omega g_j +
  \omega g_i
,  
\end{equation*}
We are therefore in the situation of Proposition~\ref{seckindincech} with $\theta_i =  -\omega g_i $. From~\eqref{traceformula} we get
\begin{equation*}
  \tr ([\omega]\cup [\eta]) = - \sum_{x}\res_x (-g_i \omega ) = -\sum_x \res_x
    \left(\omega \tint \eta  \right)
\end{equation*}
by~\eqref{setcond}. This is equal to the expression in the theorem by the well-known formula
\begin{equation*}
  \res_x (\omega \tint \eta) +
  \res_x (\eta \tint \omega) = \res_x \left( \dd ((\tint \omega )(\tint \eta)) \right) = 0
.
\end{equation*}
\end{proof}

\bibliographystyle{plain}
\bibliography{References}

\end{document}